\documentclass[10pt, reqno]{amsart}
\usepackage[english]{babel}
\usepackage[colorlinks,citecolor=green,linkcolor=red]{hyperref}
\usepackage{amsthm}
\usepackage{color}
\usepackage{mathrsfs}
\usepackage{amsmath}
\usepackage{mathtools}
\usepackage{amsfonts}
\usepackage{amssymb}
\usepackage{bm}
\usepackage{physics}
\usepackage{enumitem}
\usepackage{tikz-cd}
\usepackage{a4wide}
\usepackage{cleveref}
\usepackage{esint}
\usepackage{nicefrac}
\usepackage{dirtytalk}

\numberwithin{equation}{section}

\newtheorem{thm}{Theorem}[section]
\newtheorem{thma}{Theorem}

\newtheorem*{thm*}{Theorem}
\newtheorem{lem}[thm]{Lemma}
\newtheorem{prop}[thm]{Proposition}
\newtheorem{cor}[thm]{Corollary}
\newtheorem{cora}{Corollary}

\newtheorem{defn}[thm]{Definition}

\theoremstyle{remark}
\newtheorem{rem}[thm]{Remark}
\newtheorem{ex}[thm]{Example}

\DeclareMathOperator*{\esssup}{ess\,sup}

\DeclareMathOperator*{\aplimsup}{ap\,\limsup}
\DeclareMathOperator*{\apliminf}{ap\,\liminf}

\newcommand{\Ch}{{\mathrm {Ch}}}
\newcommand{\lip}{{\mathrm {lip}}}
\newcommand{\Lip}{{\mathrm {Lip}}}
\newcommand{\diff}{{\mathrm{d}}}
\newcommand{\DIFF}{{\mathrm{D}}}
\newcommand{\nablatilde}{\bar{\nabla}}
\newcommand{\Prbar}{\bar{\Pr}}
\newcommand{\pibar}{\bar{\pi}}

\newcommand{\qcr}{{\mathrm{QCR}}}
\renewcommand{\tr}{{\mathrm{tr}}}
\newcommand{\trbar}{{{\bar{\tr}}}}
\newcommand{\qcrbar}{{\mathrm{\bar{QCR}}}}
\newcommand{\dive}{{\mathrm{div}}}

\newcommand{\supp}{{\mathrm {supp\,}}}
\newcommand{\diam}{{\mathrm {diam\,}}}
\newcommand{\capa}{{\mathrm {Cap}}}
\newcommand{\per}{{\mathrm {Per}}}

\newcommand{\mres}{\mathbin{\vrule height 1.6ex depth 0pt width 0.13ex\vrule height 0.13ex depth 0pt width 1.3ex}}
\newcommand{\mressmall}{\mathbin{\vrule height 1.2ex depth 0pt width 0.11ex\vrule height 0.10ex depth 0pt width 1.0ex}}

\newcommand{\dist}{{\mathsf{d}}}
\newcommand{\mass}{{\mathsf{m}}}

\newcommand{\XX}{{\mathsf{X}}}
\newcommand{\YY}{{\mathsf{Y}}}

\newcommand{\FF}{{\mathcal{F}}}

\newcommand{\VV}{{\mathcal{V}}}
\newcommand{\DD}{{\mathcal{D}}}
\newcommand{\Rr}{{\mathcal{R}}}
\newcommand{\WW}{{\mathcal{W}}}
\newcommand{\mvect}{{\mathsf{M}}}
\newcommand{\nvect}{{\mathsf{N}}}

\newcommand{\defeq}{\mathrel{\mathop:}=}

\newcommand{\AK}{\mathrm{AK}}
\newcommand{\MM}{\mathcal{M}}

\newcommand{\LL}{\mathcal{L}}
\newcommand{\RR}{\mathbb{R}}
\newcommand{\NN}{\mathbb{N}}
\newcommand{\Cqcvf}{\mathcal{QC}(T\XX)}
\newcommand{\Cqcvfinf}{\mathcal{QC}^\infty(T\XX)}

\newcommand{\Ss}{{\mathrm {S}^2}(\XX)}
\newcommand{\Lp}{{\mathrm {L}}}
\newcommand{\Lploc}{\mathrm {L}_\mathrm{loc}}
\newcommand{\Lpo}{\Lp^0(\mass)}
\newcommand{\Lpc}{\Lp^0(\capa)}
\newcommand{\Lpcinf}{\Lp^\infty(\capa)}
\newcommand{\Lpu}{\Lp^1(\mass)}

\newcommand{\Lpuloc}{\Lp^1_{\mathrm{loc}}(\mass)}
\newcommand{\Lpp}{\Lp^p(\XX)}

\newcommand{\Lpt}{\Lp^2(\mass)}

\newcommand{\Lpi}{\Lp^\infty(\mass)}

\newcommand{\HSs}{{\mathrm {H^{1,2}(\XX)}}}

\newcommand{\WHCSs}{{\mathrm {H^{1,2}_C}(T\XX)}}

\newcommand{\WHHSs}{{\mathrm {H^{1,2}_H}(T\XX)}}

\newcommand{\HSsloc}{{\mathrm {H^{1,2}_{loc}(\XX)}}}

\newcommand{\LIP}{{\mathrm {LIP}}}
\newcommand{\BV}{{\mathrm {BV}}}
\newcommand{\BVv}{{\mathrm {BV}}(\XX)}
\newcommand{\LIPbs}{{\mathrm {LIP_{bs}}}}

\newcommand{\Cc}{{C_{\mathrm{c}}}}

\newcommand{\Cb}{{C_{\mathrm{b}}}}

\newcommand{\LIPloc}{{\mathrm {LIP_{loc}}}}
\newcommand{\LIPb}{{\mathrm {LIP_{b}}}}

\newcommand{\TestF}{{\mathrm {Test}}}

\newcommand{\TestV}{{\mathrm {TestV}}}

\newcommand{\TestVbar}{{\mathrm {Test\bar{V}}}}

\newcommand{\cotX}{\Lp^2(T^*\XX)}

\newcommand{\tanX}{\Lp^2(T\XX)}
\newcommand{\tanXp}{\Lp^p(T\XX)}

\newcommand{\tanXzero}{\Lp^0(T\XX)}
\newcommand{\cotanXzero}{\Lp^0(T^*\XX)}
\newcommand{\tanXinf}{\Lp^\infty(T\XX)}
\newcommand{\tanXinfn}{\Lp^\infty(T^n\XX)}
\newcommand{\cotanXp}{\Lp^p(T^*\XX)}

\newcommand{\tanbvXp}[2]{\Lp^{#1}_{#2}(T\XX)}
\newcommand{\tanbvXpn}[2]{\Lp^{#1}_{#2}(T^n\XX)}
\newcommand{\tanbvXn}[1]{\Lp^2_{#1}(T^n\XX)}

\newcommand{\tanbvXzero}[1]{\Lp^0_{#1}(T\XX)}

\newcommand{\tanXcap}{\Lp^0_\capa(T\XX)}
\newcommand{\tanXcapn}{\Lp^0_\capa(T^n\XX)}

\newcommand{\tanXcapinfn}{\Lp^\infty_\capa(T^n\XX)}
\newcommand{\tanXcapinfm}{\Lp^\infty_\capa(T^m\XX)}
\newcommand{\tanXcapinfk}{\Lp^\infty_\capa(T^k\XX)}
\newcommand{\RCD}{{\mathrm {RCD}}}

\newcommand{\PI}{{\mathrm {PI}}}

\newcommand{\fr}{\penalty-20\null\hfill$\blacksquare$}     

\newcommand{\leb}{\mathrm{Leb}}
\newcommand{\borrep}{\mathrm{BorRep}}

\let\epsilon\varepsilon

\newcommand{\Id}{{\mathrm {Id}}}

\def\Xint#1{\mathchoice
	{\XXint\displaystyle\textstyle{#1}}%
	{\XXint\textstyle\scriptstyle{#1}}%
	{\XXint\scriptstyle\scriptscriptstyle{#1}}%
	{\XXint\scriptscriptstyle\scriptscriptstyle{#1}}%
	\!\int}
\def\XXint#1#2#3{{\setbox0=\hbox{$#1{#2#3}{\int}$}
		\vcenter{\hbox{$#2#3$}}\kern-.5\wd0}}

\def\dashint{\Xint-}

\title{}
\begin{document}
	\title[Local vector measures]{Local vector measures}

\author[C. Brena]{Camillo Brena}
\address{C.~Brena: Scuola Normale Superiore, Piazza dei Cavalieri 7, 56126 Pisa} 
\email{\tt camillo.brena@sns.it}

\author[N. Gigli]{Nicola Gigli}
\address{N.~Gigli: SISSA, Via Bonomea 265, 34136 Trieste} 
\email{\tt ngigli@sissa.it}

	\begin{abstract} Consider a BV function on a Riemannian manifold. What is its differential? And what about the Hessian of a convex function? These questions have clear answers in terms of (co)vector/matrix valued measures if the manifold is the Euclidean space. In more general curved contexts, the same objects can be perfectly understood via charts. However,  charts are often unavailable in the less regular setting of metric geometry, where still the questions make sense.
	
	In this paper we propose a way to deal with this sort of problems and, more generally, to give a meaning to a concept of `measure acting in duality with sections of a given bundle', loosely speaking. Despite the generality, several classical results in measure theory like Riesz's and Alexandrov's theorems have a natural counterpart in this setting. Moreover, as we are going to discuss, the notions introduced here provide a unified framework for several key concepts in nonsmooth analysis that have been introduced more than two decades ago, such as: Ambrosio-Kirchheim's metric currents, Cheeger's Sobolev functions and Miranda's BV functions.
	
	Not surprisingly, the understanding of the structure of these objects improves with the regularity of the underlying space. We are particularly interested in the case of $\RCD$ spaces where, as we will argue, the regularity of several key measures of the type we study nicely matches the known regularity theory for vector fields, resulting in a very effective theory. 
	
	We expect that the notions developed here will help creating stronger links between differential calculus in Alexandrov spaces (based on Perelman's DC charts) and in $\RCD$ ones (based on  intrinsic tensor calculus).

	\end{abstract}
\maketitle
\setcounter{tocdepth}{3}

\tableofcontents
\section*{Introduction}

This paper is about further developing differential calculus in the nonsmooth setting of metric (measure) spaces. The starting point is the paper \cite{Gigli14}  of the second author, where the concept of $\Lp^p(\mass)$-normed $\Lp^\infty(\mass)$-module has been introduced as a means to interpret what $\mass$-a.e.\ defined tensor fields should be on a given metric measure space $(\XX,\dist,\mass)$. A typical object that is well defined in such framework is that of differential $\dd f$ of a Sobolev function $f\in {\rm W}^{1,p}(\XX)$: shortly said, this is possible because such differential is, even in the smooth world, an a.e.\ defined 1-form.

Still, when taking distributional derivatives it might very well be that one ends up with objects more singular  than these. Typical instances where this occurs are:
\begin{itemize}
\item[A)] In dealing with  the differential of a BV function. 
\item[B)] In  dealing with  the Hessian of a convex function. 
\item[C)] In dealing with the Ricci curvature tensor on $\RCD$ spaces: this  is a sort of measure defined in duality with smooth vector fields whose properties are not yet well understood (see \cite[Section 3.6]{Gigli14}).
\end{itemize}
In all these cases, the relevant concept should be something possibly concentrated on a negligible set: as such, the notion of $\Lp^p(\mass)$-normed $\Lp^\infty(\mass)$-modules does not really suffice. It is our goal here to propose a more general theory capable of dealing, in particular, with the cases above. As it will turn out, the framework we develop is fully compatible even with objects conceptually far from the examples mentioned, like metric currents as developed by Ambrosio-Kirchheim in \cite{AmbrosioKirchheim00}.

The theory, as well as the manuscript, is divided in two parts: in the first we develop the concept of `local vector measure' in the general setting of Polish spaces, while in the second we discuss how the general framework fits in the more regular environment of $\RCD$ spaces.

\subsection*{Polish theory} Measures are defined as functionals over other objects: either sets or functions or both. As such, if we want to give an abstract notion of `measure' capable of giving a meaning to the examples above, we should at the same time be ready to specify the space of objects where it is acting. Example $\rm A)$ above is particularly illuminating: consider a BV function on a Riemannian manifold. Then its differential, whatever it is, should be something acting in duality with continuous vector fields; as such one should have at disposal such vector fields before defining the differential.

With this in mind, we build our theory by modelling it on the duality
\[
\{\text{functions in }\Cb(\XX)\}\qquad\qquad\text{  in duality with }\qquad\qquad\{\text{finite Radon measures on $\XX$}\}
\]
(for comparison, notice that the duality theory in \cite{Gigli14} was modelled upon the duality between $\Lp^p$-functions and $\Lp^q$-functions, $\frac1p+\frac1q=1$). 

We thus begin our presentation introducing the concept of `normed $\Cb(\XX)$-module' that aims at being an abstract version of the space of continuous (or, more generally,  bounded) sections of a normed bundle. By definition, a normed $\Cb(\XX)$-module is a normed space $(\VV,\|\cdot\|)$ that is also a module over $\Cb(\XX)$ in the algebraic sense and for which the inequality
\begin{equation}
\label{eq:compintro}
\|f_1v_1+\cdots+f_nv_n\|\leq\max_i\|f_i\|_{\infty}\|v_i\|
\end{equation}
holds for any $n\in\NN$ and choice of  $v_i\in\VV$ and $f_i\in\Cb(\XX)$ with pairwise disjoint supports. The structure of $\Cb(\XX)$-module gives the possibility of inspecting the `local behaviour' of elements in $\VV$. For instance, for $v\in\VV$ and $A\subseteq\XX$ open we can define the seminorm $\|v\|_{|A}$ of $v$ in $A$ as $\sup\|fv\|$, the sup being taken among all $f\in\Cb(\XX)$ with support in $A$ and norm bounded by 1. Then for $C\subseteq\XX$ closed we say that the support of $v$ is contained in $C$ provided $\|v\|_{|\XX\setminus C}=0$. With these definitions it is easy to see that the compatibility condition \eqref{eq:compintro} can be read as a way of saying that the norm of $\VV$ is a sort of `sup' norm, as it implies that
\[
\|v_1+\cdots+v_n\|=\max_i\|v_i\|\qquad\text{ provided the $v_i$'s have disjoint support}
\]
(here having disjoint supports means that there are disjoint closed sets $C_i$ with $\supp v_i\subseteq C_i$).

For example,  even though   for every metric measure space $(\XX,\dist,\mass)$ and $p\in[1,\infty]$ the space $\Lp^p(\XX)$ is a module over $\Cb(\XX)$ in the algebraic sense,  condition \eqref{eq:compintro} only holds in the case $p=\infty$.

Now recall that a vector  measure on the topological dual $\VV'$ of $\VV$ as Banach space is a $\sigma$-additive function taking Borel subsets of $\XX$ and returning functionals in $\VV'$. Then a \emph{local} vector  measure defined on $\VV$ (i.e.\ acting in duality with elements of $\VV$) is a vector  measure $\nvect$ on  $\VV'$ that is compatible with the module  structure in the sense that
\begin{equation}
\label{eq:wlintro}
\nvect(A)(v)=0\qquad\forall A\subseteq\XX \text{ open and $v\in\VV$ such that $\|v\|_{|A}=0$}.
\end{equation}
We refer to property \eqref{eq:wlintro} as `locality' (or `weak locality', to distinguish it from the stronger notion discussed below). The basic example here is: $\VV:=\Cb(\XX)$, $\mu$ finite (possibly signed) Radon measure on $\XX$ and $\nvect$  given by the formula
\[
\nvect(B)(f):=\int_B f\,\dd\mu\qquad\text{$\forall B\subseteq\XX$ Borel and $f\in\Cb(\XX)$}.
\]
Section \ref{sectinit} is devoted to the study of this sort of measures. Among others, a relevant result that we obtain is a rather abstract version of Riesz's representation theorem, see Theorem \ref{weakder}. In the case of $\XX$ compact can be stated as: for any $F\in \VV'$ there is a unique local vector measure $\nvect$ on $\VV$ such that
\begin{equation}
\label{eq:rieszintro}
\nvect(\XX)(v)=F(v)\qquad\forall v\in\VV.
\end{equation}
Notice that in the case $\VV=C(\XX)$ this easily reduces to the standard Riesz's theorem. In the non-compact case  we prove that for $F\in\VV'$  there exists $\nvect$ as in \eqref{eq:rieszintro} if and only if
\begin{equation}
\label{eq:tightintro}
F(f_nv)\to0\qquad\forall v\in\VV\text{ and $\{f_n\}_n\subseteq\Cb(\XX)$ with $f_n(x)\searrow 0$ for any $x\in\XX$.}
\end{equation}
This  is reminiscent of the analogue condition appearing in the study of Daniell's integral.

Another relevant property of local vector measures concerns their total variation, that for a  generic vector  measure $\nvect$ on $\VV'$ is defined as $|\nvect|(E):=\sup\sum_i\|\nvect(E_i)\|'$,  the sup being taken among  at most countable partitions of the Borel set $E$. As it turns out, for local vector measures the total variation is always finite and the identity
\[
|\nvect|(E)=\|\nvect(E)\|'\qquad\text{$\forall E\subseteq \XX$ Borel}
\]
holds. In particular,  $|\nvect|$ is always a Radon measure and it is therefore natural to ask whether any sort of polar decomposition is in place. As it turns out, this is always the case, meaning that for any local vector measure $\nvect$ there is a unique $L:\XX\to\VV'$ such that
\begin{equation}
\notag
\nvect=L|\nvect|\quad\text{ in the sense that }\quad\nvect(B)(v)=\int_BL(x)(v)\,\dd|\nvect|(x)\qquad\text{$\forall B\subseteq\XX$ Borel and $v\in\VV$,}
\end{equation}
and it holds $\|L(x)\|'=1$ for $|\nvect|$-a.e.\ $x$ (see Proposition \ref{allhaspolarabs} for the precise formulation). The locality of $\nvect$ implies the following locality property of $L$:
\begin{equation}
\label{eq:locpol}
L(x)(v)=0,\quad|\nvect|\text{-a.e.\ on}\ A\qquad\text{$\forall A\subseteq\XX$ open such that $\|v\|_{|A}=0$},
\end{equation}
that we can also interpret by saying that $L(x)(v)$ depends only on the germ of $v$ at $x$.

At this level of generality we cannot say much more than this, but still these concepts turn out to be flexible enough to be  naturally compatible with various pre-existing notions in metric geometry. Given the conceptual proximity of the definition of $L^\infty$-module and $\Cb$-module, it is not surprising that the notion of differential of a Sobolev function as given in \cite{Gigli14} can be reinterpreted in this framework (see Section \ref{se:diffsob}). On the other hand, one thing that we gain from the studies conducted here, and that in fact motivates them, is the possibility  to give a meaning to the differential $\DIFF f$ of a BV function $f$ on arbitrary metric measure spaces. This notion of differential comes with some non-trivial calculus rule, for instance the Leibniz formula
\[
\DIFF (fg)=g\DIFF f+f\DIFF g
\]
holds for any couple $f,g$ of  bounded\emph{ continuous} functions of bounded variation (see Section \ref{sectBV}). 

As mentioned above, also the concept of metric current naturally fits in this framework, let us briefly mention how. By the definition and the results in \cite{AmbrosioKirchheim00}, an $n$ current acts on the space
\[
\DD^n(\XX)\defeq \Cb(\XX)\otimes \bigwedge^n \LIP(\XX)
\]
whose elements are  sums of objects formally written as $f\dd\varphi_1\wedge\cdots\wedge\dd\varphi_n$. Thus, by also looking at the Euclidean case, it is natural to equip $\DD^n(\XX)$ with the seminorm 
\begin{equation}\notag
\Vert v\Vert\defeq\sup_T T(v)\quad\text{for every }v\in\DD^n(\XX)
\end{equation}
where the supremum is taken among all currents $T$ with $\Vert T\Vert_{\AK}(\XX)\le 1$ (here $\|T\|_{\AK}$ is the mass of the current as defined in \cite{AmbrosioKirchheim00}). Then, tautologically, an $n$  current $T$ can be seen as an element of $\DD^n(\XX)'$ and the theories in \cite{AmbrosioKirchheim00} and in here quite nicely match, meaning that:
\begin{itemize}
\item[-] The results in \cite{AmbrosioKirchheim00}  show that $\DD^n(\XX)'$ can naturally be equipped with the structure of a normed $\Cb(\XX)$-module.
\item[-] An $n$ current $T$ always satisfies \eqref{eq:tightintro} when seen as element of $\DD^n(\XX)'$ and thus it induces a unique local vector measure $\nvect_T$ on $\DD^n(\XX)$. Also, the total variation $|\nvect_T|$ of $\nvect_T$ as defined above coincides with the mass $\|T\|_{\AK}$ of $T$ as introduced in \cite{AmbrosioKirchheim00}.
\item[-] The map sending $(f\dd\varphi_1\wedge \cdots\dd\varphi_n,g\dd\psi_1\wedge \cdots\dd\psi_m)$ to $fg\dd\varphi_1\wedge \cdots\dd\varphi_n\wedge\dd \psi_1\wedge \cdots\dd\psi_m$ induces a unique bilinear and continuous map from $\DD^n(\XX)\times \DD^m(\XX)$ to $\DD^{n+m}(\XX)$ of norm $\leq 1$. 
\item[-] In particular, the collection of normed $\Cb(\XX)$-modules  $\{\DD^n(\XX)\}_{n\in\NN}$ possesses a natural algebra structure and - reading in this language some of the results in \cite{AmbrosioKirchheim00}  - we see that the `differentiation' map $\LIP(\XX)\ni f\mapsto 1\dd f\in \DD^1(\XX)$ satisfies  Leibniz and chain rules and it is closed in a natural sense. 
\end{itemize}
See Section \ref{se:curr}.

\medskip

The concept of locality as expressed in \eqref{eq:wlintro} and in \eqref{eq:locpol} is the most we can expect for local vector measures defined on arbitrary $\Cb(\XX)$-modules. It is so because, in some sense, for an element of such a module we are capable of saying whether it is 0 on an open set, but we cannot give a reasonable meaning to it being 0 on a Borel set. Still, in many practical situations a relevant normed $\Cb(\XX)$-module $\VV$ is given as suitable space of bounded elements of a larger $\Lp^p(\mass)$-normed $\Lp^\infty(\mass)$-module $\mathscr M$ equipped with the norm
\[
\|v\|:=\||v|\|_{\Lp^\infty(\mass)}\qquad\forall v\in\VV.
\]
If this is the case it is important to understand how the structure of $\Lp^\infty(\mass)$-module and that of local vector measures on $\VV$ interact. We are going to see in Section \ref{normedmodules} that a local vector measure $\nvect$ on such a module $\VV$ with $|\nvect|\ll\mass$ satisfies the stronger locality property
\[
\nvect(B)(v)=0\qquad\forall B\subseteq\XX \text{ Borel and $v\in\VV$ such that $|v|=0$ $\mass$-a.e.\ on $B$}
\]
if and only if there is an element $M_\nvect\in \mathscr M^*$ (the dual in the sense of modules) with $|M_\nvect|=1$ $|\nvect|$-a.e.\ such that 
\[
L_\nvect(v)=M_\nvect(v)\quad|\nvect|\text{-a.e.}\qquad\forall v\in\VV.
\]
We shall call measures of this kind \emph{strongly local measures}. Notice that if $\mathscr M$ is an Hilbert module, then the writing above (and  Riesz's theorem for Hilbert modules) means that we can represent a strongly local measure $\nvect$ as ${\sf v}|\nvect|$ for some ${\sf v}\in\mathscr M$ with $|{\sf v}|=1$ $|\nvect|$-a.e., i.e.\ that
\[
\nvect(B)(v)=\int_Bv\cdot{\sf v}\,\dd|\nvect|.
\]
This closely corresponds to the usual intuition that wants a `vector valued measure' to be writable via polar decomposition as its mass times a vector field of norm 1.

\subsection*{$\RCD$ theory} In this part of the paper we focus the attention on $\RCD$ spaces, whose study actually motived this manuscript: here the starting  observation is about a `coincidence' occurring when handling certain non-smooth objects. It is indeed the case that:
\begin{itemize}
\item[-] In all of the three examples mentioned at the beginning of the introduction, the measure does not see sets of null 2-capacity: this is easily seen for Examples A and B at least if  the underlying space is the Euclidean one, for the Ricci curvature tensor on $\RCD$ spaces we refer to \cite{Gigli14}. To be more concrete: the polar decomposition of the differential $\DIFF f$ of a real valued BV function on the Euclidean space takes the form ${\sf v}|\DIFF f|$ for some Borel vector field ${\sf v}$ of norm 1 $|\DIFF f|$-a.e.\ and some non-negative Radon measure $|\DIFF f|$ giving 0 mass to sets of null 2-capacity. Notice that in fact the total variation measure of a function of bounded variation is well defined on arbitrary metric measure spaces \cite{MIRANDA2003} and does not see  sets of null 2-capacity  \cite{BGBV}. 
\item[-] On $\RCD$ spaces vector fields can be defined up to sets of null 2-capacity (unlike general metric measure spaces, where they typically are only defined up to sets of measure zero - at least in the axiomatization proposed in \cite{Gigli14}). More precisely, the analysis carried out in \cite{debin2019quasicontinuous} shows that Sobolev vector fields have a unique `quasi-continuous representative' defined up to sets of null 2-capacity: the analogy here is with the well known case of Sobolev functions, but in this case all the concepts are built upon functional analytic foundation, as no topology on the tangent module is given, so one cannot speak of continuity of vector fields in the classical sense.
\end{itemize}
The coincidence is in the fact that the best known regularity for vector fields matches the known regularity for the relevant (mass of the) measure, suggesting the existence of a good duality theory. Of course, we don't think this is a  coincidence at all, but rather an instance of the not infrequent phenomenon that sees the presence of solid analytic foundations for geometrically relevant objects.

This second part of the work is devoted to exploring  the generalities of the theory in this setting and to tailoring it to the study of the differential of BV functions. In particular we shall mostly rely on the concept of strongly local measure discussed in Section \ref{normedmodules}. We cover also the case of vector valued BV functions, that presents additional difficulties and seems intractable over general metric measure spaces.  The outcome of the analysis is that such distributional differential $\DIFF f$ can always be represented as  ${\sf v}|\DIFF f|$, where ${\sf v}$ is a suitable quasi-continuous vector field of norm 1 $|\DIFF f|$-a.e.\  and  $|\DIFF f|$ is the classical total variation measure of $f$ as introduced in \cite{MIRANDA2003}, at least in the scalar case. This description, finer than the one available on arbitrary spaces, is consistent with all the recent fine calculus tools and integration by parts developed on finite dimensional $\RCD$ spaces. For instance, the Gauss-Green formula established in \cite{bru2019rectifiability} can now be interpreted as the representation, in the sense above, of the distributional differential of the characteristic function of a set of finite perimeter. In this sense also the Leibniz rule for the product of two bounded (but not necessarily continuous) BV functions recently obtained in \cite{BGBV} naturally fits in this framework, meaning that we have
\[
\DIFF (fg)=\bar f\DIFF g+\bar g\DIFF f,\qquad\forall f\in {\mathrm{BV}}(\XX)\cap \Lp^\infty(\XX),
\]
where $\bar f:=\frac{f^\vee+f^\wedge}2$  is  the precise representative of $f$  (here $f^\vee(x)\defeq\aplimsup_{y\rightarrow x} f(y)$ and $f^\wedge(x)\defeq\apliminf_{y\rightarrow x} f(y)$, see \eqref{veeandwedge}), similarly for $g$ and the differentials are defined as strongly local vector measures in the sense discussed here. We recall this fact in Proposition \ref{leibnizrcd}.

\bigskip

In this manuscript we are not going to study  the Hessian of convex functions and the Ricci curvature tensor: this will be the main goal of an upcoming paper.

\subsection*{Acknowledgments}
The first author acknowledges the support of the PRIN 2017 project `Gradient flows, Optimal Transport and Metric Measure Structures'.

\section{The theory for Polish spaces}
\subsection{Definitions and results}\label{sectinit}
In the first part of this note, our framework is a Polish space $(\XX,\tau)$, which means that there exists a distance $\dist:\XX\times\XX\rightarrow\RR$ inducing the topology $\tau$ and such that $(\XX,\dist)$ is a complete and separable metric space. For simplicity of exposition, we fix such distance $\dist$ and hence we will often consider the Polish space $(\XX,\tau)$  as the complete and separable metric space $(\XX,\dist)$. We point out that for what concerns this part of the work, the choice of the distance is immaterial and what is important is the topology $\tau$. However, having a fixed distance  is natural if one reads the first part of the work in perspective of the second part, where we will work with metric measure spaces, which are triplets $(\XX,\dist,\mass)$ where $(\XX,\dist)$ is a complete separable distance on $\XX$ and $\mass$ is a Borel measure, finite on balls.
 
We denote the Borel $\sigma$-algebra of $\XX$ by $\mathcal{B}(\XX)$ and we adopt the standard notation for the various function spaces. 

\bigskip

In the following definition, we introduce the concept of normed $\Cb(\XX)$-module. Recall that from the algebraic perspective there is a well defined concept of module over the commutative ring with unity  $\Cb(\XX)$, it being a commutative group with unity $\VV$ equipped with an operation $\cdot:\Cb(\XX)\times\VV\to\VV$ satisfying
\[
\begin{split}
f\cdot(v+w)&=f\cdot v+f\cdot w,\\
(f+g)\cdot v&=f\cdot v+g\cdot v,\\
(fg)\cdot v&=f\cdot(g\cdot v),\\
{\bf 1}\cdot v&=v,
\end{split}
\]
for any $f,g\in \Cb(\XX)$ and $v,w\in\VV$. Here ${\bf 1}$ is the multiplicative identity of $\Cb(\XX)$, i.e.\ the function identically equal to 1. In what follow we shall omit the `dot' and write the product of $f$ and $v$ simply as $fv$.

The space $\Cb(\XX)$ is also a Banach space when equipped with its natural `$\sup$' norm $\Vert \,\cdot\,\Vert_\infty$. We shall then consider normed spaces that are also modules over $\Cb(\XX)$ in the algebraic sense and for which a certain compatibility between the normed and algebraic structure is present:
\begin{defn}[Normed $\Cb(\XX)$-modules]\label{defnmodules}
 Let $(\VV,\Vert\,\cdot\,\Vert)$ be a normed space that is also a module over the ring $\Cb(\XX)$ in the algebraic sense. We say that $(\VV,\Vert\,\cdot\,\Vert)$ is a  normed $\Cb(\XX)$-module if
\begin{equation}\label{compeq}
		\Vert  f_1 v_1+\dots+f_n v_n\Vert\le \max_{i} \Vert f_i\Vert_{\infty} \max_{i}\Vert v_i\Vert,
\end{equation}
whenever $\{f_i\}_{i=1,\dots,n}\subseteq\Cb(\XX)$ have pairwise disjoint supports and $\{v_i\}_{i=1,\dots,n}\subseteq\VV$. 
\end{defn}
Even though the expression `normed $\Cb(\XX)$-module' has an algebraic meaning (i.e.\ a normed vector space that is also a module over $\Cb(\XX)$), in what follows  when we write `normed $\Cb(\XX)$-module', \textbf{we always refer to the notion introduced in the definition above.}

Simple examples of normed $\Cb(\XX)$-modules are:
\begin{itemize}
\item[-] The space $(\Cb(\XX),\Vert \,\cdot\,\Vert_\infty)$ itself.
\item[-] Let $(\XX,\dist,\mass)$ be a metric measure space and $\mathscr M$  a  $\Lpo$-normed $\Lpo$-module. Then any subspace   $\VV$ of $\mathscr M$  closed under multiplication by $\Cb(\XX)$ and made of elements with pointwise norm in $\Lpi$ is a normed  $\Cb(\XX)$-module once endowed with the norm $\Vert\abs{\,\cdot\,}\Vert_{\Lpi}$.
\end{itemize}

Here are two simple consequences of our standing assumptions for $\Cb(\XX)$-modules that we are going to  use with no reference.
\begin{rem}
	The following properties hold:
	\begin{enumerate}
		\item if $f(x)=\lambda\in\RR$ for every $x\in\XX$, then $f v=\lambda v$ for every $v\in\VV$. Indeed,  this  holds for $\lambda=1$ by the algebraic definition of module, then  extends  to $\lambda\in\mathbb{Q}$ by algebra and then to $\lambda\in\RR$, by the continuity granted by \eqref{compeq};
		\item if $\{f_i\}_{i=1,\dots,n}\subseteq\Cb(\XX)$ have pairwise disjoint support and $\{v_i\}_{i=1,\dots,n}\subseteq\VV$, then 
\[
			\Vert  f_1 v_1+\dots+f_n v_n\Vert\le \max_{i} \left(\Vert f_i\Vert_{\infty} \Vert v_i\Vert\right).
\]		To show this, just divide each non zero $f_i$ by $\Vert f_i\Vert_{\infty}$ and multiply each $v_i$ by the same quantity.\fr
	\end{enumerate}
\end{rem}
Occasionally we shall consider modules over appropriate subrings of $\Cb(\XX)$:
\begin{defn}[Normed $\mathcal R$-modules]\label{def:Rap}
Let $\mathcal R\subseteq\Cb(\XX)$ be a subring (hence, contains the unity). We say that $\mathcal R$ \textbf{approximates open sets} if $\mathcal R$ is a lattice closed under multiplication by constant functions and such that  for every open subset $A\subseteq \XX$, there exists a sequence $\{f_k\}_k\subseteq\mathcal R$ such that $f_k(x)\nearrow \chi_A(x)$ for every $x\in\XX$. 

A normed $\mathcal R$-module is a normed vector  space $(\VV,\Vert\,\cdot\,\Vert)$ that is also a module over $\mathcal R$ in the algebraic sense such that property \eqref{compeq} holds for any $\{v_i\}_{i=1,\dots,n}\subseteq\VV$ and $\{f_i\}_{i=1,\dots,n}\subseteq\mathcal R$ with pairwise disjoint supports.

When speaking about normed $\mathcal R$-modules we shall always refer to this sort of structure and always assume that $\mathcal R$ approximates open sets.
\end{defn}
\begin{ex}
$\Cb(\XX)$ itself and the subring generated by $\LIPbs(\XX)$ (i.e.\ $\dist$-Lipschitz functions with bounded support) and the constant functions approximate open sets.\fr
\end{ex}
In the following remark we state some properties of subrings that approximates open sets that we are going to exploit in the sequel.
\begin{rem}\label{rem:Rpropr}
Let $\mathcal R\subseteq\Cb(\XX)$ that approximates open sets and let $A\subseteq \XX$ open. Then:
\begin{enumerate}
	\item by the lattice property  of $\mathcal R$ we can find $\{f_k\}_k\subseteq\mathcal R$ such that $f_k(x)\nearrow \chi_A(x)$ for every $x\in\XX$ and also $f_k(x)\in [0,1]$ for every $x\in\XX$ and $k\in\NN$. Also, we can assume that $\supp f_k\subseteq A$ for every $k$, using again the properties of $\mathcal R$;
	\item if $K\subseteq A$ is compact, we can modify $\{f_k\}_k$, still remaining in $\mathcal R$, in such a way that, in addition to the properties stated in $(1)$, it holds that for every $k\in\NN$ $ f_k=1$ on a neighbourhood of $K$. This is due to Dini's monotone convergence Theorem and the lattice property of $\mathcal R$.  In particular we have
\begin{equation}
\label{eq:separK}
\begin{split}
&\text{Let $K\subseteq \XX$ compact and $A\subseteq\XX$ open with $K\subseteq A$. Then there exists a function}\\
&\text{ $\varphi\in \mathcal R$ valued in $[0,1]$ such that $\varphi=1$ on a neighbourhood of $K$ and $\supp\varphi\subseteq A$.}
\end{split}
\end{equation}		
	\item 	For every $f\in\Cb(\XX)$ there exists a sequence $\{f_n\}_n\subseteq\Rr$ with $f_n\nearrow f$. Notice that by Dini's monotone convergence Theorem, this approximation is uniform on compact sets.
	
	Indeed, take $f\in\Cb(\XX)$, say $f(x)\in[0,1]$ for every $x\in\XX$.
	Take $k\in\NN$. For every $j=0,\dots,k-1$, consider a sequence $\{\chi_{j,k}^n\}_n\subseteq\Rr$ taking values in $[0,1]$ and such that $\chi_{j,k}^n\nearrow\chi_{\{f>\frac{{j+1}}k\}}$, then let $$g_k^n\defeq\tfrac1k\sum_{j=0}^{k-1} \chi_{j,k}^n\in\Rr.$$
	Notice that $g^n_k\leq f$ for every $k,n$, then let $(n_i,k_i),_{i\in\NN}$ be an enumeration of $\NN^2$ and, finally, set $f_0\defeq 0$ and for $i\ge 1$, $f_i\defeq g_{k_i}^{n_i}\vee f_{i-1}$. It is easy to verify that  $\{f_i\}_i\subseteq\Rr$ provides a suitable approximating sequence.
	\fr
\end{enumerate}
\end{rem}
%
We recall that functions in $\Cb(\XX)$ have the following  separation property, stronger than \eqref{eq:separK}: 
\begin{equation}
\label{eq:separ}
\begin{split}
&\text{Let $C\subseteq \XX$ closed and $A\subseteq\XX$ open with $C\subseteq A$. Then there exists a function}\\
&\text{ $\varphi\in \Cb(\XX)$ valued in $[0,1]$ such that $\varphi=1$ on a neighbourhood of $C$ and $\supp\varphi\subseteq A$.}
\end{split}
\end{equation}
This is an instance of Urysohn's lemma in the normal space $(\XX,\tau)$. More concretely, using the distance $\dist$ we can post-compose the function $\frac{\dist(x,\XX\setminus A)}{\dist(x,C)+\dist(x,\XX\setminus A)}$ with a continuous function $\psi:\RR\to[0,1]$ identically 1 on a neighbourhood of 1 and identically 0 on a neighbourhood of 0.

\bigskip

The presence of both a norm and a product with  functions allows to localize the concept of norm and to give some notion of `support' as follows:
\begin{defn}[Local seminorms] Let $\VV$ be a normed  $\mathcal R$-module. Then for  $A\subseteq\XX$  open and $v\in\VV$, we define
\[
	\Vert v\Vert_{| A}\defeq\sup\left\{\Vert f v\Vert: f\in\mathcal R,\ \supp f\subseteq A\text{ and }\Vert f\Vert_\infty \le 1 \right\}.
\]
Then define the germ seminorm of $v$, $\abs{v}_g:\XX\rightarrow\RR$, by
\[
		\abs{v}_g(x)\defeq\inf_A\Vert v\Vert_{| A},
\]	where the infimum is taken among all open neighbourhoods of $x$.
\end{defn}
\begin{defn}[`Supports']\label{def:supp} Let $\VV$ be a normed  $\mathcal R$-module. Then for  $C\subseteq\XX$  closed and $v\in\VV$  we say   $\supp v \subseteq C$ provided $\|v\|_{\XX\setminus C}=0$.

More generally, for $B\subseteq\XX$ Borel we say that   $\supp v\subseteq B$ provided $\supp v\subseteq C$ for some $C\subseteq B$ closed.
\end{defn}

Let us collect few simple properties of these definitions:
\begin{enumerate}
\item The concepts of local seminorm, germ seminorm and support all depend also on the ring $\mathcal R$, so that if $\VV$ is both a normed $\mathcal R$-module and $\mathcal R'$-module, then the associated notions of seminorm and support may depend on which ring we are using in the definitions above. Nevertheless, this will not make much difference when considering local vector measures, see Remark \ref{nonimportaR}.
\item For every $A\subseteq\XX$ open and $x\in\XX$ both $\Vert\,\cdot\,\Vert_{| A}$ and $|{\,\cdot\,}|_g(x)$ are seminorms on $\VV$.
\item If $\{x_n\}_n$ is converging to $x$, then eventually $x_n$ belongs to any given neighbourhood of $x$, hence
\[
\abs{v}_g(x)\defeq\inf_A\Vert v\Vert_{| A}\geq \inf_A\limsup_{n}\abs{v}_g(x_n)=\limsup_{n}\abs{v}_g(x_n).
\]
Thus $\abs{v}_g$ is upper semicontinuous, hence Borel measurable. 
\item For every $\varphi\in\mathcal R$, $v\in\VV$ we have
\begin{equation}
\label{eq:locnorm}
\abs{\varphi v}_g=|\varphi|\abs{v}_g\qquad\text{ on }\XX.
\end{equation}
To see this, let $\varepsilon>0$, $x\in\XX$ and $A\subseteq\XX$ be an open neighbourhood of  $x$ so small that $\Vert v\Vert_{| A}\leq \abs{v}_g(x)+\varepsilon$ and $|\varphi-\varphi(x)|\leq\varepsilon$ on $A$. Then letting $f$ vary in the set of functions in $\mathcal R$ with support in $A$ we get
\[
\begin{split}
\Vert \varphi v\Vert_{| A}=\sup_f\Vert f\varphi v\Vert\leq \sup_f\Big(\Vert (\varphi-\varphi(x))fv\Vert +|\varphi|(x)\Vert fv\Vert \Big)\leq \varepsilon\Vert v\Vert +|\varphi|(x)\Vert  v\Vert_{| A},
\end{split}
\]
where in the last inequality we used the fact that $\|(\varphi-\varphi(x))f\|_\infty\leq \varepsilon$ and the compatibility condition \eqref{compeq}.  This proves $\leq$ in \eqref{eq:locnorm}. The opposite inequality follows along the same lines starting from $ \Vert f\varphi v\Vert\geq  -\Vert (\varphi-\varphi(x))fv\Vert +|\varphi|(x)\Vert fv\Vert $.
\item A direct consequence of Definition \ref{def:supp} is that
\begin{equation}
\label{eq:suppprod}
\supp v\subseteq B\qquad\Rightarrow\qquad \supp(\varphi v)\subseteq\supp \varphi \cap B.
\end{equation}
Indeed the inclusion $\supp(\varphi v)\subseteq B$ is obvious from the definition. On the other hand for any $f\in \mathcal R$ with $\supp f\subseteq\XX\setminus\supp \varphi$ we have $f(\varphi v)=0$, which  is the same as to say  $\|\varphi v\|_{\XX\setminus\supp\varphi }=0$.
\item The inclusion \eqref{eq:suppprod} gives
\begin{equation}
\label{eq:c3}
f\in\mathcal R\ \text{ and }\ \supp v\subseteq\text{\{interior of $\{f=1\}$\}}\qquad\Rightarrow\qquad fv=v.
\end{equation}
Indeed, the conclusion is equivalent to $(1-f)v=0$, i.e.\ to $\supp((1-f)v)=\emptyset$. Now let $C\subseteq\{\text{interior of $\{f=1\}$}\}$ be closed such that $\supp v\subseteq C$, notice that  $C\cap \supp(1-f)=0$ and conclude by \eqref{eq:suppprod}.
\item If $\VV$ is a normed $\Cb(\XX)$-module (i.e.\ if $\mathcal R=\Cb(\XX)$) we have
\begin{equation}
\label{eq:disjsupp}
\{v_i\}_{i=1,\dots,n}\subseteq \VV\text{ with disjoint supports }\qquad\Rightarrow\qquad \Vert v_1+\dots+v_n\Vert= \max_i \Vert v_i\Vert,
\end{equation}
where having disjoint supports means that there are pairwise disjoint closed sets $C_i$ such that $\supp v_i\subseteq C_i$ for every $i$.

Indeed, we use the metric $\dist$ to see that $(\XX,\tau)$ is normal, and hence find (iteratively) $A_1,\dots,A_n$ open and pairwise disjoint such that $C_i\subseteq A_i$ for every $i$. Then, take $f_i$ for $C_i\subseteq A_i$ as in \eqref{eq:separ} and use \eqref{compeq} to conclude that $\leq$ holds in \eqref{eq:disjsupp}.   On the other hand, with this choice of functions $f_i$ we see that \eqref{eq:suppprod} and \eqref{eq:c3} imply $f_j\sum_iv_i=v_j$ and therefore $\|v_j\|\leq \|f_j\|_\infty\|\sum_iv_i\|$ by \eqref{compeq} (applied with $n=1$).
\item 	Definition \ref{def:supp} does not identify the support of an element of $\VV$ as a subset of $\XX$, but rather defines when an element of $\VV$ has support contained in a set. It is tempting to set
\begin{equation}\label{strangesupp}
	\supp v\defeq\overline{\{\abs{v}_g>0 \}}
\end{equation}
The problem with this definition is that it might be that
\begin{equation}
\label{eq:strano}
\Vert v\Vert_{\XX\setminus\supp v}>0,
\end{equation}
which violates the intuitive idea behind concept of support and of locality of the norm, see Example \ref{localityvv}.

Nevertheless, we point out that on one hand that in all the practical examples we have in mind, the definition of support as in \eqref{strangesupp} works as well as the one in Definition \ref{def:supp} and on the other that in order for an example like the one below one needs $\XX$ to be non compact and to use some version of the Axiom of Choice strictly stronger than the Axiom of Countable Dependent choice.
\end{enumerate}

\begin{ex}\label{localityvv}
Let $\XX\defeq\NN$ be endowed with the compete and separable distance $\dist(n,m)\defeq 1-\delta_{n}^m$, notice that $\Cb(\XX)=\ell^\infty$ and let $\VV:=\Cb(\XX)'=(\ell^\infty)'$. It is readily verified that the natural product operation defined as $f\cdot L(g):=L(fg)$ for any $L\in \VV$, $f,g\in\ell^\infty=\Cb(\XX)$ endow $\VV$ with the structure of normed $\Cb(\XX)$-module.

Let $W\subseteq\ell^\infty$ the subspace of sequences having limit and $L\in\VV$ be an element of the dual of $\Cb(\XX)$ that is obtained by extending - via Hahn-Banach -  the functional that associates to $f\in W$  its limit.

We claim that 
\begin{equation}
\label{eq:claimL}
|L|_g(x)=0\qquad\forall x\in \XX.
\end{equation}
Indeed, since the singleton $\{x\}$ is open, it is sufficient to prove that $f\cdot L=0$ for every $f\in \ell^\infty$ that is identically 0 outside $\{x\}$. But this is trivial by definition of product, because for any such $f$ and any $g\in\ell^\infty$ we have $fg\in W$ with limit 0, hence $f\cdot L(g)=0$ proving our claim \eqref{eq:claimL}.

It follows that the support as defined in \eqref{strangesupp} is empty and thus, since clearly $L\neq 0$, that \eqref{eq:strano} holds.
\fr
\end{ex}
Let now $(\VV',\Vert\,\cdot\,\Vert')$ denote the dual space of $(\VV,\Vert\,\cdot\,\Vert)$ (as normed vector space). It is well known that $(\VV',\Vert\,\cdot\,\Vert')$ is a Banach space.

We recall now the definition of ($\sigma$-additive) vector valued measure, see e.g.\  \cite[Chapters 1 and 2]{DiestelUhl77}. 
\begin{defn}
	A $\VV'$-valued measure is a map 
	$$
	\nvect:\mathcal{B}(\XX)\rightarrow \VV'
	$$
	that is \emph{$\sigma$-additive}, in the sense that if $\{A_k\}_{k\in\NN}$ is a sequence of pairwise disjoint sets in $\mathcal{B(\XX)}$, then
		$$\nvect\left(\bigcup_{k} A_k\right)=\sum_{k} \nvect(A_k),$$ where the convergence of the series has to be understood as convergence in norm in $\VV'$,
	\end{defn}
Notice that in particular the convergence of the series in the above equation is unconditional i.e.\ it is independent of the ordering of the terms.

\bigskip

In this paper we are concerned with a particular type of such measures, where the measure and the module structure interact in the following way:
\begin{defn}[Local vector measures]\label{mvmdef} Let $\VV$ be a normed $\mathcal R$-module.
A local vector measure defined on $\VV$ is a $\VV'$-valued measure $\nvect:\mathcal{B}(\XX)\rightarrow \VV'$
that is \emph{weakly local}, in the sense that
$$\nvect(A)(v)=0\quad\text{for every $A\subseteq\XX$ open and }v\in\VV\text{ such that }\Vert v\Vert_{|A}=0. $$
We denote the set of such local vector measures by $\MM_\VV$.
\end{defn}
  In what follows, we are going to consider local vector measures defined on the space $\VV$ fixed, unless it is specified otherwise. Notice that one may always assume that $\VV$ is complete, as the dual of the completion coincides with $\VV'$,  the completion naturally carries the structure of normed $\Cb(\XX)$-module and such structure remains compatible with the weak locality imposed above.

If $\nvect$ is a local vector measure and $v\in\VV$, we will often consider the finite signed Borel measure $$v\,\cdot\,\nvect\defeq \nvect(\,\cdot\,)(v).$$ 
By the the regularity of $v\,\cdot\,\nvect$ it trivially follows that
\begin{equation}
\label{eq:c1}
A\subseteq \XX\text{ open},\ 
B\subseteq A\text{ Borel},\ 
\Vert v\Vert_{|A}=0
\qquad\Rightarrow\qquad \nvect(B)(v)=0
\end{equation}
and this further implies that
\begin{equation}
\label{eq:c4}
 \supp v\subseteq B\subseteq B'\text{ with $B,B'\subseteq\XX$ Borel}\qquad\Rightarrow\qquad \nvect(B)(v)=\nvect(B')(v).
\end{equation}
Indeed, for $C\subseteq B$ closed such that $\supp v\subseteq C$ we have $\|v\|_{\XX\setminus C}=0$. Since $\XX\setminus C$ is open, the weak locality of $\nvect$ gives $\nvect(\XX\setminus C)(v)=0$ and thus \eqref{eq:c1} and the trivial inclusion $B'\setminus B\subseteq \XX\setminus C$ imply \eqref{eq:c4}.

We also notice that  for  $B\subseteq \XX$ Borel, $w\in\VV$ and  $f\in \mathcal R$ equal to 1 on an open neighbourhood $A$ of $B$, we have $\|(1-f)w\|_{|A}=0$ and thus $\nvect(A)(fw)=\nvect(A)(w)$ by weak locality. Hence  \eqref{eq:c1} gives
\begin{equation}
\label{eq:c2}
 \nvect(B)(w)=\nvect(B)(f w).
 \end{equation}

%
%
%

\bigskip

We recall the definition of total variation of a vector valued measure, that we are going to use specialized to the case of local vector measures.
\begin{defn} Let $\VV$ be a normed $\mathcal R$-module and   $\nvect$ be a $\VV'$-valued measure. Its total variation is the (countably additive) extended real valued Borel measure defined by $$\abs{\nvect}(E)\defeq\sup_{\pi}\sum_{A\in\pi}\Vert \nvect(A)\Vert'$$ where the supremum is taken over all finite Borel partitions $\pi$ of $E$. If $\abs{\nvect}(\XX)<\infty$, we say that $\nvect$  has bounded variation.
\end{defn}

We will see with the following Proposition \ref{localisfiniteabs} that the local vector measures we have just defined have automatically bounded variation. From the definition of total variation, it immediately follows 
\begin{equation}
\label{eq:trivialvar}
\abs{\nvect(E)(v)}\le\abs{\nvect}(E)\Vert v\Vert\qquad\text{ for every $E$ Borel and $v\in\VV$}.
\end{equation}

\begin{rem}
	 We remark that there is no effort in taking the (measure theoretic) completion of $\nvect$, defining it to be $0$ on all the subsets of $\abs{\nvect}$-negligible sets. Therefore, if we write $$\mathcal{N}\defeq\{Z\subseteq\XX:\text{there exists $Z'\subseteq\XX$ Borel such that $\abs{\nvect}(Z')=0$ and $Z\subseteq Z'$}\},$$
	 $\nvect$ is well defined on the $\sigma$-algebra generated by the union of $\mathcal{B(\XX)}$ and $\mathcal{N}$. 
	 \fr
\end{rem}

\begin{lem} Let $\VV$ be a normed $\mathcal R$-module and   $\nvect$ be a local vector measure on it.
	
	Then for every $A\subseteq\XX$ open we have
	\begin{equation}\label{nearliplintmp}
		\abs{\nvect(A)(v)}\le \abs{\nvect}(A)\Vert {v}\Vert_{|A}\quad{\text{for every }v\in\VV},
	\end{equation}
	where the right hand side is intended as 0 in the case $\infty\cdot 0$.
\end{lem}
\begin{proof} If $\Vert {v}\Vert_{|A}=0$ the conclusion follows by weak locality. Thus we can assume $\Vert {v}\Vert_{|A}>0$ and then $\abs{\nvect}(A)<\infty$ (as otherwise the conclusion is obvious). Then the restriction of $|\nvect|$ to $A$ is inner regular and thus for any fixed $\varepsilon>0$ we can find $K\subseteq A$ compact set such that $\abs{\nvect}(A\setminus K)\le\varepsilon$.  take  $\varphi\in\Cb(\XX)$ as in \eqref{eq:separK}  for $K\subseteq A$. 
	Then, $$\abs{\nvect(A)(v)}\le \abs{\nvect(K)(v)}+\varepsilon \Vert {v}\Vert\stackrel{\eqref{eq:c2}}=\abs{\nvect(K)(\varphi v)}+\varepsilon \Vert {v}\Vert\stackrel{\eqref{eq:trivialvar}}\le \abs{\nvect}(K)\Vert \varphi v\Vert+\varepsilon \Vert {v}\Vert$$
	and, as $\varepsilon>0$ is arbitrary, the claim is proved recalling that $\Vert \varphi v\Vert\le\Vert{v}\Vert_{|A} $.
\end{proof}
We then have the following general result:
\begin{prop}\label{localisfiniteabs}
Let $\VV$ be a normed $\mathcal R$-module and   $\nvect$ be a local vector measure on it. 

Then $\nvect$ has bounded variation. More precisely
\begin{equation}\label{conicidence}
\abs{\nvect}(A)=\Vert\nvect(A)\Vert'\quad\text{for any $A\subseteq\XX$ Borel}.
\end{equation}

\end{prop}
\begin{proof}
	We divide the proof in several steps.\\
	\textsc{Step 1}.
	Let $A\subseteq \XX$ Borel with $\abs{\nvect}(A)=\infty$ let and $m\in\RR$. Taking into account the very definition of total variation, the definition of the dual norm and the inner regularity of the finite signed measure $v\,\cdot\,\nvect$ for $v\in\VV$, we can take $\{v_l\}_{l=1,\dots,L}\subseteq\VV$ with $\Vert v_l\Vert\le 1$ and $\{K_l\}_{l=1,\dots,L}$ pairwise disjoint compact subsets of $A$ such that $$m< \sum_{l=1}^L \nvect(K_l)(v_l).$$ 
	Then by \eqref{eq:separK}  we can take $\{\psi_l\}_{l=1,\dots,L}\subseteq\mathcal R$ with values in $[0,1]$ such that $\psi_l=1$ on a neighbourhood of $K_l$ and $\{\supp\psi_l\}_{l}$ are pairwise disjoint.
	We set then $v\defeq\sum_{l=1}^L \psi_l v_l\in\VV$ and we notice that by weak locality $\nvect(K_l)(v_l)=\nvect(K_l)(v)$. Hence, if we set $B\defeq\bigcup_{l=1}^L K_l$, $$m< \sum_{l=1}^L \nvect(K_l)(v_l)= \sum_{l=1}^L \nvect(K_l)(v)= \nvect(B)(v)\le \Vert\nvect(B)\Vert', $$
	where the last inequality follows as $\Vert v\Vert\le 1$ by \eqref{compeq}.
	To sum up, we have proved in this step that given $A\subseteq\XX$ Borel with $\abs{\nvect}(A)=\infty$ and $m\in\RR$, we can find a (compact) set $B\subseteq A$ such that $\Vert \nvect(B)\Vert'\ge m$.
	\\\textsc{Step 2}. Assume by contradiction $\abs{\nvect}(\XX)=\infty$. We construct two sequences $\{A_k\}_k$ and $\{B_k\}_k$ of Borel subsets of $\XX$ iteratively, starting from $A_0\defeq\XX$ and $B_0\defeq\XX$. 
	Precisely, given $k\in\NN$, $ k\ge 1$, assume that we have defined $A_0,\dots,A_k$ and $B_0,\dots, B_{k}$. Notice that by the construction we are going to do, $\abs{\nvect}(A_k)=\infty$ and $B_k\subseteq A_k$. Then, by additivity, either $\abs{\nvect}(B_k)=\infty$ or $\abs{\nvect}(A_k\setminus B_k)=\infty$. In the former case we set $A_{k+1}\defeq B_k$, in the latter $A_{k+1}\defeq A_k\setminus B_k$. In either case we then take $B_{k+1}\subseteq A_{k+1}$ such that $\Vert\nvect(B_{k+1})\Vert'\ge {k+1}$, using step 1. 
	
	Now notice that for every $k$, $\Vert\nvect(B_{k+1})\Vert'\ge {k+1}$ and either $B_{k+1}\subseteq B_k$ or $B_k\cap B_{l}=\emptyset$ for every $l\in\NN$, $l\ge k+1$ (this possibility may depend on $k$). If the second possibility occurs infinitely often, we can find a subsequence of pairwise disjoint subsets of $\XX$, $\{B_{k_l}\}_l$ such that $\Vert\nvect(B_{k_l})\Vert'\rightarrow\infty$. Otherwise, we can find a decreasing sequence of Borel subsets of $\XX$, $\{B_{k_l}\}_l$, such that $\Vert\nvect(B_{k_l})\Vert'\rightarrow\infty$. 
	But in both cases this leads to a contradiction with the countable additivity of $\nvect$ (that involves the convergence in norm).
	\\\textsc{Step 3}. We show \eqref{conicidence}. By regularity, there is no loss of generality in assuming $A$ open. Let $\varepsilon>0$. With the same notation as in step 1, we find $$\abs{\nvect}(A)-\varepsilon\le \sum_{l=1}^L \nvect(K_l)(v_l).$$ 
	Being $\abs{\nvect}$ a finite measure, we can find pairwise disjoint open sets $\{A_l\}_l$ such that $K_l\subseteq A_l\subseteq A$  and $\abs{\nvect}(A_l\setminus K_l)<\varepsilon/L$ for every $l$. Take then $\{\psi_l\}_l\subseteq\Cb(\XX)$ as in step 1, but with the additional constraint that $\supp\psi_l\subseteq A_l$ for every $l$. Then, by weak locality,
	$$  \sum_{l=1}^L \nvect(K_l)(\psi_l v_l)\le\sum_{l=1}^L \nvect(A_l)(\psi_l v_l)+\varepsilon=\sum_{l=1}^L \nvect(A)(\psi_l v_l)+\varepsilon=\nvect(A)(v)+\varepsilon$$
	where $v\defeq\sum_{l=1}^L \psi_l v_l\in\VV$. Being $\varepsilon>0$ arbitrary, we can conclude the proof.
\end{proof}
From the finiteness of the total variation we easily get the following:
\begin{prop}\label{vectvalbanach} Let $\VV$ be a normed $\mathcal R$-module.
	Then the space $(\MM_\VV,\abs{\,\cdot\,}(\XX))$ is a Banach space.
\end{prop}
\begin{proof}
	It is trivial to verify that $\abs{\,\cdot\,}(\XX)$ is indeed a norm. Let now $\{\nvect_n\}_n$ be a Cauchy sequence. 
	Then, if $B$ is Borel and $v\in\VV$, also $\{\nvect_n(B)(v)\}_n\subseteq\RR$ is a Cauchy sequence, so that we can define (notice that in this way we get immediately weak locality)
	\begin{equation}\label{deflimit}
		\nvect(B)(v)\defeq\lim_n \nvect_n(B)(v).
	\end{equation}
	We have $$\abs{\nvect(B)(v)}=\lim_n\abs{\nvect_n(B)(v)}\le\liminf_n \abs{\nvect_n}(B)\Vert v\Vert$$
	for every $B$ Borel. In particular, $\nvect(B)\in\VV'$ and
	$$\Vert\nvect(B)\Vert'\le\liminf_n \abs{\nvect_n}(B).$$
	Clearly, $\nvect$ is finitely additive, but the equation above implies also $\sigma$-additivity: indeed, if $\{B_k\}_k$ is a sequence of pairwise disjoint Borel sets, setting $B^k\defeq\bigcup_{j=1}^k B_j$ we can compute, for any $m$,
	$$ \Vert\nvect(B^\infty)-\nvect(B^k)\Vert'\le\liminf_n\abs{\nvect_n}(B^\infty\setminus B^k)\le\liminf_n \abs{\nvect_n-\nvect_m}(\XX)+ \abs{\nvect_m}(B^\infty\setminus B^k)$$
	and notice that taking first $m$ big enough and then $k$ big enough, the right hand side of the above inequality converges to $0$. Therefore $\nvect$ is a local vector measure.
	
	Now if $B$ is Borel and $v\in\VV$ with $\Vert v\Vert\le 1$, similar computations as above show that $$\Vert (\nvect-\nvect_n)(B)\Vert'\le \liminf_m\abs{\nvect_m-\nvect_n}(B)$$ and then, thanks to the definition of total variation and the super additivity of the $\liminf$, $$\abs{\nvect-\nvect_n}(\XX)\le\liminf_m\abs{\nvect_m-\nvect_n}(\XX).$$
	This implies convergence in norm of $\nvect_n$ to $\nvect$.
\end{proof}
Justified by this proposition,  here and below, when we write $\MM_\VV$, we mean the Banach space $(\MM_\VV,\abs{\,\cdot\,}(\XX))$.

 In view of the following definition, we briefly recall the definition of Bartle integral that we take from \cite[P. 5]{DiestelUhl77} (see also the original article \cite{Bar}). If $\nvect$ is a vector valued measure and $f=\sum_{i=1}^n c_i\chi_{A_i}$ is a simple function, where $\{c_i\}_{i}\subseteq\RR$ and $\{A_i\}_i$ are pairwise disjoint Borel subsets, then we consider the map 
 \begin{equation}
\label{eq:defbart}
 A\mapsto \int_A f\dd{\nvect}\defeq\sum_{i=1}^n c_i \nvect(A\cap A_i)\in\VV'.
 \end{equation}
 It is clear that for any such $f$ and at most countable disjoint family $\{B_j\}_j\subseteq\XX$ we have
 \[
 \Big\|\sum_j\int_{B_j}f\dd{\nvect}\Big\|'\leq \|f\|_{\Lp^\infty(|\nvect|)} \sum_j|\nvect|(B_j)= \|f\|_{\Lp^\infty(|\nvect|)}|\nvect|(\cup_jB_j),
 \]
 having used also \eqref{conicidence}. This inequality shows that \eqref{eq:defbart} defines a linear and continuous map from the space of simple functions (endowed with the supremum norm) to the space of $\VV'$-valued measures, which therefore can be extended to a linear and continuous map from $\Lp^\infty(|\nvect|)$ to the space of $\VV'$-valued measures (see also Bartle's bounded convergence theorem  \cite[Theorem 2.4.1]{DiestelUhl77}). Also, recalling \eqref{eq:c1} for the case of simple $f$'s and then arguing by approximation we see that the measures in the image are weakly local, i.e.\ are local vector measures defined on $\VV$.
 
We summarize all this in  the following definition:
%
%
%
\begin{defn}\label{bartdef} Let $\VV$ be a normed $\mathcal R$-module, $\nvect$ be a local vector measure defined on it and let $f:\XX\rightarrow\RR$ be a bounded Borel function. We define $f\nvect$ as the local vector measure given by $$f \nvect(A)\defeq \int_A f\dd{\nvect}$$
	where the integral has to be understood as a Bartle integral, i.e.\ in the sense described above.
\end{defn}
We point out that Bartle's bounded convergence theorem (see e.g.\ \cite[Theorem 2.4.1]{DiestelUhl77}) ensures that
\begin{equation}
\label{eq:bbct}
\left.\begin{array}{l}
\sup_n\|f_n\|_{\Lp^\infty(|\nvect|)}<\infty,\\ 
f_n\to f\quad |\nvect|\text{-a.e.}
\end{array}\right\}
\qquad\Rightarrow\qquad f_n\nvect(A)\to f\nvect(A)\quad in\ \VV'\qquad\forall A\subseteq\XX\ Borel.
\end{equation}
We  notice the following general fact:
\begin{prop}\label{absfoutside} Let $\VV$ be a normed $\mathcal R$-module,  $\nvect$ be a local vector measure defined on it and let $f:\XX\rightarrow\RR$ be a bounded Borel function. Then it holds, as measures,
	\begin{equation}\label{letotvar}
		\abs{f\nvect}=\abs{f}{\abs{\nvect}}.
	\end{equation}
\end{prop}
\begin{proof}
	We show first that $	\abs{f\nvect}\le\abs{f}{\abs{\nvect}}$. Clearly it suffices to show that if $B$ is Borel, then $\Vert f \nvect(B)\Vert'\le \int_B\abs{f}\dd{\abs{\nvect}}$. This follows from the triangle inequality if $f$ is a simple function and then the claim is proved by approximation. 
	
	Therefore, we conclude if we show that $$\int_\XX\abs{ f}\dd{\abs{\nvect}}\le \abs{f\nvect}(\XX). $$ As Proposition \ref{localisfiniteabs} shows that $\abs{\nvect}(\XX)<\infty$, an approximation argument justified by the inequality  in \eqref{letotvar} that we have just proved yields that we can assume that $f$ is simple, say $f=\sum_j c_j\chi_{B_j}$, where $c_j\in\RR$ for every $j$ and $\{B_j\}_j$ is a finite Borel partition of $B$. 
	Now we can compute \begin{equation*}\notag
		\int_B \abs{f}\dd{\abs{\nvect}}=\sum_j \abs{c_j}\abs{\nvect}(B_j)=\sum_j \abs{c_j\nvect}(B_j)\stackrel{\eqref{eq:defbart}}=\sum_j\abs{f\nvect}(B_j)=\abs{ f\nvect}(B).\qedhere
	\end{equation*}

\end{proof}
\begin{lem}\label{foutside} Let $\VV$ be a normed $\mathcal R$-module, and let  $\nvect$ be a local vector measure.
	
	Then, for every $f\in\mathcal R $ and $v\in\VV$, we have
	\begin{equation}\label{claimfoutside}
		(f v)\,\cdot\,\nvect=f(v\,\cdot\,\nvect)=v\,\cdot\,(f \nvect).
	\end{equation} 
\end{lem}
\begin{proof}
	The second equality in \eqref{claimfoutside} follows directly from the definition of $f\nvect$ (and an approximation argument), so it is enough to prove the equality of the first and last term.
	By regularity of the signed measures $(f v)\,\cdot\,\nvect$ and $ v\,\cdot\,(f \nvect)$ we can just show that $$\nvect(A)(f v)=f\nvect(A)(v) $$
	when $A$ is open. Recall that Proposition \ref{localisfiniteabs} implies that $\nvect$ and $f\nvect$ have bounded variation. As $f$ is bounded, we can assume that $\abs{f(x)}\le 1$ for every $x$. Fix now $\varepsilon>0$ and let $$-1-\epsilon=t_0<t_1<\dots<t_n=1+\epsilon$$ be a collection of real numbers such that for every $i=1,\dots,n$, $t_i-t_{i-1}\le \varepsilon$ and
	\begin{equation}\label{liplintmp}
		\abs{\nvect }(\{f=t_i\})=0 \quad\text{and} \quad\abs{f \nvect }(\{f=t_i\})=0)\qquad\text{for every }i=0,\dots,n.
	\end{equation}
	Consider now the open sets $\{A_i\defeq A\cap\{f\in(t_{i-1},t_i)\}\}_{i=1,\dots,n}$ and notice that, for every $i=1,\dots,n$, $$f\nvect(A_i)(v)=t_i\nvect(A_i)(v)+(f-t_i)\nvect(A_i)(v)$$
	and also $$\nvect(A_i)(f v)=t_i\nvect(A_i)(v)+\nvect(A_i)(f v-t_i v).$$
	Subtracting these two equalities term by term and summing over $i$, recalling \eqref{liplintmp},
	\begin{equation}\label{nnearliplintmp}
		\abs{\nvect(A)(f v)-f\nvect(A)(v )}\le \sum_{i=1}^n \abs{\nvect(A_i) (f v - t_i v)}+\abs{(f-t_i)\nvect(A_i)(v)}.
	\end{equation}
	We then use \eqref{letotvar} and \eqref{nearliplintmp} to deduce from \eqref{nnearliplintmp} that
	$$ \abs{\nvect(A)(f v)-f\nvect(A)(v )}\le \sum_{i=1}^n 2\varepsilon\abs{\nvect }(A_i)\Vert v\Vert\le 2\varepsilon \abs{\nvect}(\XX)\Vert v\Vert.$$
	As $\varepsilon>0$ is arbitrary, this concludes the proof.
\end{proof}

Having a notion of `total variation' naturally leads to the following definition:
\begin{defn}[Polar decomposition] Let $\VV$ be a normed $\mathcal R$-module, and let  $\nvect$ be a local vector measure. 
	We say that  $\nvect$ admits a polar decomposition if there exists a weakly$^*$ $\abs{\nvect}$-measurable map $L_\nvect:\XX\rightarrow\VV'$ such that 	$\nvect=L_\nvect \abs{\nvect}$, in the sense that for every $v\in\VV$, we have $L_\nvect(v)\in\Lp^1(\abs{\nvect})$ and
	\begin{equation}\label{reprpolarabs}
		\nvect(A)(v)= \int_A  L_\nvect(v)\dd{\abs{\nvect}}\quad\text{for every $A\subseteq\XX$ Borel},
	\end{equation}
	where here and in what follows we write $L_\nvect(v)$ for the map $x\mapsto L_\nvect(x)(v)$.  We require moreover that for every $x\in\XX$ it holds that
	\begin{equation}\label{inequalitygerm}
		\abs{L_\nvect(x)(v)}\le \abs{v}_g (x)\quad\text{for every }v\in\VV.
	\end{equation}
\end{defn}
It easily follows from the above definition that the polar decomposition is `weakly unique', in the sense that if $\nvect=L_\nvect\abs{\nvect}=L'_\nvect\abs{\nvect}$, then for every $v\in\VV$, $$L_\nvect(v)=L'_\nvect(v)\quad\abs{\nvect}\text{-a.e.}$$
Also, if $f:\XX\rightarrow\RR$ is a bounded Borel function, then 
$$ 
f\nvect(A)(v)=\int_A f L_\nvect(v)\dd{\abs{\nvect}}.
$$
{Finally, we remark that if $L_\nvect$ satisfies \eqref{inequalitygerm}, then for every $x\in\XX$, $L_\nvect(x)$ is tight.}

In view of the following proposition, it is useful to recall the definition of essential supremum (see e.g.\ \cite[Lemma 3.2.1]{GP19} for the well known proof of existence and uniqueness and \cite[Section 2.4]{MNP91}  for much more on the topic).
\begin{lem}
	Let $(\XX,\mu)$	be a measure space with $\mu$ $\sigma$-finite.
  If $\{f_\alpha\}_{\alpha\in A}$ is a collection of extended real valued $\mu$-measurable functions, then there exists a unique (up to equality $\mu$-a.e.) $\mu$-measurable function $g:\XX\rightarrow\{\pm\infty\}$, called the essential supremum of $\{f_\alpha\}_{\alpha\in A}$ and denoted by $\esssup_{\alpha\in A} f_\alpha$ (or $\mu-\esssup_{\alpha\in A} f_\alpha$ when we want to stress the dependence on the measure), such that 
  \begin{enumerate}[label=\roman*)]
  	\item $f_\alpha\le g\ \mu$-a.e.\ for every $\alpha\in A$,
  	\item if $f_\alpha\le h\ \mu$-a.e.\ for every $\alpha\in A$, then $h\ge g\ \mu$-a.e.
    \end{enumerate} 
\end{lem}

\begin{prop}\label{allhaspolarabs} Let $\VV$ be a normed $\mathcal R$-module, and let  $\nvect$ be a local vector measure. 

Then $\nvect$ admits the (`weakly unique') polar decomposition $L_\nvect{\abs{\nvect}}$, where $L_\nvect$ satisfies the `weak' identity
	\begin{equation}\notag
	L_\nvect(v)=\dv{(v\,\cdot\,\nvect)}{\abs{\nvect}}\quad\abs{\nvect}\text{-a.e.\ for every }v\in\VV.
	\end{equation} 
Moreover, it holds that
\begin{equation}\label{essusp}
\abs{\nvect}-\esssup_{v\in\VV, \Vert v\Vert\le 1} L_\nvect(v)=1.
\end{equation}

\end{prop}
\begin{proof}
	For every $v\in\VV$ \eqref{nearliplintmp} grants that $v\,\cdot\,\nvect\ll\abs{\nvect}$,
	so that we can define $$\rho_v(x)\defeq  \dv{(v\,\cdot\,\nvect)}{\abs{\nvect}}\in\Lp^1(\abs{\nvect})$$ and notice that \begin{equation}\label{pretrivbound}
	\nvect(A)(v)=\int_A\rho_v\dd{\abs{\nvect}}
\end{equation} whenever  $A\subseteq\XX $ is Borel. 
	
	To define the map $L_\nvect$ we follow the argument given in \cite[Lemma 3.2]{GigliHan13}.
	Let the maps $\leb:\Lploc^1(\abs{\nvect})\rightarrow\mathcal{B(\XX)}$ and $\borrep:\Lploc^1(\abs{\nvect})\rightarrow\{f:\XX\rightarrow\RR \text{ Borel}\}$ be given by Corollary \ref{thmapp} in the Appendix. 
	For every $x\in\XX$, define $$\VV_x\defeq\left\{v\in\VV: x\in\leb(\rho_v)\right\}.$$ Property $iii)$ in Corollary \ref{thmapp} and the linearity of the map $v\mapsto\rho_v$ grant that for every $x\in\XX$, $\VV_x$ is a vector subspace of $\VV$ and that $$L_\nvect(x)(v)\defeq \borrep(\rho_v)(x)\quad\text{for }v\in \VV_x$$
	is linear in $v$. Now, \eqref{nearliplintmp}, \eqref{pretrivbound} and an immediate approximation argument easily yield that for every $x\in\XX$
	$$
	\Vert\rho_v\Vert_{\Lp^\infty(|\nvect|\mressmall B_r(x))}\le \Vert v\Vert_{|B_r(x)}
	$$
	so that, using also \eqref{buoncontrollo} we get that for every $x\in\XX$,
	\begin{equation}\label{trivbound}
		\abs{L_\nvect(x)(v)}\le\abs{ v}_g(x)\quad\text{for every } v\in \VV_x.
	\end{equation}  Fix $x\in\XX$ and consider the equivalence relation on $\VV$ given by the seminorm $\abs{\,\cdot\,}_g(x)$. Let then ${\WW}_x$ be the quotient space, endowed with the factorization of $\abs{\,\cdot\,}_g(x)$. Also, by \eqref{trivbound}, $L_\nvect(x)$ factorizes to  the projection of $\VV_x$ in ${\WW}_x$. Then, using Hahn-Banach Theorem, we can extend the factorization of $L_\nvect(x)$ to ${\WW}_x$ and then, taking the composition with the projection, we have an extension of $L_\nvect(x)$ to a map defined on $\VV$ which still satisfies \eqref{trivbound} for every $v\in\VV$.
	
	Property $i)$ in Corollary \ref{thmapp} implies that for any $v\in\VV$, we have $x\in\leb(\rho_v)$ for $\abs{\nvect}$-a.e.\ $x$, so that, using also property $ii)$, $$L_\nvect(v)=\rho_v\quad\abs{\nvect}\text{-a.e.}$$ This implies for every $v\in\VV$, $L_\nvect(v)$ is $\abs{\nvect}$-measurable and satisfies $$\nvect(A)(v)=\int_A \rho_v\dd{\abs{\nvect}}=\int_A L_\nvect(v)\dd{\abs{\nvect}}\quad \text{for every $A\subseteq\XX$ Borel}.$$
	
	Now we prove the last claim. Set for brevity $g\defeq \esssup_{v\in\VV, \Vert v\Vert\le 1}  L(v)$.
	As for every $B\subseteq \XX$ Borel and $v\in\VV$ it holds $${\int_B L(v)\dd{\abs{\nvect}}}=\abs{\nvect(B)(v)}\le \abs{\nvect}(B)\Vert v\Vert,$$
	we see that $g\le 1\ \abs{\nvect}$-a.e.
Now, let $B\subseteq\XX$ Borel such that $\abs{\nvect}(B)<\infty$ and let $\pi$ be a finite Borel partition of $B$.
	If $\varepsilon>0$, we can find $\{v_A\}_{A\in\pi}\subseteq\VV$ such that $\Vert v_A\Vert\le 1$ for every $A\in\pi$ and $$\sum_{A\in\pi}\Vert \nvect(A)\Vert'\le\sum_{A\in\pi} \nvect(A)(v_A)+\varepsilon=\sum_{A\in\pi}\int_A L_\nvect(v_A)\dd{\abs{\nvect}}+\varepsilon\le \sum_{A\in\pi}\int_A g\dd{\abs{\nvect}}+\varepsilon.$$
	Being $\varepsilon$ and $\pi$ arbitrary, it follows
	$$ \abs{\nvect}(B)\le\int_B g \dd{\abs{\nvect}}$$
	that, combined with $g\le 1\ \abs{\nvect}$-a.e.\ and the fact that $\abs{\nvect}$ is finite, implies $g =1\ \abs{\nvect}$-a.e.
\end{proof}
\begin{rem} As the proof shows, rather than imposing \eqref{inequalitygerm} to hold, we could only ask for the bound
\begin{equation}
\notag
\abs{L_\nvect(x)(v)}\le \|v\|\qquad\abs{\nvect}\text{-a.e.\ }x\in\XX
\end{equation}
to hold for every $v\in\VV$. Nevertheless, even with this weaker requirement, all the conclusions of the proposition remain in place, up to the fact that   \eqref{inequalitygerm}
has to be interpreted as 
$$
|L_\nvect(x)(v)|\le |v|_g(x)\qquad|\nvect|\text{-a.e.\ $x$, for every $v\in\VV$},
$$
where this last inequality is a consequence of \eqref{buoncontrollo} of Corollary \ref{thmapp} in the Appendix, as in the proof of Proposition \ref{allhaspolarabs}.

We also point out that the use of Hahn-Banach theorem is not really required in the proof. In fact, the rather constructive argument gives a collection $(\VV_x,L_\nvect(x))$ indexed by $x\in\XX$ with $\VV_x$ subspace of $\VV$ and $L_\nvect(x)\in \VV_x'$ for every $x\in\XX$ with the following properties:
\begin{itemize}
\item[i)] For every $v\in\VV$ the set of $x\in\XX$ such that $v\in \VV_x$ is Borel and with complement $|\nvect|$-negligible,
\item[ii)] The  map $x\mapsto L_\nvect(x)(v)$ set, say, to 0 if $v\notin\VV_x$ is Borel and bounded in norm by $\|v\|$,
\item[iii)] The identity \eqref{reprpolarabs} (that makes sense thanks to the above and $|\nvect|(\XX)<\infty$) holds for every $v\in\VV$.
\end{itemize}
These properties could also be used as definition of polar decomposition: the price one pays for doing so is to keep track of the subspaces $\VV_x$ where the operators $L_\nvect(x)$ are defined, but in doing so it gains the Borel regularity stated in $\rm i)-ii)$ above, in place of $|\nvect|$-measurability, and a Choice-free proof. Notice in particular that the Axiom of Choice  in a form stronger than Countable Dependent Choice - i.e.\ in the form of the general Hahn-Banach theorem  - is used in the proof once for each point $x\in\XX$ in order to extend the operators $L_\nvect(x)$ to regions where almost surely they won't be applied and that such extensions are irrelevant from the perspective of the defining formula  \eqref{reprpolarabs}.
\fr\end{rem}

	\begin{rem}\label{nonimportaR} Let $\VV$ be both a normed $\Rr$-module and a normed $\Rr'$-module, where $\Rr$ and $\Rr'$ approximate open sets. Assume also $\Rr\subseteq\Rr'$.
		Let $\nvect:\mathcal B(\XX)\rightarrow \VV'$ be a $\VV'$-valued measure.
		Then the following assertions are equivalent (notice that the local seminorm $\Vert\,\cdot\,\Vert_{|A}$ may not be independent of the choice of the subring $\Rr$ or $\Rr'$):
		\begin{enumerate}
			\item $\nvect$ is weakly local, considering $\VV$ as a normed $\Rr$-module, 
			\item $\nvect$ is weakly local, considering $\VV$ as a normed $\Rr'$-module. 
		\end{enumerate}
		We prove now this assertion. Let us denote with $\Vert\,\cdot\,\Vert_{A,\Rr}$ and $\Vert\,\cdot\,\Vert_{A,\Rr'}$ the local seminorms induced by the structure of normed $\Rr$-module and normed $\Rr'$-module, respectively.
		Clearly, $\Vert\,\cdot\,\Vert_{A,\Rr}\le\Vert\,\cdot\,\Vert_{A,\Rr'}$ as $\Rr\subseteq\Rr'$, so that it is immediate to see that $(1)\Rightarrow (2)$. Conversely, assume that $\nvect$ is weakly local, considering $\VV$ as a normed $\Rr'$-module. Take $v\in\VV$ with $\Vert v\Vert_{A,\Rr}=0$, for some open set $A$. Now we write, exploiting Proposition \ref{allhaspolarabs},
		$\nvect(A)(v)=\int_A L_\nvect(v)\dd{|\nvect|}$. Take now $\{f_k\}_k\subseteq \Rr$ as in item $(1)$ of Remark \ref{rem:Rpropr} and we use Lemma \ref{foutside} together with dominated convergence to compute
		$$
		\nvect(A)(v)=\lim_k \int_A f_k L_\nvect(v)\dd{|\nvect|}=\lim_k \int_A L_\nvect( f_k v)\dd{|\nvect|}=\lim_k \nvect(A)(f_k v)=0,
		$$
		where we used that $\Vert v\Vert_{A,\Rr}=0$ in the last equality. The conclusion follows.		
		
		Notice that we indeed showed what follows: if $\nvect$ satisfies one of the equivalent conditions above, then, for every $A\subseteq\XX$ open,
		$$
		|\nvect(A)(v)|\le \abs{\nvect}(A)\Vert v\Vert_{A,\Rr}\le\abs{\nvect}(A)\Vert v\Vert_{A,\Rr'}
		$$
		(the total variation of $\nvect$ is by definition independent of the choice of the subring $\Rr$ or $\Rr'$). This is due to the fact that we assume that $\Rr$ approximates open sets. 
		
		We point out that if $\nvect=L_\nvect|\nvect|$ is the polar decomposition given by Proposition \ref{allhaspolarabs} for the $\Rr'$-normed module $\VV$, it may be false that (with the obvious notation) for every $x\in\XX$,
		$$
		|L_\nvect(x)(v)|\le |v|_{g,\Rr}(x)\quad\text{for every $v\in\VV$.}
		$$ 
		Clearly, this issue can be immediately solved building  the polar decomposition for the $\Rr$-normed module $\VV$ instead of considering the polar decomposition for the $\Rr'$-normed module $\VV$.
		\fr
	\end{rem}

If $\mu$ is a finite (signed) measure on the Polish space $\XX$, then
\begin{equation}
\label{eq:riesz}
\varphi\quad\mapsto\quad F(\varphi):=\int \varphi\,\dd\mu
\end{equation}
is a continuous linear functional on $C(\XX)$ with operator norm equal to $|\mu|(\XX)$. The classical and celebrated Riesz-Markov-Kakutani theorem ensures that  if $\XX$ is compact, then all linear functionals on $C(\XX)$ can be represented this way. The non-compact case is more delicate, and handled by the Daniell-Stone theorem: if $\XX$ is Polish then  finite measures correspond, via the map \eqref{eq:riesz}, to those functionals $F:\Cb(\XX)\to\RR$ such that
\begin{equation}
\label{eq:tightF}
F(\varphi_n)\to 0\quad\text{ whenever $(\varphi_n)\subseteq\Cb(\XX)$ is such that $\varphi_n(x)\searrow 0$ for every $x\in\XX$.}
\end{equation}
We shall call property \eqref{eq:tightF}  \emph{tightness} (in the literature it is also called `continuity in 0'). Given that we are now going to investigate the validity of a Riesz-like theorem in our setting, it is natural to look for a counterpart of tightness in the framework we are working now. We propose the following:
\begin{defn}[Tightness]\label{def:tight} Let $\VV$ be a normed $\mathcal R$-module, and  $F\in\VV'$. We say that $F$ is tight with respect to $\mathcal{R}$ if for every $v\in\VV$ and every sequence
\begin{equation}\label{defsequ}
\{\varphi_n\}_n\subseteq\mathcal{R}\quad\text{with }\varphi_n(x)\searrow 0\text{ for every }x\in\XX,
\end{equation}
we have $F(\varphi_n v)\rightarrow 0$.

We will dispense with specifying the ring $\mathcal{R}$ if it is clear from the context and in the case $\mathcal{R}=\Cb(\XX)$.
\end{defn}
\begin{rem}\label{tightrem}
	Notice that, if every $v\in\VV$ has compact support (i.e.\ contained in a compact set - this occurs in particular if the space is compact), then every functional $F\in\VV'$ is tight. This follows easily from Dini's monotone convergence theorem. 
	
	In general, already the case $\VV=\Cb(\XX)$ shows that (if one assumes a sufficiently strong version of Choice, then)  not every functional $F\in\VV'$ is tight: see e.g.\   the functional $\ell\in\VV'$ defined in Example \ref{localityvv}.\fr
\end{rem}
The link between the concept of tightness and that of measure (positive real-valued for the moment - but this will later be further clarified) is given in the following key lemma. Notice that here the assumption that $\mathcal R$ approximates open sets is crucial.
\begin{lem}\label{bbtightiffmeas}
Let $\VV$ be a normed $\Rr$-module and  $F\in\VV'$. Then $F$ is tight with respect to $\mathcal{R}$ if and only if 
\begin{equation}\label{mudef}
	\mu(A)\defeq\sup
	\left\{F(v):v\in\VV,\ \Vert v\Vert \le 1\text{ and } 
	\supp v\subseteq A \right\}
\end{equation}
is the restriction to open sets of a finite Borel measure.
\end{lem}
\begin{proof}
Assume that $\mu$ as in \eqref{mudef} is the restriction to open sets of a finite Borel measure. We prove that $F$ is tight with respect to $\mathcal{R}$. Fix then $v\in\VV$ and let $\{\varphi_n\}_n$ be as in \eqref{defsequ}. Fix also $\varepsilon>0$. We use  the regularity of $\mu$ to find a compact set $K\subseteq\XX$ such that $\mu(\XX\setminus K)\le\varepsilon$. By Dini's monotone convergence Theorem, up to discarding finitely many functions in $\{\varphi_n\}_n$, we can assume that ${\varphi_{{n}}(x)}< \varepsilon$ for every $x\in K$. By \eqref{eq:separK}  we can take a sequence $\{\phi_n\}_{n}\in\mathcal R$ valued in $[0,1]$ such that for every $n$, $\supp\phi_n\subseteq\{\varphi_n<\varepsilon\}$ and $\phi_n=1$ on a neighbourhood of $K$. Then,
$$ \abs{F(\varphi_n v)}\le\abs{F(\varphi_n(1-\phi_n) v)}+\abs{F(\varphi_n\phi_n v)}\le \Vert v\Vert\left(\mu(\XX\setminus K)+\varepsilon\right)\le 2\varepsilon\Vert v\Vert,$$
where we used the definition of $\mu$ as $\supp (\varphi_n(1-\phi_n) v)\subseteq\XX\setminus K$.  Then the conclusion follows as $\varepsilon>0$ is arbitrary.

Conversely, assume that $F$ is tight with respect to $\mathcal{R}$. We prove that $\mu$ as defined in \eqref{mudef} is the restriction to open sets of a finite Borel measure. To this aim,
we can use Carathéodory criterion (see e.g.\ \cite{AFP00}) and is then enough to verify (all the sets in consideration are assumed to be open):
\begin{enumerate}
	\item $\mu(A)\le\mu(B)$ if $A\subseteq B$,
	\item $\mu(A\cup B)\ge\mu(A)+\mu(B)$ if $\dist(A,B)>0$,
	\item $\mu(A)=\lim_k\mu(A_k)$ if $A_k\nearrow A$,
	\item$\mu(A\cup B)\le \mu(A)+\mu(B)$.
\end{enumerate}
We notice that $(1)$ and $(2)$ follow trivially from the definition of $\mu$ and that $(2)$ does not even need the sets to be well separated (so that we do not even need to consider the distance $\dist$). 

We prove now property $(3)$. Take $v\in\VV$ with $\Vert v\Vert \le 1$ and $\supp v\subseteq A$. Let then $C\subseteq A$ be a closed set with $\supp v\subseteq C$.  
By the fact that $\mathcal R$ approximates open sets, take $\{\psi_n\}_n\subseteq \mathcal R$ and, for every $k$, $\{\varphi^k_n\}_n\subseteq \mathcal R$ that are valued in $[0,1]$ such that $\psi_n \nearrow \chi_{\XX\setminus C}$ and such that for every $k$, $\varphi_n^k\nearrow\chi_{A_k}$. We can, and will, assume that $\supp\psi_n\subseteq \XX\setminus C$ and that $\supp\varphi_n^k\subseteq A_n^k$ for every $n,k$. Let now $(k_i,n_i)_{i\in\NN}$ be an enumeration of $\NN^2$ and define
\[
\xi_i:=\psi_i\vee\hat \varphi_i\in\mathcal R\qquad\text{ where }\qquad \hat\varphi_j:= \bigvee_{j\leq i} \varphi^{k_j}_{n_j}.
\]
It is clear that $\xi_i\nearrow 1$, hence by tightness we have
\begin{equation}
\label{eq:dat}
F(v)=\lim_i F(\xi_i v).
\end{equation}
Also, the identity $\xi_i=\hat\varphi_i+(\psi_i-\hat\varphi_i)^+$, the inclusion $\supp((\psi_i-\hat\varphi_i)^+)\subseteq\supp \psi_i \subseteq \XX\setminus C$ and \eqref{eq:suppprod} give  $\supp (\xi_i v)\subseteq  A_{\hat k_i}$, where $\hat k_i:=\max_{j\leq i}k_j$. Therefore recalling  \eqref{eq:dat}   we have
$$F(v)=\lim_i F(\xi_i v)\le \limsup_i \mu(A_{\hat k_i})= \lim_k\mu(A_k)$$
and since this holds for every $v\in\VV$ with $\|v\|\leq 1$, we proved
$\mu(A)\le \lim_k \mu(A_k)$ and  the claim.

We pass to $(4)$. Take $v\in\VV$ with $\supp v\subseteq A\cup B$, say $\supp v\subseteq C\subseteq A\cup B$ for some $C\subseteq\XX$ closed. As $\mathcal R$ approximates open sets, take functions $\{\psi^A_n\}_n,\{\psi^B_n\}_n\subseteq\mathcal R$ that are valued in $[0,1]$, such that $\psi^A_n\nearrow A$ and for every $n$, $\supp\psi^A_n\subseteq A$ and such that the corresponding properties hold for $\{\psi^B_n\}_n$. Take $\{\psi_n\}_n\subseteq\mathcal R$ valued in $[0,1]$ such that $\psi_n\nearrow \chi_{\XX\setminus C}$ and for every $n$, $\supp\psi_n\subseteq\XX\setminus C$.
Finally, let $\{\xi_n\}_n\subseteq\mathcal R$ be defined by $\xi_n\defeq \psi_n\vee\psi_n^A\vee \psi_n^B$. It is easy to verify that $\xi_n\nearrow 1$ and - arguing as before - that for every $n$, we can write $$\xi_n v=(\psi_n^A \vee \psi_n^B) v+ (\psi_n-(\psi_n^A \vee \psi_n^B))^+ v=\psi_n^A v+(\psi_n^B -\psi_n^A)^+ v.$$ Hence
$$F(v)=\lim_n F(\gamma_n v)\le \limsup_n F(\psi_n^A v)+F((\psi_n^B -\psi_n^A)^+ v)\le \mu(A)+\mu(B)$$
and   taking the supremum among all $v$ as above we conclude.
\end{proof}
We then have the following version of the Riesz representation theorem (compare e.g.\ with  \cite[Section 7.10]{Bogachev07}):
\begin{thm}[Riesz's theorem for local vector measures]\label{weakder} Let $\VV$ be a normed $\Rr$-module and  $F\in\VV'$ be tight. 
There exists a unique local vector measure $\nvect_F$ defined on $\VV$ such that
\begin{equation}\label{defining}
		\nvect_F(\XX)(v)=F(v)\quad\text{for every }v\in\VV.
\end{equation}
Moreover, it holds that $\abs{\nvect_F}=\mu$, where $\mu$ is the finite Borel measure given by Lemma \ref{bbtightiffmeas}.
\end{thm}
\begin{proof}
Let $\mu$ be the finite Borel measure given by Lemma \ref{bbtightiffmeas}.
Let $B\subseteq\XX$ Borel. Given $\varepsilon>0$, take $K_\varepsilon\subseteq B\subseteq A_\varepsilon$, with $K_\varepsilon$ compact and $A_\varepsilon$ open such that $\mu(A_\varepsilon\setminus K_\varepsilon)\le \varepsilon$. Let $\varphi_\varepsilon\in\mathcal R$ be as in \eqref{eq:separK}  for $K_\varepsilon\subseteq A_\varepsilon$.
Notice that
\begin{equation}\label{riesztmp}
	\abs{F(\varphi_\varepsilon v)}\le\mu(A_\varepsilon)\Vert \varphi_\varepsilon v\Vert\le\mu(A_\varepsilon)\Vert  v\Vert.
\end{equation} 

We define the linear functional 
\begin{equation}
\label{eq:apprN}
\nvect_F(B)(v)\defeq\lim_{\varepsilon\searrow 0} F(\varphi_\varepsilon v)
\end{equation}
where $\varphi_\varepsilon$ is any function as above for that given $\varepsilon>0$. We claim that this limit exists and is independent of the choice of $\{K_\varepsilon,A_\varepsilon,\varphi_\varepsilon\}_{\varepsilon>0}$. 
Indeed, take $\varepsilon>0$ and let $\varphi_\varepsilon$ and $\varphi_\varepsilon'$ as above. Then, with the obvious notation, $\varphi_\varepsilon=\varphi_{\varepsilon}'=1$ on a neighbourhood of $K_\varepsilon\cap K_{\varepsilon}'$. Thus $\supp( \varphi_\varepsilon-\varphi_{\varepsilon}')\subseteq( (A_\varepsilon\cup A_{\varepsilon}')\setminus( K_\varepsilon\cap K_{\varepsilon}'))$ implies
\begin{equation}\label{rieszfund}
	\abs{F(\varphi_\varepsilon v)- F(\varphi_{\varepsilon} 'v)} \le \Vert v\Vert \mu((A_\varepsilon\cup A_{\varepsilon}')\setminus( K_\varepsilon\cap K_{\varepsilon}'))\le 2 \Vert v\Vert \varepsilon.
\end{equation} 
Notice that the inequality above holds in particular if $\varphi_\varepsilon'=\varphi_{\varepsilon'}$ with $0<\varepsilon'<\varepsilon$. Also, by \eqref{riesztmp}, $\nvect_F(B)\in\VV'$ for every $B$ Borel. 

Notice that  the very definition of $\nvect_F$ implies that \eqref{defining} holds (too see this, just argue as for \eqref{rieszfund} with $1$ in place of $\varphi_{\varepsilon}'$). 
We show now that $\nvect_F$, as just defined, is indeed a local vector measure. By \eqref{riesztmp}, it holds $$\abs{\nvect_F(B)(v)}\le\mu(B)\Vert v\Vert\quad\text{for every $B\subseteq$ Borel},$$
so that  $\nvect_F(B)\in\VV'$ with 
\begin{equation}\label{fjsdnodwq}
	\Vert \nvect_F(B)\Vert'\le \mu(B)\qquad\text{for every }B\subseteq\XX\text{ Borel}.
\end{equation}
Finite additivity of $\nvect_F(\,\cdot\,)$ follows easily from its definition using a suitable choice of cut-off functions, while to prove $\sigma$-additivity one only has to use \eqref{fjsdnodwq} noticing that if $\{B_k\}_k$ is a sequence of pairwise disjoint Borel sets, it holds  that $\sum_{k=n}^\infty\mu( B_k)\rightarrow 0$ as $n\rightarrow\infty$.
Therefore $\nvect_F$ is a $\VV'$-valued measure.

To show weak locality we pick $A\subseteq\XX$ open, $v\in\VV$ with $\|v\|_{|A}=0$ and notice that in the construction above we can pick $A_\varepsilon=A$ for every $\varepsilon>0$. With this choice we have $\varphi_\varepsilon v=0$ and thus $\nvect(A)(v)=0$, as desired.

Also, by \eqref{fjsdnodwq}, $\abs{\nvect_F}\le \mu$ and, by \eqref{defining}, it is clear that $\mu(\XX)\le\abs{\nvect_F}(\XX)$ so that we have indeed $\mu=\abs{\nvect_F}$.

Uniqueness follows immediately from \eqref{defining} and \eqref{conicidence} as they grant that
\[
|\nvect-\tilde\nvect|(\XX)=\|\nvect-\tilde\nvect\|'(\XX)=\sup_{\|v\|\leq 1}(\nvect-\tilde\nvect)(\XX)(v)=0
\]
for every $\nvect,\tilde\nvect$ satisfying the conclusion.
\end{proof}

\begin{rem} The standard proof of Riesz's theorem typically starts taking $L\in C(\XX)'$ (say $\XX$ compact), decomposes it in its positive and negative parts to reduce to the case of positive functionals, then by monotonicity finds the value of the measure on open and/or compact sets and finally by approximation the value of the measure on any set. Inspecting our arguments, we see that the mathematical principles behind the proof of Theorem \ref{weakder} are similar, even though the lack of an order on $\VV$ forces us to avoid arguments by monotonicity in favour of those based on approximation and domination (as in \eqref{eq:apprN}, \eqref{fjsdnodwq}).

Let us illustrate how - somehow conversely - one can recover from our statement the classical Riesz's theorem about the dual of $\Cc(\XX)$ for $\XX$ locally compact metric space. Start noticing that  $\VV=\Cc(\XX)$ is a normed $\Cb(\XX)$-module and that, as seen in Remark \ref{tightrem},  any $F\in\Cc(\XX)'$ is automatically tight.  Thus by Theorem \ref{weakder} and Proposition \ref{allhaspolarabs} we can  represent $F$ as $L\abs{\nvect}$, so that 
$$
F(f)=\int_\XX L(f)\dd{\abs{\nvect}}\quad\text{for every }f\in\Cc(\XX).
$$ 
Using local compactness and then separability, we build a countable sequence $\{\varphi_n\}_n\subseteq\Cc(\XX)$ such that $\varphi_n(x)\in[0,1]$ for every $x\in\XX$ and the interiors of $\{\varphi_n=1\}$ cover  $\XX$ (to show this use e.g.\ the Lindel\"{o}f property of $\XX$). Then we define  $\sigma:\XX\rightarrow\RR$ as $$\sigma(x)\defeq L(x)(\varphi_n)\quad\text{on the interior of $\{\varphi_n=1\}$}$$ (such function is  well defined up to $\nvect$-negligible sets and is a  $|\nvect|$-measurable map). Building upon Lemma \ref{foutside}, it is easy to verify that the identity $L(f)(x)= f(x)\sigma(x)$ holds $\abs{\nvect}$-a.e.\ for every $f\in \Cc(\XX)$. On the other hand, the identity \eqref{essusp} in this case reads as 
 $$
\esssup_{f\in \Cc(\XX),\ \Vert f\Vert\le 1} L(f)=1,$$
which in turn easily implies $\sigma(x)\in\{\pm 1\}$ $\abs{\nvect}$-a.e.. Collecting what observed so far we see that the measure $\mu\defeq\abs{\nvect}\mres \{\sigma=1\}-\abs{\nvect}\mres \{\sigma=-1\},$ satisfies
$$ F(f)=\int_\XX f\dd{\mu}\quad\text{for every }f\in\Cc(\XX),$$
as desired.
\fr
\end{rem}
A direct consequence of this last result is the following isomorphism of Banach spaces:
\begin{cor}\label{Rieszcor}  Let $\VV$ be a normed $\Rr$-module and  consider the Banach space 
\begin{align}
{\sf T}&\defeq\left(\left\{F\in\VV': \text{$F$ is tight w.r.t.\ $\mathcal R$}\right\},\Vert\,\cdot\,\Vert'\right).\notag
\end{align}
Then the map 
\[
\Psi:\MM_\VV\rightarrow {\sf T}\quad\text{defined as}\quad \nvect\mapsto \nvect(\XX)
\]
is a bijective isometry.
\end{cor}
\begin{proof}
It is easy to show that $\sf T$ is a Banach space. Also, we know that $\Psi$ is linear, takes values in $\VV'$,  preserves the norm (by \eqref{conicidence}) and that its image contains ${\sf T}$ thanks to Theorem \ref{weakder}.

Thus it only remains to prove that  $\Psi(\nvect)$ is a tight element of $\VV'$ for any $\nvect\in \MM_\VV$. Thus fix $\nvect$, let $v\in\VV$ and let $\{\varphi_n\}_n$ be as in \eqref{defsequ}. Also, let $\varepsilon>0$ and take, by regularity of $|\nvect|$, a compact set $K\subseteq\XX$ such that $\abs{\nvect}(\XX\setminus K)<\varepsilon$. By Dini's monotone convergence Theorem, up discarding finitely many functions of $\{\varphi_n\}_n$, we can assume that $\varphi_n<\varepsilon$ on $K$ for every $n$, By continuity,  $\varphi_n<\varepsilon$ on an open neighbourhood of $K$, say $A_n$, for every $n$. Set also $S\defeq\sup_n\Vert \varphi_n\Vert_{\infty}<\infty.$
We can then compute, recalling \eqref{nearliplintmp} and using the trivial bound $\|\varphi_n v\|_{|A_n}\leq\varepsilon\|v\|$,
$$\abs{\nvect(\XX)(\varphi_nv)}\le \abs{ \nvect(A_n)(\varphi_n v)}+\abs{\nvect(\XX\setminus A_n)}(\varphi_nv)\le\varepsilon\abs{\nvect}(A_n)\Vert v\Vert +S\varepsilon\Vert v\Vert\le\varepsilon \Vert v\Vert  \left(\abs{\nvect}(\XX)+S\right),$$
so that the claim follows as $\varepsilon>0$ is arbitrary.
\end{proof}
Another direct consequence of Theorem \ref{weakder}, this time in conjunction with the classical theorem by Alexandrov about weak sequential completeness of the space of measures, is the following result. Notice that in order to apply Alexandrov's theorem we need to work in the case $\mathcal R=\Cb(\XX)$.
\begin{cor}[Alexandrov's theorem for local vector measures]\label{alex} Let $\VV$ be a normed $\Cb(\XX)$-module and $\{\nvect_n\}_n\in \MM_\VV$ be a sequence such that for every $v\in\VV$ the sequence $n\mapsto\nvect_n(\XX)(v)\in\RR$ is Cauchy.

Then there exists $\nvect\in \MM_\VV$ such that
\[
\nvect(\XX)(v)=\lim_n\nvect_n(\XX)(v)\qquad\forall v\in\VV.
\]
\end{cor}
\begin{proof} Define $F\in\VV'$ as $F(v):=\lim_n\nvect_n(\XX)(v)$ (linearity of $F$ is obvious, continuity follows from Banach-Steinhaus Theorem). To conclude it is enough, by Theorem \ref{weakder}, to show that  $F$ is tight.  

Thus fix $v$ and define $G_n,G:\Cb(\XX)\to\RR$ as $G_n(\,\cdot\,)\defeq F_n(\,\cdot\,v)$ and $G(\,\cdot\,)\defeq F(\,\cdot\,v)$:  then the conclusion follows by  the classical Alexandrov's Theorem. We  give the details. By Riesz's Theorem (\cite[Theorem 7.10.1]{Bogachev07}), as $G_n$ is tight, it is induced by a Baire measure $\mu_n$, that is 
$$ G_n(f)=\int_\XX f\dd\mu_n\quad\text{for every }f\in\Cb(\XX).$$
Recall that $\mu_n$ is a Borel measure, as Baire measures and Borel measures coincide for metric (hence Polish) spaces, e.g.\ \cite{Bogachev07}.
 Now, $\mu_n$ is weakly fundamental (\cite[Definition 2.2.2]{bogaweak}) as 
$$
\int_\XX f\dd\mu_n=G_n(f)\rightarrow G( f),
$$
hence by \cite[Theorem 2.3.9]{bogaweak}, we see that $\mu_n$ converges weakly to a Borel measure $\mu$, so that $G(f)=\int_\XX f\dd\mu$. It follows that $G$ is tight, by \cite[Theorem 7.10.1]{Bogachev07} again.
\end{proof}
We also notice that another byproduct of Theorem \ref{weakder}, and in particular of the equality $|\nvect_F|=\mu$ is that
\begin{equation}
\label{eq:locmass}
A\subseteq\XX \text{ open and $\nvect(A)(v)=0$ for every $v\in\VV$ with $\supp v\subseteq A$}\qquad\Rightarrow\qquad |\nvect|(A)=0.
\end{equation}

\bigskip

In classical measure theory it often happens that one first defines a measure via its action on  a certain class of regular functions (say Lipschitz) and then, once the measure is constructed, its action on more general functions (say continuous) is uniquely defined by some density argument.

The following proposition establishes a construction of this sort in our setting. In the statement below notice that \eqref{eq:extn2} is not a direct consequence of \eqref{eq:extn} because in computing the total variation of $\hat\nvect$ and $\nvect$ we use the norm in $\VV'$, $\WW'$ respectively.
\begin{prop}\label{lipbscb}
	Let $\VV$ be a normed $\Cb(\XX)$-module and let $\WW\subseteq\VV$ be  a subspace that, with the inherited product and norm, is a normed $\mathcal R$-module. Assume also that  the space $\Cb(\XX)\cdot\WW$ of $\Cb(\XX)$-linear combinations of elements of $\WW$ is dense in $\VV$ and let $\nvect$ be a local vector measure defined on $\WW$.
	
	Then there is a unique local vector measure $\hat\nvect$ defined on $\VV$ that extends $\nvect$, i.e.\ such that
\begin{equation}
\label{eq:extn}
\hat\nvect(B)(v)=\nvect(B)(v)\qquad\forall v\in\WW,\ B\subseteq\XX\ Borel
\end{equation}
and such measure $\hat\nvect$ also satisfies
\begin{equation}
\label{eq:extn2}
\begin{split}
|\hat\nvect|(B)=|\nvect|(B),\qquad\forall B\subseteq\XX\ Borel.
\end{split}
\end{equation}
More explicitly, for every $A\subseteq\XX$ open and $v\in\VV$,
	\begin{equation}\label{eq:extn3}
		|\hat\nvect(A)(v)|\le |\nvect|(A)\Vert v\Vert_{A},
	\end{equation}
where the local seminorm at the right hand side is with respect the structure of $\Rr$-normed module for $\VV$.

%
%
%
%
\end{prop}
\begin{proof}\ First notice that \eqref{eq:extn3} is a `self-improvement' due to Remark \ref{nonimportaR}, once we have proved the remaining parts of the statement.\\
{\bf Existence} Recalling Definition \ref{bartdef} we wish to define
\begin{equation}
\label{eq:defext}
\hat\nvect(B)\big(\sum_if_iv_i\big):=\sum_if_i\nvect(B)(v_i)
\end{equation}
for any choice of $n\in\NN$, $f_i\in\Cb(\XX)$, $v_i\in\WW$, $i=1,\ldots,n$ and $B\subseteq \XX$ Borel.

To prove that this is a well posed definition we claim that, with the same notations, it holds
\begin{equation}
\label{eq:claimxext}
\big|\sum_if_i\nvect(B)(v_i)\big|\leq |\nvect|(B)\Big\|\sum_{i=1}^n f_i v_i\Big\|.
\end{equation}
To see this we fix  $\varepsilon>0$ and find  $K_\varepsilon\subseteq \XX$ compact so that  $|\nvect|(\XX\setminus K_\varepsilon)\leq\varepsilon$. Then we use item (3)  of Remark \ref{rem:Rpropr} to find,  for any $i=1,\dots,n$, a function $\varphi^\varepsilon_i\in\mathcal R$ with $|f_i-\varphi^\varepsilon_i|<\varepsilon$ on $K_\varepsilon$ and thus  on some neighbourhood $A_\varepsilon$ of $K_\varepsilon$ independent on $i$. Then  find $\psi \in\Rr$ taking values in $[0,1]$ with support in $A_\varepsilon$ and identically equal to 1  on  $K_\varepsilon$.

With these choices we have 
\[
\begin{split}
\big|\sum_if_i\nvect(B)(v_i)\big|&\leq\big|\sum_i(1-\psi)f_i\nvect(B)(v_i)\big|+\big|\sum_i\psi(f_i-\varphi^\varepsilon_i)\nvect(B)(v_i)\big|+\big|\nvect(B)\big(\sum_i\psi\varphi^\varepsilon_i v_i\big )\big|\\
&\leq \sum_i\big(|(1-\psi)f_i||\nvect|\big)(\XX) \|v_i\|+\sum_i\big(|\psi(f_i-\varphi^\varepsilon_i)||\nvect|\big)(\XX)(v_i)+|\nvect|(B)\big\|\sum_i\psi\varphi^\varepsilon_i v_i\big\|\\
&\leq \varepsilon\sum_i\|f_i\|_\infty\|v_i\|+\varepsilon|\nvect|(\XX)\sum_i\|v_i\|+|\nvect|(B)\big\|\sum_i\psi\varphi^\varepsilon_i v_i\big\|
\end{split}
\]
and 
\[
\begin{split}
\big\|\sum_i\psi\varphi^\varepsilon_i v_i\big\|&\leq \sum_i\|\psi(f_i-\varphi^\varepsilon_i) v_i\|+\big\|\psi\sum_i f_i v_i\big\|\leq \varepsilon  \sum_i\|v_i\|+\big\|\sum_i f_i v_i\big\|.
\end{split}
\]
These last two inequalities and the arbitrariness of $\varepsilon$ give the claim \eqref{eq:claimxext}.

In turn, the bound  \eqref{eq:claimxext} ensures that the right hand side of \eqref{eq:defext} depends only on $v:=\sum_{i=1}^n f_i v_i$ and not on the way $v$ is written as such sum. In particular, the definition \eqref{eq:defext} is well posed and we have
\begin{equation}
\label{eq:boundhatn}
|\hat\nvect(B)(v)|\leq  |\nvect|(B)\|v\|.
\end{equation}
for every $B\subseteq\XX$ and  $v\in \Cb(\XX)\cdot \WW$. The fact that $\hat\nvect$ is linear on the space of such $v$'s is obvious by definition and this last inequality shows continuity: together with the density assumption this ensures that $\hat\nvect(B)$ can be uniquely extended to an element of $\VV'$, still denoted  $\hat\nvect(B)$. It is clear by definition that the map $B\mapsto   \hat\nvect(B)$ is additive; $\sigma$-additivity follows  trivially from the bound \eqref{eq:boundhatn} and the $\sigma$-additivity of $|\nvect|$.

To prove weak locality of $\hat\nvect$ we pick $A\subseteq\XX$ open and $v\in \VV$ with $\|v\|_{|A}=0$. Then by the very definition of support we can find   $\varphi\in \Cb(\XX)$ with support in $\XX\setminus A$ and $\supp v\subseteq \{\varphi=1\}$. Then we pick  $\{v_j\}_j\in \Cb(\XX)\cdot \WW$ converging to $v$ and notice that  $\varphi v_j\to\varphi v=v$ (recall \eqref{eq:c3}). On the other hand, writing $v_j=\sum_if_{ij}v_{ij}$ with $f_{ij}\in\Cb(\XX)$ and $v_{ij}\in\WW$ we see that
\[
\sum_i\varphi f_{ij}\nvect(A)(v_{ij})=\sum_if_{ij}\nvect(A)(\varphi v_{ij})=0\qquad\forall i\in\NN
\]
by weak locality of $\nvect$ and the fact that $\supp(\varphi v_{ij})\subseteq \supp \varphi \subseteq \XX\setminus A$. Hence $\hat\nvect(\varphi v_i)=0$ for any  $i$ and letting $i\to\infty$ we conclude, by the arbitrariness of $A,v$, that $\hat\nvect$ is weakly local, as desired.

It is then clear that \eqref{eq:extn} holds and that inequality \eqref{eq:boundhatn} gives $\leq$ in \eqref{eq:extn2}. The opposite inequality is trivial because, recalling \eqref{conicidence}, we see that $|\hat\nvect|(B)$ is the operator norm of $\hat\nvect(B)$ in $\VV'$, whereas $|\nvect (B)|$ is the operator norm of $\nvect (B)$ in $\WW'$ i.e.\ of the restriction of $\hat\nvect(B)$ to $\WW$.

\noindent{\bf Uniqueness} Let $\hat\nvect$ be an extension of $\nvect$, $f\in \Cb(\XX)$, $\varphi\in\mathcal R$ and $v\in\WW$. Then for any $B\subseteq \XX$ Borel we have
\[
|\hat\nvect(B)(fv)-\varphi\nvect(B)(v)|\stackrel{\eqref{claimfoutside}}=|(f-\varphi)\hat\nvect(B)(v)|\stackrel{\eqref{letotvar}}\leq\|v\|(|f-\varphi||\hat\nvect|)(B)\leq \|v\|\|f-\varphi\|_{\Lp^1(|\nvect|)}.
\] 
Now observe that since $|\nvect|$ is a finite measure, item (3) of Remark \ref{rem:Rpropr}  ensures that for  any $f\in\Cb(\XX)$ there is $\{\varphi_n\}_n\in\mathcal R$ uniformly bounded converging to $f$ pointwise. Thus the convergence is also in $\Lp^1(|\nvect|)$, hence the above and \eqref{eq:bbct} imply that any extension $\hat\nvect$ must satisfy $\hat\nvect(fv)=f\nvect(v)$ for any $v\in\WW$ and $f\in\Cb(\XX)$.  By linearity and the density of $\Cb(\XX)\cdot\WW$ in $\VV$ it follows that any extension $\hat\nvect$, if it exists, must satisfy \eqref{eq:defext}. Since such equation defines the value of $\hat\nvect$ on $\Cb(\XX)\cdot\WW$ and this is dense in $\VV$, uniqueness is proved. 
\end{proof}

\bigskip

We conclude the section describing some operations on local vector measures. We start with the push forward through continuous maps and start with the following definition:
\begin{defn}[Push-forward module]\label{defpfmod}
Let $(\XX,\dist)$ and $(\YY,\rho)$ be two complete and separable metric spaces, $\VV$  a normed $\Cb(\XX)$-module and $\varphi\in  C(\XX,\YY)$.

The normed $\Cb(\YY)$-module  $\varphi_* \VV$ is defined as $\VV$ as normed vector space and equipped  with the structure of  algebraic module over $\Cb(\YY)$ by
\begin{equation}\notag
	f v\defeq f\circ\varphi\, v\quad\text{for every }f\in\Cb(\YY)\text{ and }v\in\VV.
\end{equation}
\end{defn}
It is easy to verify that $\varphi_*\VV$ is a normed $\Cb(\YY)$-module: we just have to notice that if $\{f_i\}_{i=1,\dots,n}\subseteq\Cb(\YY)$ have pairwise disjoint support, then also $\{f_i\circ\varphi\}_{i=1,\dots,n}\subseteq\Cb(\XX)$ have the same property.

The following proposition shows how if we have a local vector measure on $\VV$ we can naturally build a local vector measure on $\varphi_*\VV$ via a push-forward operation:
\begin{prop}[Push-forward of local vector measures] With the same notation and assumptions as in Definition \ref{defpfmod} the following holds. Let $\nvect$ be a local vector measure defined on $\VV$. Define a map by
\begin{equation}\label{definingpushf}
	\varphi_*\nvect(B)(v)\defeq\nvect(\varphi^{-1}(B))(v)\quad\text{for every $B\subseteq\YY$ Borel and $v\in\varphi_*\VV$}.
\end{equation}
Then $\varphi_*\nvect$ is a local vector measure defined on $\varphi_*\VV$ and $\abs{\varphi_*\nvect}=\varphi_*\abs{\nvect}$.
\end{prop}
\begin{proof}
We show that $\varphi_*\nvect$ is indeed a local vector measure. The only non trivial verification to be done is weak locality. Take then an open set $A\subseteq\YY$ and $v\in\varphi_*\VV$ with $\Vert v\Vert_{| A}=0$; in other words, $f\circ\varphi\, v=0$ for every $f\in\Cb(\YY)$ with $\supp f\subseteq A$. We have to show that 
$$ \nvect(\varphi^{-1}(A) )(v)=0.$$
Take any $\varepsilon>0$, then, by regularity, take a compact set $K\subseteq \varphi^{-1}(A)\subseteq\XX$ with \begin{equation}\label{solita}
	\abs{\nvect}(\varphi^{-1}(A)\setminus K)<\varepsilon.
\end{equation} 
Then $\varphi(K)$ is compact and contained in $A$, hence there is  $f\in\Cb(\YY)$ with $\supp f\subseteq A$ and $f(y)=1$ on a neighbourhood of $\varphi(K)$.  Therefore, by \eqref{solita} and weak locality of $\nvect$,
\begin{equation*}
	\abs{\nvect (\varphi^{-1}(A))(v)}\le \varepsilon\Vert v\Vert+ \abs{\nvect (K)(v)}\le \varepsilon\Vert v\Vert+\abs{\nvect(K)(f\circ\varphi v)}= \varepsilon\Vert v\Vert,
\end{equation*}
where we used the fact that $\supp f\subseteq A$ for the last equality. Being $\varepsilon>0$ arbitrary, this proves the claim.

By \eqref{definingpushf} and \eqref{conicidence}, we conclude that for every $B\subseteq\YY$ Borel, 
\begin{equation*}
	\abs{\varphi_*\nvect}(B)= \abs{\nvect}(\varphi^{-1}(B))=\varphi_*\abs{\nvect }(B).\qedhere
\end{equation*}
\end{proof}

\bigskip 

It may happen that we have a normed $\Cb(\XX)$-module $\VV$ and we want to consider its Cartesian product with itself. To do this, first we have to endow  $\VV^k$ with a norm, then the normed $\Cb(\XX)$-module operations will be defined component-wise. 
Let $k\in\NN$, $k\ge 1$, endow $\VV^k$ with a norm equivalent to the norm $$\Vert (v_1,\dots,v_k)\Vert\defeq \Vert (\Vert v_i\Vert)_{i=1,\dots,k}\Vert_e\quad\text{for every }v=(v_1,\dots,v_k)\in\VV.$$
Notice that the canonical map $\Phi:(\VV')^k\rightarrow (\VV^k)'$ defined by $$\Phi(\phi_1,\dots,\phi_k)(v_1,\dots,v_k)=\phi_1(v_1)+\cdots+\phi_k(v_k)$$
is an isomorphism.

If one has in mind that $\VV$ is some space of vector fields over a manifold, then $\VV^k$ would correspond to some related tensor field. Then a little bit of matrix operation is possible over corresponding local vector measures:
\begin{defn}\label{intmatrix}
	Let $\nvect$ be a local vector measure defined on  $\VV^n$.
	Let moreover $m\in\NN$, $m\ge 1$ and let $f=(f_{i,j})_{1\le i\le m, 1\le j\le n}:\XX\rightarrow\RR^{m\times n}$ be a bounded Borel map. We define $f \nvect$ as the local vector measure defined on  $\VV^m$ by
	\begin{equation}\notag
		f \nvect(A)\defeq\int_A f\dd{\nvect}\defeq\left(\sum_{j=1}^n\int_A f_{i,j}\dd{\nvect_j}\right)_i\quad\text{if $A\subseteq\XX$ is Borel},
	\end{equation} 
	where we exploited the canonical identification $(\VV')^k\simeq(\VV^k)'$.
\end{defn}
Notice that if $f:\XX\rightarrow\RR$ is a bounded Borel function and $\nvect$ is a local vector measure defined on $\VV^n$, then
$$ f{\nvect}= (f \Id_{n\times n})\nvect.$$

\subsection{Examples}
\subsubsection{Currents in metric spaces}\label{se:curr}
For this section, $(\XX,\dist)$ is a complete and separable metric space (here the distance matters) and  we fix $n\in\NN$, $n\ge 1$ (the case $n=0$ being trivial).

The following notions are extracted from \cite{AmbrosioKirchheim00}.
\begin{defn}\label{corrdefn}
A $n$-current is a multilinear map $$T:\LIPb(\XX)\times  \LIP(\XX)^n\rightarrow\RR$$ such that 
\begin{enumerate}[label=\roman*)]
\item  there exists a finite (positive) measure $\mu$ such that 
\begin{equation}\label{eqmass}
\abs{T(f,\pi_1,\dots,\pi_n)}\le \prod_{j=1}^n\Lip(\pi_j)\int_\XX \abs{f}\dd{\mu}
\end{equation}	
for every $f\in\LIPb(\XX)$ and $\pi_1,\dots,\pi_n\in \LIP(\XX)$. The minimal measure $\mu$ satisfying \eqref{eqmass} (that can be proved to exist) will be called the mass of $T$ and denoted by $\Vert T\Vert_{\AK}$;
\item if $f\in\LIPb(\XX)$ and for $j=1,\dots,n$, $\{\pi_j^i\}_i\subseteq\LIP(\XX)$ is a sequence of equi-Lipschitz functions such that $\pi_j^i\rightarrow\pi_j$ pointwise, then 
\[
\lim_i T(f , \pi_1^i,\dots,\pi^i_n)=T(f , \pi_1,\dots,\pi_n);
\]
\item if $f\in\LIPb(\XX)$ and for some $j=1,\dots,n$ we have that $\pi_j$ is constant on a neighbourhood of $\{f\ne 0\}$, then $$T(f , \pi_1,\dots,\pi_n)=0.$$
\end{enumerate}
\end{defn}
When the dimension $n$ is clear from the context, we call $n$-currents simply currents.
It is clear  from  \eqref{eqmass} that, if $T$ is a current, then it can be uniquely extended to a map $$T:\Cb(\XX)\times  \LIP(\XX)^n\rightarrow\RR$$
still satisfying \eqref{eqmass} for every $f\in\Cb(\XX)$ and $\pi_1,\dots,\pi_n\in \LIP(\XX)$. As such extension is unique, we will not introduce a different notation for it.
By \cite[(3.2)]{AmbrosioKirchheim00}, we have that a current is alternating in the last $n$ entries, so that we can set (all the algebraic operations are with respect to the field of real numbers $\RR$)
\begin{equation}\notag
	\DD^n(\XX)\defeq \Cb(\XX)\otimes \bigwedge^n \LIP(\XX)
\end{equation}
and consider a current $T$ as a linear map $T:	\DD^n(\XX)\rightarrow \RR$.
We also have a natural map $$ \Cb(\XX)\times \LIP(\XX)^n\rightarrow \DD^n(\XX)$$
and we write (just as a notation) $f\diff \pi_1\wedge\cdots\wedge \pi_n$ for the image of $(f,\pi_1,\dots,\pi_n)$ through such map. Therefore, by $T(f\diff\pi_1\wedge\cdots\wedge\diff\pi_n)$, we mean $T(f,\pi_1,\dots,\pi_n)$.

Notice also that if $T$ is a current and $f\in\Cb(\XX)$, then $f T$ defines a current by (see the discussion below \cite[(2.5)]{AmbrosioKirchheim00})  $$f T(v)\defeq T(f v)\quad\text{for every $v\in \DD^n(\XX)$}$$
and by \eqref{eqmass} (cf.\ the key result \cite[(2.5)]{AmbrosioKirchheim00} that encodes locality) it holds, as measures,
\begin{equation}\label{massmult}
\Vert f T\Vert_{\AK}\le \abs{f}\Vert T\Vert_{\AK}.
\end{equation}

We define a seminorm on $\DD^n(\XX)$ as follows:
\begin{equation}\label{defnsemi}
\Vert v\Vert\defeq\sup_T T(v)\quad\text{for every }v\in\DD^n(\XX)
\end{equation}
where the supremum is taken among all currents $T$ with $\Vert T\Vert_{\AK}(\XX)\le 1$. We claim that if  $v\in\DD^n(\XX) $ is so that $\Vert v\Vert=0$, then $\Vert f v\Vert=0$ for any $f\in\Cb(\XX)$. Indeed
$$
\Vert f v\Vert=\sup_{T:\Vert T\Vert_{\AK}\le 1} T(f v)=\sup_{T:\Vert T\Vert_{\AK}(\XX)\le 1} (f T)( v)\le \sup_{\tilde T:\Vert \tilde T\Vert_{\AK}(\XX)\le \Vert f\Vert_\infty} \tilde T(v)=0
$$
where the inequality above is due to \eqref{massmult}.
We then identify elements of $\DD^n(\XX)$  that are equal up to an element of zero norm, so that we have a normed vector space $(\DD^n(\XX),\Vert\,\cdot\,\Vert)$. Notice that our claim  grants that the structure of algebraic module over $\Cb(\XX)$ descends to the quotient.

We show now that $(\DD^n,\Vert\,\cdot\,\Vert)$ is a normed $\Cb(\XX)$-module. The only non trivial verification is that \eqref{compeq} holds and this in turn follows from \eqref{massmult}.
 Indeed, take $\{f_i\}_{i=1,\dots,m}\subseteq\Cb(\XX)$ with pairwise disjoint supports and $\{v_i\}_{i=1,\dots,m}\subseteq\DD^n(\XX)$. We have to show that, for every current $T$ with $\Vert T\Vert_{\AK}(\XX)\le 1$, it holds that $$T(f_1 v_1+\cdots+f_m v_m)\le \max_j \Vert f_j\Vert_\infty\max_j\Vert v_j\Vert.$$
Now, using the definition of norm on $\DD^n(\XX)$ and \eqref{massmult} we have
\begin{align*}
T\bigg(\sum_{j=1}^m  f_j v_j\bigg)&\le\sum_{j=1}^n \abs{f_j T( v_j)}\le\sum_{j=1}^m \Vert f_j T\Vert_{\AK}(\XX)\Vert v_j \Vert \\&\le \max_j \Vert v_j\Vert \sum_j \left({\abs{f_j}} \Vert T\Vert_{\AK}\right)(\XX)\le \max_j\Vert v_j\Vert \max_j \Vert f_j\Vert_\infty\Vert T\Vert_{\AK}(\XX),
\end{align*}
so that the claim follows.

The following proposition shows how metric currents fit in the framework of local vector measures: we show that metric currents are precisely those local vector measures defined on $\DD^n(\XX)$ that satisfy the weak continuity property stated in item $ii)$ of Definition \ref{corrdefn} and moreover that the two concepts of mass (the one for metric currents defined in \cite{AmbrosioKirchheim00} and the one for local vector measures) coincide. 
\begin{prop}\label{currismvm}
Let $T$ be a current. Then $T$ is a tight element of $\DD^n(\XX)'$. In particular, $T$ induces   a unique local vector measure $\nvect_T$ defined on $\DD^n(\XX)$ such that 
\begin{equation}\label{rieszcorr}
		\nvect_T(A)(v)=T(v)\quad\text{for every }A\subseteq \XX\text{ open and }v\in\DD^n(\XX)\text{ with }\supp v\subseteq A.
\end{equation}
Moreover, it holds that $\Vert T\Vert_{\AK}=\abs{\nvect_T}$. 

Conversely, every tight element of $\DD^n(\XX)'$ satisfying item $ii)$ of Definition \ref{corrdefn} is a current; in other words, every local vector measure $\nvect$ defined on $\DD^n(\XX)$ such that $\nvect(\XX)$ satisfies item $ii)$ of Definition \ref{corrdefn} is given by a current, in the sense that \eqref{rieszcorr} holds.
\end{prop}
\begin{proof}
The fact that $T\in \DD^n(\XX)'$ follows (artificially) from the definition of norm on $\DD^n(\XX)$. Tightness is an immediate consequence of \eqref{eqmass} together with dominated convergence. Therefore we can apply Theorem \ref{weakder} and obtain a local vector measure $\nvect_T$ satisfying \eqref{rieszcorr}. 

By the very definition of norm on $\DD^n(\XX)$, for every $v\in\DD^n(\XX)$ we have that $|T(v)|\le \Vert T\Vert_{\AK}(\XX)\Vert v\Vert$ so that, using \eqref{conicidence} and \eqref{rieszcorr}, $$\abs{\nvect_T}(\XX)=\Vert T\Vert '\le \Vert T\Vert_{\AK}(\XX).$$ 
Take then an open set $A\subseteq\XX$. By \cite[Proposition 2.7]{AmbrosioKirchheim00} and the regularity of the measure $\Vert T\Vert_{\AK}$, we can show that 
\begin{equation}\label{rieszcorrproof}
\Vert T\Vert_{\AK}(A)=\sup \sum_i T(f_i \diff\pi_1^i\wedge\cdots\wedge\diff\pi_n^i)
\end{equation}
where the supremum is taken among all finite collections $\{f_i\}_i\subseteq\Cb(\XX)$ with pairwise disjoint support, with $\supp f_i\subseteq A$ and $\Vert f_i\Vert_\infty\le 1$ for every $i$, and finite families  $\{\pi_j^i\}_i\subseteq\LIP(\XX)$ of $1$-Lipschitz functions, for $j=1,\dots,n$.
Now notice that if $\pi_1,\dots,\pi_n$ are $1$-Lipschitz, then for every current $T$ it holds $$T(1 \diff\pi_1\wedge\cdots\wedge\diff\pi_n)\le \Vert T\Vert_{\AK}(\XX)$$ 
 so that \begin{equation}\label{normdiffs}
 	\Vert 1 \diff\pi_1\wedge\cdots\wedge\diff\pi_n\Vert\le 1.
 \end{equation}
We can now bound the right hand side of \eqref{rieszcorrproof} using \eqref{rieszcorr}, \eqref{compeq} and what we have noticed above, by $$\nvect_T(A)\left(\sum_i f_i \diff\pi_1^i\wedge\cdots\wedge\diff\pi_n^i\right)\le \abs{\nvect_T}(A)\Big\Vert \sum_if_i \diff\pi_1^i\wedge\cdots\wedge\diff\pi_n^i\Big\Vert\le\abs{\nvect_T}(A),$$
so that $\Vert T\Vert_{\AK}\le \abs{\nvect_T}$ as measures. Then, as we have already proved $\abs{\nvect_T}(\XX)\le\Vert T\Vert_{\AK}(\XX)$, we have $\Vert T\Vert_{\AK}= \abs{\nvect_T}$ as measures.

Finally, let $\nvect$ be a local vector measure defined on $\DD^n(\XX)$ such that $\nvect(\XX)$ satisfies item $ii)$ of Definition \ref{corrdefn}.  If $f\diff\pi_1\wedge\cdots\wedge\diff\pi_n$ is as in  item $iii)$ of Definition \ref{corrdefn}, then $S(f\diff\pi_1\wedge\cdots\wedge\diff\pi_n)=0$ for every current $S$, then $\Vert f\diff\pi_1\wedge\cdots\wedge\diff\pi_n\Vert=0$ and hence $\nvect(\XX)(f\diff\pi_1\wedge\cdots\wedge\diff\pi_n)=0$.
Let now $f\in\LIPb(\XX)$ and let $\pi_1,\dots,\pi_n\subseteq\LIP(\XX)$ be $1$-Lipschitz. Notice that the module structure of $\DD^n(\XX)$ ensures that $f\diff\pi_1\wedge\cdots\wedge\diff\pi_n=f(1\diff\pi_1\wedge\cdots\wedge\diff\pi_n)$ and thus  Lemma \ref{foutside}, Proposition \ref{absfoutside} and \eqref{normdiffs} give
\begin{align*}
	 \nvect(\XX)(f\diff\pi_1\wedge\cdots\wedge\diff\pi_n)= f\nvect(\XX)(1\diff\pi_1\wedge\cdots\wedge\diff\pi_n)\leq |f||\nvect|(\XX)= \int_\XX \abs{f}\dd{\abs{\nvect}}.
\end{align*}
This proves that item $i)$ in Definition \ref{corrdefn} holds and concludes the proof that $\nvect(\XX)$ is a current. 
\end{proof}
\begin{rem}
The push forward operator $\varphi_*$ defined in the previous section has nothing to do with the push forward for currents defined in \cite[Definition 2.4]{AmbrosioKirchheim00} (notice that in particular the latter is only defined for $\varphi$ Lipschitz as it defines the push-forward of a current by duality with the pullback of Lipschitz functions via the map $\varphi$).\fr
\end{rem}
\begin{rem}
Proposition \ref{currismvm} allows us to consider $n$-currents as local vector measures defined on $\DD^n(\XX)$. Notice that, in general, not every element of $\DD^n(\XX)$ is tight, hence not every element of $\DD^n(\XX)$ corresponds to a current. 

Moreover, not every tight functional of $\DD^n(\XX)'$ is given by a current, this is to say that there are local vector measures defined on $\DD^n(\XX)$ that are not given by currents (which is not a surprise). A counterexample to this lack of coincidence is given in Example \ref{countcurr} below, in which we exhibit a tight functional of $\DD^n(\XX)'$ that does not satisfy  axiom $ii)$ of Definition \ref{corrdefn}. Notice however that, by the proof of Proposition \ref{currismvm}, every tight functional of $\DD^n(\XX)'$  satisfies  axioms $i)$ and $iii)$ of Definition \ref{corrdefn}.
\fr
\end{rem}
\begin{ex}\label{countcurr}
Let $(\XX,\dist,\mass)\defeq([-1,1],\dist_e,\LL^1)$.
It is easy to see that for $l\in \Lpu$ the map
$$\DD^1(\XX)=\Cb(\XX)\otimes \LIP(\XX)\rightarrow\RR\quad\text{defined by}\quad f\diff g\mapsto\int_\XX l f g'\dd{\mass}$$
defines a current (see \cite[Example 3.2]{AmbrosioKirchheim00}). Thus,
if $f\diff g\in\Cb(\XX)\otimes C^1(\XX)\subseteq \DD^1(\XX) $, it holds that $\Vert f\diff g\Vert \ge \abs{f(0)g'(0)}$.
 Consider now the map $$\DD^1(\XX)\supseteq\Cb(\XX)\otimes C^1(\XX)\rightarrow\RR\quad\text{defined by}\quad f\diff g\mapsto f(0)g'(0).$$
 Using Hahn-Banach Theorem, we can extend the map above to a functional of $\DD^1(\XX)'$, that is automatically tight, as $\XX$ is compact (Remark \ref{tightrem}).
 However, we see that the map above can not be a current: indeed,  axiom $ii)$ of Definition \ref{corrdefn} is clearly violated.\fr
\end{ex}
We want to think to  the space $\DD^n(\XX)$ defined above as the space of $n$-forms on the metric space $(\XX,\dist)$ and Proposition \ref{currismvm} corroborates this  in showing that there is a natural duality between appropriate operators on $\DD^n(\XX)$ and currents. Still, in the classical smooth setting the space of differential forms has a natural algebra structure and it is natural to wonder whether the same holds in our setting. We are going to show that this is indeed the case.

Thus let  $n,m\in\NN$, $n,m\ge 1$ and notice that we have a natural exterior product operation
$$ \wedge: \DD^n(\XX)\times \DD^m(\XX)\rightarrow \DD^{n+m}(\XX)$$
defined as $$(f \diff\pi_1\wedge\cdots\wedge\diff\pi_n,g \diff\pi_{n+1}\wedge\cdots\wedge\diff\pi_{n+m})\mapsto f g \diff\pi_1\wedge\cdots\wedge\diff\pi_{n}\wedge\diff\pi_{n+1}\wedge\cdots\wedge\diff\pi_{n+m}$$
and extended by bilinearity. We also write $$v\wedge w\defeq \wedge(v,w)\quad\text{for every }v\in\DD^n(\XX), w \in \DD^m(\XX).$$

This `algebraic' structure is compatible with the norm on $\DD^n(\XX)$ imposed via `analytic' considerations:
\begin{prop}\label{propwedge}
For every $v\in\DD^n(\XX),w \in \DD^m(\XX)$, it holds that 
$$ \Vert v\wedge w\Vert\le \Vert v\Vert \Vert w\Vert.$$
\end{prop}
\begin{proof}
We define, for $k\in\NN$, $k\ge 1$, the set $B^k\subseteq\DD^k(\XX)$ as 
$$ B^k\defeq\left\{\sum_{i\in\NN} f_i \diff \pi_1^i\wedge\cdots\wedge\diff\pi^i_k:\{f_i\}_i\subseteq\Cb(\XX), \sum_{i\in\NN} \abs{f_i}\le 1,  \Lip(\pi_j^i)\le 1\right\}.$$
Clearly, $B^n\wedge B^m\subseteq B^{n+m}$, in the sense that $\wedge(B^n,B^m)\subseteq B^{n+m}$.

By \cite[Proposition 2.7]{AmbrosioKirchheim00}, it holds that for $k\in\NN$, $k\ge 1$, 
\begin{equation}\label{AKprop27}
	\Vert T\Vert_{\AK}(\XX)=\sup_{v\in B^k} T(v)\quad\text{for every $k$-current $T$.}
\end{equation}
Let now $T$ be a $(n+m)$-current. If $v\in \DD^n(\XX)$, we can define a $m$-current $T\mres v$ (see the discussion below \cite[(2.5)]{AmbrosioKirchheim00})  by
$$ T\mres v(w)\defeq T(v\wedge w)\quad\text{for every }w\in\DD^m(\XX),$$
where we notice that the following discussion implies that this definition is well posed even after taking the quotient on $\DD^n(\XX)$ with respect to the seminorm defined by \eqref{defnsemi}.

Using \eqref{AKprop27} repeatedly and what noticed above, we have that for $v\in\DD^n(\XX)$,
\begin{align*}
\Vert T\mres v\Vert_{\AK}(\XX)&=\sup_{p\in B^m}	{T\mres v(p)}=\sup_{p\in B^m}{T(v\wedge p)}=\sup_{p\in B^m}{T(p\wedge v)}\\&= \sup_{p\in B^m}{ T\mres p(v)}\le \Vert v\Vert\sup_{p\in B^m}\Vert{T\mres p}\Vert_{\AK}(\XX)\\&=\Vert v\Vert
\sup_{p\in B^m}\sup_{q\in B^n} T\mres p(q)=\Vert v\Vert\sup_{p\in B^m}\sup_{q\in B^n} T (p\wedge q)\\&\le \Vert v\Vert\sup_{r\in B^{n+m}} T(r)=\Vert v\Vert \Vert T\Vert_{\AK}(\XX). 
\end{align*}
Now, for $v,w$ as in the statement,
$$ \Vert v\wedge w\Vert=\sup_T T(v\wedge w)=\sup_T (T\mres v) (w)\le \sup_T \Vert T\mres v\Vert_{\AK}(\XX) \Vert w\Vert,$$
where the suprema are taken among all $(n+m)$-currents $T$ with $\Vert T\Vert_{\AK}(\XX)\le 1$.
Together with what just remarked, this concludes the proof.
\end{proof}
\begin{rem}
	Notice that we have a natural surjective linear map
	$$
	\bigwedge\nolimits^n   \DD^1(\XX)\rightarrow \DD^n(\XX)
	$$
	where the domain has to be seen as algebraic wedge product. Moreover, if we endow the domain with the projective norm i.e.\ 
	$$
	\bigwedge\nolimits^n \DD^1(\XX)  \ni v\mapsto \Vert v\Vert\defeq\inf\left\{ \sum_{i=1}^k \Vert v_1^i\Vert_{\DD^1(\XX)} \cdots\Vert v_n^i\Vert_{\DD^1(\XX)}: v=\sum_{i=1}^k v_1^i\wedge\cdots\wedge v_n^i    \right\}
	$$
	and take then the quotient of $\bigwedge\nolimits^n   \DD^1(\XX)$, we see that this map descends to the quotient and has norm bounded by $1$ thanks to Proposition \ref{propwedge}.
	\fr
\end{rem}
\begin{rem}
Notice that in the case $n=1$, the space $\DD^1(\XX)$ can be seen as a sort of metric cotangent module. Indeed, we have a natural map $$\diff: \LIP(\XX)\rightarrow\DD^1(\XX)\quad f\mapsto 1 \diff f$$
satisfying $$\Vert \diff f\Vert\le \Lip(f)\quad\text{for every }f\in\LIP(\XX)$$
(here equality in general does not hold). Also, the Leibniz rule 
$$ \diff (f g)=f\diff g+g\diff f\quad\text{for every }f,g\in\LIPb(\XX)$$ holds by \cite[Theorem 3.5]{AmbrosioKirchheim00} and, with a similar argument, we can prove that the chain rule
$$ \diff (\phi\circ f)=\phi'\circ f\diff f\quad\text{for every }\phi\in C^1(\RR)\cap\LIP(\RR)\text{ and }f\in\LIP(\XX)$$
holds. Finally if $\{f_i\}_i$ is a sequence of equi-Lipschitz functions pointwise convergent to $f$,
then 
$$\diff f_i\stackrel{*}{\rightharpoonup} \diff f$$
thanks to the requirement $ii)$ of Definition \ref{corrdefn}, where we isometrically embedded $\DD^1(\XX)$ into the dual space of the space of $1$-currents.\fr
\end{rem}
In this section we have seen how currents on metric spaces can be seen as elements of the dual of a suitable normed $\Cb(\XX)$-module.
In literature, there have been other attempts to describe a pre-dual space of the space of currents (e.g.\ \cite{PankkaSoultanis,SchioppaCurrents,WilliamsCurrents}). We now compare briefly our approach to the one in \cite{PankkaSoultanis}. First, let us clarify that we do not exhibit a pre-dual space of the space of currents, as not every element of $\DD^n(\XX)'$ is a current, (indeed, in order to be a current, an element of $\DD^n(\XX)'$ must be tight and satisfy item $ii)$ of Definition \ref{corrdefn} - we have not been able to find a norm on $\DD^n(\XX)$ which is compatible with such notions of convergence). On the other hand in \cite{PankkaSoultanis} the space of currents is identified with the sequentially continuous dual of the space $\bar{\Gamma}_c^n(\XX)$ (see \cite{PankkaSoultanis} for the relevant definitions). 

We show now that the map $\widehat{\,\cdot\,}$ in  \cite[Theorem 1.1]{PankkaSoultanis} is compatible with the notions developed here. In order to do so, the reader is assumed to be familiar with the machinery used and developed in \cite{PankkaSoultanis}. Let $T$ be a current, which then induces a local vector measure $\nvect_T={L}_T\abs{\nvect_T}$.
Take then $\widehat{T}\in \bar{\Gamma}_c^n(\XX)^*$.
Now, the proofs of \cite[Theorem 7.1 and Theorem 1.1]{PankkaSoultanis} show that if $\omega\in\bar{\Gamma}_c^n(\XX) $ then we have
$$ \widehat{T}(\omega)\defeq \int_\XX \bar{L}_T(x)(\omega(x))\dd{\abs{\nvect_T}(x)}$$
where the measurability of the integrand is part of the statement. The map $\bar{L}_T(x)$ is obtained by first considering the unique extension of $L_T(x)$ to the space $\overline{\rm Poly}^n(U)$, where $U$ is an open neighbourhood of $x$, and then by considering the induced map on the stalk over $x$, as by weak locality $L(x)(v)$ depends only on the germ of $v$ at $x$.

\subsubsection{Differential of Sobolev functions}\label{se:diffsob}

Recall that a metric measure space is a triplet $(\XX ,\dist,\mass)$ where $\XX$ is a set, $\dist$ is a (complete and separable) distance on $\XX$ and $\mass$ is a non negative Borel measure that is finite on balls.  
We adopt the convention that metric measure spaces have full support, that is to say that for any $x\in\XX$, $r>0$, we have $\mass(B_r(x))>0$.

The Cheeger energy (see \cite{Cheeger00,Shanmugalingam00,AmbrosioGigliSavare11,AmbrosioGigliSavare11-3}) associated to a metric measure space $(\XX,\dist,\mass)$ is the convex and lower semicontinuous functional defined on $\Lpt$ as
\[
	\Ch(f)\defeq\frac{1}{2}\inf\left\{\liminf_k \int_\XX\lip(f_k)^2\dd{\mass}:f_k\in\LIPb(\XX)\cap\Lpt,\ f_k\rightarrow f \text{ in }\Lpt\right\}
\]
where $\lip(f)$ is the so called local Lipschitz constant
$$ \lip(f)(x)\defeq \limsup_{y\rightarrow x}\frac{\abs{f(y)-f(x)}}{\dist(y,x)},$$
which has to be understood to be $0$ if $x$ is an isolated point.
The finiteness domain of the Cheeger energy will be denoted as ${\rm W}^{1,2}(\XX)$ and will be endowed with the complete norm $\Vert f\Vert^2_{{\rm W}^{1,2}(\XX)}\defeq\Vert f\Vert_{\Lpt}^2+2 \Ch(f)$. 
It is possible to identify a canonical object $\abs{\diff f}\in\Lpt$, called minimal relaxed slope, providing the integral representation  \begin{equation}\label{cdsnjoa}
	\Ch(f)=\frac{1}{2}\int_\XX\abs{\diff f}^2\dd{\mass}\quad\text{for every }f\in{\rm W}^{1,2}(\XX).
\end{equation}

We assume the reader familiar with the concepts of $\Lp^\infty/\Lp^0$-normed modules as developed in \cite{Gigli14}. Here we recall that one of the main results in  \cite{Gigli14} is about existence and uniqueness of a `cotangent module' and of an associated notion of `differential of a Sobolev function', meaning that: there exists a unique (up to unique isomorphism) couple $(\cotX,\diff)$ where $\cotX$ is a $\Lp^2$-normed $\Lp^\infty$-module and $\diff:{\rm W}^{1,2}(\XX)\rightarrow\cotX$ is linear and such that 
\begin{enumerate}[label=\roman*)]
	\item $\abs{\diff f}$ (as just above \eqref{cdsnjoa}) coincides with the pointwise norm of $\diff f$ $\mass$-a.e.\ for every $f\in{\rm W}^{1,2}(\XX)$,
	\item $\cotX$ is generated (in the sense of modules) by $\left\{\diff f:f\in{\rm W}^{1,2}(\XX)\right\}$.
\end{enumerate}
We define the tangent module $\tanX$ as the dual (in the sense of modules) of $\cotX$.
We define $\cotanXzero$ as the $\Lp^0$-completion of the cotangent module $\cotX$ and also (this definition is canonically equivalent to the previous one if $p=2$)
\begin{equation}\notag
	\cotanXp\defeq\left\{v\in\cotanXzero:\abs{v}\in\Lpp\right\}\quad\text{for $p\in[1,\infty]$}.
\end{equation}
Similarly, we define $\tanXzero$ as the $\Lp^0$-completion of $\tanX$ and
\begin{equation}
\label{eq:deftp}
  \tanXp\defeq\left\{v\in\tanXzero:\abs{v}\in\Lpp\right\}\quad\text{for $p\in[1,\infty]$}.
\end{equation}

\bigskip

In this manuscript we proposed an axiomatization of the concept of module (that aims at being an abstract approach to the space of sections of a given bundle) and duality different from the one in \cite{Gigli14}. It is therefore natural to wonder whether even in this new approach we have an existence \& uniqueness result like the above. The answer is `yes under mild conditions' and is given in the theorem below. We notice that:
\begin{itemize}
\item[i)] If $\mathscr M$ is an $\Lp^2$-normed module, then the subspace $\VV:=\{v\in\mathscr M:|v|\in \Lp^\infty(\mass)\}$ equipped with the norm $\|v\|:=\||v|\|_{\Lp^\infty}$ is a normed $\Cb(\XX)$-module, the product operation being the one inherited from the $\Lp^\infty(\mass)$-module structure.
\item[ii)] Is perfectly natural to assume that the reference measure is finite, in order to have the integrability of $|\diff f|$ for every $f\in{\rm W}^{1,2}(\XX)$. The alternative would be to develop a theory for local vector measures with locally finite mass - and thus acting in duality with objects with bounded support. This is viable but we won't proceed in this direction.
\item[iii)] The situation here - in particular for what concerns uniqueness - is more complicated than the one in \cite{Gigli14} because we have to build not only the local vector measure $\DIFF f$, but also  the module $\VV$ on which it acts (as opposed to the construction of the differential in  \cite{Gigli14} that `stands on its own'). Assuming $\cotX$ to be reflexive helps in getting the desired uniqueness.
\end{itemize}
\begin{thm}\label{exunSobo}
Let $(\XX,\dist,\mass)$ be a metric measure space with finite mass and assume that $\cotX$ is reflexive. Then there exists a unique couple $(\DIFF,\VV)$, where $\DIFF:{\rm W}^{1,2}(\XX)\rightarrow\MM_\VV$ is linear and $\VV$ is a normed $\Cb(\XX)$-module such that:
\begin{enumerate}[label=\roman*)]
	\item $\abs{\DIFF f}=\abs{\diff f}\mass$ as measures,
	\item for every $v\in\VV$ we have
	\begin{equation}\label{sdcjkasd}
		\Vert v\Vert =\sup \nvect(\XX)(v),
	\end{equation}
where the supremum is taken among all local vector measures $\nvect$ belonging to the $\Cb(\XX)$ module generated by the image of $\DIFF$ with $\abs{\nvect}(\XX)\le 1$,
\item if $\{v_k\}_k\subseteq\VV$ is bounded and such that for every $A$ Borel and $f\in{\rm W}^{1,2}(\XX)$, $\DIFF f(A)(v_k)$ has a finite limit, then there exists $v\in\VV$ such that $\DIFF f(A)(v_k)\rightarrow\DIFF f(A)(v)$. 
\end{enumerate}
Uniqueness is intended up to unique isomorphism in the following sense: whenever $(\tilde{\DIFF},\tilde{\VV})$ is another such couple, there exists a unique couple of $\Cb(\XX)$ linear (bijective) isometries $(\Phi,\Psi)$ where $\Phi:\MM_\VV\rightarrow\MM_{\tilde{\VV}}$ and $\Psi: \VV\rightarrow\tilde{\VV}$ are such that $\Phi\circ{\DIFF}=\tilde{\DIFF}$ and $\Psi(v)\,\cdot\,\Phi(\nvect)=v\,\cdot\,\nvect$. 

Finally, the definitions \begin{equation}\label{defnvv}
	\VV\defeq\left\{v\in\tanX:\abs{v}\in\Lpi\right\}=\tanXinf
\end{equation} and \begin{equation}\label{defndd}
	\DIFF f(A)(\,\cdot\,)\defeq \int_A\diff f(\,\cdot\,)\dd{\mass}
\end{equation}   provide a realization of the unique couple as above.
\end{thm}
\begin{proof}
We divide the proof in three steps.
\\\textsc{Step 1}. We verify that the couple $(\DIFF,\VV)$ given by  \eqref{defnvv} and \eqref{defndd} satisfies the requirements. It is clear that $\DIFF f$ is a local vector measure for every $f\in{\rm W}^{1,2}(\XX)$ whose total variation is bounded from above by $\abs{\diff f}\mass$. The equality $\abs{\DIFF f}=\abs{\diff f}\mass$ follows from \cite[Corollary 1.2.16]{Gigli14}. Also, ${\rm W}^{1,2}(\XX)\ni f\mapsto\DIFF f\in\MM_\VV$ is linear by linearity of $f\mapsto\diff f$. Notice that \eqref{sdcjkasd} is a consequence of the density (in the sense of $\Lp^p$-normed $\Lp^\infty$-modules) of the image of the map $\diff:{\rm W}^{1,2}(\XX)\rightarrow\cotX$ together with the definition of pointwise norm for $\tanX$ (\cite[Proposition 1.2.14]{Gigli14}) and an immediate approximation argument. We prove now item $iii)$, take $\{v_k\}_k\subseteq\VV$ as in the statement and notice that since $\mass(\XX)<\infty$ such sequence if also bounded in $\tanX$. Since such space is reflexive, there is a non-relabelled subsequence weakly converging to a limit $v\in \tanX$ (e.g.\ by Eberlein-Smulian's Theorem, but in fact in our setting the reflexivity of $\cotX$ implies its separability - because it trivially implies the reflexivity of ${\rm W}^{1,2}(\XX)$, that in turn implies separability of ${\rm W}^{1,2}(\XX)$ - see  \cite[Proposition 42]{ACM14} - that in turn trivially implies the separability of $\cotX$, so there is no need of the deep Eberlein-Smulian's Theorem). Now notice that the $\tanXinf$-norm is $\tanX$-lower semicontinuous to conclude that $v\in\VV$ as well.

Now notice that for $f\in {\rm W}^{1,2}(\XX)$ and $A\subseteq\XX$ Borel we have $\chi_A\dd f\in \cotX$, hence the weak convergence implies
$$
\lim_k \DIFF f(A)(v_k)=\lim_k \int_{A}\diff f(v_k)\,\dd\mass=\int_A \diff f(v)\,\dd\mass=\DIFF f(A)(v),
$$
as desired.
\\\textsc{Step 2}. We prove that the maps $(\Phi,\Psi)$, if they exist, are unique.  Recall that we require $\Phi\circ{\DIFF}=\tilde\DIFF$ and $\Psi(v)\,\cdot\,\Phi(\nvect)=v\,\cdot\,\nvect$ for any $v\in\VV$. Then, taken $\{g_i\}_{i=1,\dots,n}\subseteq\Cb(\XX)$ and $\{f_i\}_{i=1,\dots,n}\subseteq{\rm W}^{1,2}(\XX)$ we have that for any $v\in\VV$ it holds 
$$\sum_{i=1}^n g_i\tilde{\DIFF} f_i(\XX)(\Psi (v)) =\Phi\left(\sum_{i=1}^n g_i\DIFF f_i\right)(\XX)(\Psi(v))=\sum_{i=1}^n g_i\DIFF f_i(\XX)(v),$$
which, thanks to item $ii)$, forces the uniqueness of $\Psi$.
Uniqueness of $\Phi$ follows immediately from the request $\Psi(v)\,\cdot\,\Phi(\nvect)=v\,\cdot\,\nvect$, as $\Psi$ is required to be surjective.
\\\textsc{Step 3}. We take a couple $(\DIFF,\VV)$ verifying items $i)$, $ii)$ and $iii)$ and we prove existence of the maps $(\Phi,\Psi)$ as in the statement, provided that the other couple verifying items $i)$, $ii)$ and $iii)$  is the canonical one given by \eqref{defnvv} and \eqref{defndd}. This will be clearly enough. Both maps will be denoted with $\hat\cdot$.

For every $v\in\VV$, we define
\begin{equation}\notag
	\abs{v}_*\defeq\mass-\esssup_{f\in{\rm W}^{1,2}(\XX)} L_f(v)\chi_{\{\diff f\ne 0\}},
\end{equation}
where $\DIFF f=L_f\abs{\DIFF f}$ is the polar decomposition, notice that $\abs{\,\cdot\,}_*$ is well defined as $|\DIFF f|\ll\mass$ by item $i)$.
Notice now that  item $ii)$ together with an easy approximation argument  based on  Proposition \ref{absfoutside}  yields that 
\begin{equation}\notag
	\Vert v\Vert=\sup \nvect(\XX)(v),
\end{equation}
where the supremum is taken among the local vector measures $\nvect$ with $\abs{\nvect(\XX)}\le 1$ and 
$$
\nvect\in\left\{\sum_{i} \chi_{A_i}\DIFF f_i: \{A_i\}_i \text{ is a Borel partition of $\XX$ and $\{f_i\}_i\subseteq{\rm W}^{1,2}(\XX)$}\right\}.
$$
It then follows   that $\Vert \abs{v}_*\Vert_{\Lpi}=\Vert v\Vert$. 

Given $v\in\VV$, we consider the map 
$$
\cotX\ni\sum_i \chi_{A_i}\diff f_i \mapsto \sum_i\chi_{A_i} L_{f_i}(v)\abs{\diff f_i}= \sum_i\chi_{A_i} \dv{(v\,\cdot\,\DIFF f_i)}{\mass}(v)\in\Lpu,
$$
where the equality is due to item $i)$. By the trivial bound $ \sum_i\chi_{A_i}|\diff f_i|\Vert v\Vert$ for the right hand side, the fact that $\mass$ is finite and \cite[Proposition 1.4.8]{Gigli14}, we see that it defines an element of $\tanX$, that we call $\hat{v}$, and which satisfies, by the definition of $\abs{\,\cdot\,}_*$, the identity $\abs{\hat{v}}=\abs{v}_*\ \mass$-a.e. Notice that the map $v\mapsto\hat{v}$ is $\Cb(\XX)$ linear and satisfies 
\begin{equation}
\label{eq:ugu}
\DIFF f(A)(v)=\int_A \diff f(\hat{v})\dd{\mass}\qquad\text{ or equivalently }\qquad v\cdot \DIFF f=\diff f(\hat{v})
\end{equation}
for every $A\subseteq\XX$ Borel and $v\in\VV$. Also, if $\sum_i \chi_{A_i}\diff f_i\in\cotX$, it holds that 
\begin{equation}\label{pointwiseess}
	\bigg|{\sum_i \chi_{A_i}\diff f_i}\bigg|=\mass-\esssup_{v\in\VV,\Vert v\Vert\le 1} \sum_i \chi_{A_i}\diff f_i(\hat{v}),
\end{equation}
as, if $f\in{\rm W}^{1,2}(\XX)$, $$\abs{\diff f}=\mass-\esssup_{v\in\VV,\Vert v\Vert\le 1}L_f(v)\abs{\diff f}=\mass-\esssup_{v\in\VV,\Vert{v}\Vert\le 1}\diff f(\hat{v}),$$ where, as above, the second equality comes  from  item $i)$ (or, which is the same, from \eqref{eq:ugu}).

We set now $M\defeq\left\{\hat{v}:v\in\VV\right\}$ and we claim that $M=\tanXinf$. We prove first that $M\subseteq\tanX$ is dense. If by contradiction $M$ was not dense, we could find a functional $Q\in(\tanX)^*=\cotX$ (by  \cite[Proposition 1.2.13]{Gigli14} and the assumption that $\cotX$ is reflexive) such that $Q\ne 0$ but $Q(\hat{v})=0\ \mass$-a.e.\ for every $v\in\VV$. By density, we take $\{Q_k\}_k\subseteq \cotX$ with $Q_k\rightarrow Q$ in $\cotX$ and also $\abs{Q-Q_k}\rightarrow 0\ \mass$-a.e.\ and $Q_k$ is of the form $\sum_i \chi_{A_i^k}\diff f_i^k$. Now, if $v\in\VV$,
$$\abs{Q_k(\hat{v})}\le \abs{Q_k-Q}(\hat{v})+ \abs{Q(\hat{v})}\le \abs{Q_k-Q}\Vert v\Vert \quad\mass\text{-a.e.}$$
so that, taking into account \eqref{pointwiseess}, 
$$ \abs{Q_k}\le \abs{Q-Q_k}\quad\mass\text{-a.e.}$$
which implies that $Q=0$, a contradiction. Therefore we have proved that $M\subseteq\tanX$ is dense.

Take now $v\in\tanX$ with $\abs{v}\in\Lpi$. By density, we take a sequence $\{v_n\}_n\subseteq\VV$ such that $\hat{v}_n\rightarrow v$ in $\tanX$ and also $\abs{\hat{v}_n-v}\rightarrow 0\ \mass$-a.e. As $M$ is stable under multiplication by characteristic functions of Borel subsets of $\XX$ (thanks to item $iii)$ and the $\Cb(\XX)$ linearity of the map $v\mapsto\hat{v}$) we can further assume that $\{v_n\}_n\subseteq\VV$ is bounded. Now, thanks to item $iii)$, we see that $v\in M$.

Now, to any $\nvect\in\MM_\VV$ we associate $\hat{\nvect}\in\MM_M$ by $\hat{v}\,\cdot\,\hat{\nvect}\defeq v\,\cdot\, \nvect$. As $\VV$ is isometric to $$\left(M,\Vert\abs{\,\cdot\,}\Vert_{\Lpi}\right)$$ (via the $\Cb(\XX)$ linear isometry $\hat\cdot$), we see that the map $\nvect\mapsto\hat{\nvect}$ is a $\Cb(\XX)$ linear isometry. Also, $\hat{\DIFF f}(A)=\int_A\diff f(\,\cdot\,)\dd{\mass}$. Indeed, if $v\in\VV$, then $\hat{v}\,\cdot\,\hat{\DIFF f}=v\,\cdot\,\DIFF f=L_f(v)\abs{\diff f}=\diff f(\hat{v})$.
\end{proof}
\begin{rem}
Theorem \ref{exunSobo} can be easily adapted to integrability exponents different from $2$ within the range $(1,\infty)$.
\fr
\end{rem}

\subsubsection{Differential of BV functions}\label{sectBV}
In this section, we build local vector measures that describe the distributional derivatives of functions of bounded variation. We study here the case of an arbitrary metric measure space and a real valued function of bounded variation. Then, in the setting of an $\RCD(K,\infty)$ space, improve considerably the result, see Section  \ref{sectBVRCD}.

\bigskip

We assume that the reader is familiar with the theory of functions of bounded variation in metric measure spaces developed in \cite{amb00,amb01,MIRANDA2003}.
We recall now the main notions. For   $A\subseteq\XX$ open, $\LIPloc(A)$  denotes the space of Borel functions that are Lipschitz in a neighbourhood of $x$, for any $x\in A$. If $(\XX,\dist)$ is locally compact, $\LIPloc(A)$ coincides with the space of functions that are Lipschitz on compact subsets of $A$.

Fix a metric measure space $(\XX,\dist,\mass)$. 
Given $f\in\Lpu$, we define, for any $A\subseteq\XX$ open,
\[
	\abs{\DIFF f}(A)\defeq\inf\left\{\liminf_k \int_\XX\lip(f_k)\dd{\mass} :f_k\in\LIPloc(A)\cap\Lpu,\ f_k\rightarrow f \text{ in }\Lpu\right\}.
\]
We say that $f$ is a function of bounded variation, i.e.\ $f\in\BVv$, if $f\in\Lpu$ and $\abs{\DIFF f}(\XX)<\infty$. 
If this is the case, $\abs{\DIFF f}(\,\cdot\,)$ turns out to be the restriction to open sets of a finite Borel measure  that we denote with the same symbol and we call total variation.
Notice that, by its very definition, the total variation $\abs{\DIFF f}(A)$ is lower semicontinuous with respect to $\Lpu$ convergence for $A$ open, is subadditive and $\abs{\DIFF (\phi \circ f)}\le L \abs{\DIFF f}$ whenever $f\in\BVv$ and $\phi$ is $L$-Lipschitz.

Several classical results concerning BV calculus have been generalized to the abstract framework of metric measure spaces.
Among them, the Fleming-Rishel coarea formula, which states that given $f\in\BVv$, the set $\{f>r\}$ has finite perimeter for $\LL^1$-a.e.\ $r\in\RR$ and
\[
	\int_\XX h\dd{\abs{\DIFF f}}=\int_\RR\dd{r} \int_\XX h\dd{\per(\{f>r\},\,\cdot\,)}\quad\text {for any Borel function $h:\XX\rightarrow[0,\infty].$}
\]
In particular,
\begin{equation}
	\label{coareaeqdiff}
	{\abs{\DIFF f}}(A)=\int_\RR\dd{r}  {\per(\{f>r\},A)}\quad\text {for any $A\subseteq \XX$ Borel}.
\end{equation}

Now we need the definition of divergence (\cite{Gigli14,buffa2020bv}). Notice that in the definition below the module $\mathrm{L}^\infty(T\XX)$ is defined as in \eqref{eq:deftp}, i.e.\ starting from the modules $\cotX$, $\tanX$ and algebraic operations; in particular, no notion of Sobolev function other than ${\rm W}^{1,2}(\XX)$ is required.
\begin{defn}\label{divedefn}
	Let $p\in\{2,\infty\}$. For $v\in\mathrm{L}^p(T\XX)$ we say that $v\in D(\dive^p)$ if there exists a function $g \in\Lpp$ such that 
	\begin{equation}\label{divedefneq}
		\int_\XX \diff f(v)\dd{\mass}=-\int_\XX  f g \dd{\mass}\quad\text{for every }f\in{\rm W}^{1,2}(\XX)\text{ with bounded support},
	\end{equation}
	and such $g$, that is uniquely determined, will be denoted by $\dive\, v$.
\end{defn}
Notice that if $v\in D(\dive^2)\cap D(\dive^\infty)$, then the the two objects $\dive\,v$ as above coincide, in particular, $\dive\,v\in\Lpt\cap\Lpi$.
From \eqref{divedefneq} it follows that $\supp(\dive\,v)\subseteq C$ for any $C\subseteq\XX$ such that $\supp v\subseteq C$.

Another direct consequence of the definition is that if $v\in D(\dive^p)$ has bounded support (i.e.\ support contained in a bounded set) then
\begin{equation}
\label{eq:intdiv}
\int_\XX \dive\,v\,\dd\mass=0
\end{equation}
as it can be checked by picking $f$ in \eqref{divedefneq} identically equal to 1 on a set containing the support of $v$. Also, the following version of the Leibniz rule holds: if $v\in D(\dive^\infty)$ and $f\in\LIPb(\XX)$, then $f v\in D(\dive^\infty)$ and
\begin{equation}\label{calcdiveeq}
	\dive(f v)=\diff f(v)+f\dive\, v.
\end{equation}  This follows from \eqref{divedefneq} and the fact that if $g\in{\rm W}^{1,2}(\XX)$ has bounded support and $f\in\LIPb(\XX)$, then $f g\in{\rm W}^{1,2}(\XX)$ has bounded support and satisfies $\diff(f g)=f\diff g+g\diff f$.
In the case $p=2$, again from the algebra properties of bounded Sobolev functions together with an easy approximation argument, we have that if $v\in D(\dive^2)\cap\tanXinf$ and $f\in\Ss\cap\Lpi$, then $f v\in D(\dive^2)$ and the calculus rule above holds. In the case $p=2$, we will often omit to write the superscript $2$ for what concerns the divergence.

\bigskip


The following representation formula is  basically proved in  \cite{DiMarino14} (see also \cite{buffa2020bv} and \cite[Proposition 2.1]{BGBV} for what concerns this formulation).
\begin{prop}[Representation formula]\label{reprfordiffregularpre2}

	Let $(\XX,\dist,\mass)$ be a metric measure space and $f\in\BVv$. Then, for every $A$ open subset of $\XX$, it holds that
	\begin{equation}\label{intagainst}
		\abs{\DIFF f}(A)=\sup\left\{\int_A f \dive\, v\dd{\mass}\right\},
	\end{equation}
	where the supremum is taken among all $v\in\mathcal{W}_A$, where
	\begin{equation}\label{qualivect1}
		\mathcal{{W}}_A\defeq\left\{v\in D(\dive^\infty):\abs{v}\le 1\ \mass\text{-a.e.\ }\supp v\subseteq A\right\}.
	\end{equation}
\end{prop}
This statement might appear surprising because it characterizes $\BVv$ functions via duality with vector fields that, in turn, are defined in duality with functions in ${\rm W}^{1,2}(\XX)$ (as discussed before Definition \ref{divedefn}). We thus make the following observations:
\begin{itemize}
\item[i)] Approaching Sobolev/BV functions via integration by parts in general metric measure setting has been one of the main achievements in  \cite{DiMarino14}. In such reference, the definition is given in duality with the notion of \emph{derivation} which is there defined as suitable map from Lipschitz functions to $\Lp^0(\mass)$.
\item[ii)] Since Lipschitz functions are always Sobolev, at least locally, vector fields as considered in \eqref{qualivect1} are included in the class of derivations as used in \cite{DiMarino14} to define BV functions. In particular, it is obvious a priori from the definitions in \cite{DiMarino14} that the inequality $\geq$ holds in \eqref{intagainst}.
\item[iii)] The opposite inequality follows from the results in \cite{DiMarino14}. Specifically, it is trivial to notice that `$\Lp^\infty$ derivations with divergence in $\Lp^\infty$' are also `$\Lp^2$ derivations with divergence in $\Lp^2$' (at least locally) and these latter ones   can be used - thanks to \cite{DiMarino14} - to define ${\rm W}^{1,2}$ functions.  It then follows by abstract machinery that these sort of $\Lp^2$ derivation are (or better, uniquely induce) a vector field in $\tanX$ and if we actually start with an $\Lp^\infty$ derivations with divergence in $\Lp^\infty$, the corresponding vector field will be in $\mathrm{L}^\infty(T\XX)\cap  D(\dive^\infty)$ with the same pointwise norm and divergence of the original derivation (see also \cite[Lemma 3.12]{buffa2020bv}). This line of though gives $\leq$ in \eqref{intagainst}.
\end{itemize}

With this said, we have the following result:
\begin{thm}\label{difffloccaomp}
	Let $(\XX,\dist,\mass)$ be a metric measure space and let $\VV$ be the subspace of $\tanXinf$ made of $\Cb(\XX)$-linear combinations of vector fields in $D( \dive^\infty)$.
	
	Then for every $f\in \BVv$ there exists a unique local vector measure $\DIFF f$ defined on $ \VV$ such that 
	\begin{equation}\label{difffloccaompeq}
		\DIFF f(\XX)(v)=-\int_\XX f\dive \,v\,\dd\mass\qquad\text{for every }v\in D(\dive^\infty).
	\end{equation}
\end{thm}

\begin{proof} Start noticing that \eqref{calcdiveeq} grants that $\dive^\infty$ is a normed $\LIPb(\XX)$-module (we equip $\dive^\infty$ with the norm of $\tanXinf$) and that $\LIPb(\XX)$ is a subring of $\Cb(\XX)$ that approximates open sets in the sense of Definition \ref{def:Rap}. Now define $\FF:D(\dive^\infty)\rightarrow\RR$ as 
\[		\FF(v)\defeq -\int_\XX f\dive\,v\dd{\mass}.
\]
	Notice that Proposition \ref{reprfordiffregularpre2} shows that $\FF\in(D(\dive^\infty))'$ with $\Vert \mathcal F\Vert'=\abs{\DIFF f}(\XX)$ and that $$\sup
	\left\{\FF(v):v\in D(\dive^\infty),\ \Vert v\Vert \le 1,\ \supp v\subseteq A\right\} =\abs{\DIFF f}(A).$$
	In particular, the set function defined by the supremum in the left hand side of the equation above
	is the restriction to open sets of a finite Borel measure, so that by Lemma \ref{bbtightiffmeas} the functional $\FF$ is tight. 
	
	Therefore by Theorem \ref{weakder} there is a unique local vector measure $\DIFF f$ defined on $\dive^\infty$ for which \eqref{difffloccaompeq} holds and by Proposition \ref{lipbscb} such measure can be uniquely extended to a local vector measure on $\VV$.
\end{proof}
\begin{rem}
	Given $f\in\BVv$, we can take the polar decomposition of its distributional derivative $\DIFF f=L\abs{\DIFF f}$ given by Theorem \ref{difffloccaomp}. Therefore, taking into account also Lemma \ref{foutside}, we have that
	for every $g\in\Cb(\XX)$ and $v\in D(\dive^\infty)$ such that $g v\in D(\dive^\infty)$, we have
\[
		\int_\XX f\dive(g v)\dd{\mass}=-\int_\XX g L(v)\dd{\abs{\DIFF f}}
\]
	where, in particular, $\Vert L(v)\Vert_{\Lp^\infty(\abs{\DIFF f})}\le \Vert v\Vert$ by \eqref{inequalitygerm}.
	We can see this result as a particular case of \cite[Theorem 4.13]{buffa2020bv}. Following similar arguments it also possible to obtain the full result of \cite[Theorem 4.13]{buffa2020bv}, working with local vector measures defined on the $\Cb(\XX)$ module generated by $\mathcal{D M^\infty}(\XX)$ (\cite[Definition 4.1]{buffa2020bv}).\fr
	
	%
	%
	%
	
\end{rem}
\bigskip
We prove now the basic calculus rules for \textbf{continuous} functions of bounded variation. In the framework of $\RCD(K,N)$ spaces, we will have a much more powerful result, see Theorem \ref{volprop}.
\begin{prop}[Chain rule]\label{propchaincont}
Let $(\XX,\dist,\mass)$ be a metric measure space, let $f\in\BVv\cap C(\XX)$ and  let $\phi\in\LIP(\RR)$ be such that $\phi(0)=0$. Then \begin{equation}\label{abspushf}
	|\DIFF f|(f^{-1}(N))=0\quad\text{for every Borel set $N\subseteq\RR$ such that $\LL^1(N)=0$}.
\end{equation} In particular, $\phi$ is differentiable at $f(x)$ for $\abs{\DIFF f}$-a.e.\ $x$.
Moreover $\phi\circ f\in\BVv$ and
\begin{equation}\label{chaincont}
	\DIFF (\phi\circ f)=\phi'\circ f\DIFF f.
\end{equation}

\end{prop}

\begin{proof}
Take $N\subseteq\RR$ with $\LL^1(N)=0$. Then we use \eqref{coareaeqdiff} and the fact that the perimeter of a set is concentrated on its topological boundary to compute 
\begin{equation*}
		\abs{\DIFF f} (f^{-1}(N))=\int_\RR \per({\{f>t\}},f^{-1}(N)) \dd{t} =\int_N \per({\{f>t\}},f^{-1}(N)) \dd{t}=0.
\end{equation*}
In particular, by Rademacher's Theorem, we have that $\phi$ is differentiable at $f(x)$ for $\abs{\DIFF f}$-a.e.\ $x$.
	
With an easy approximation argument, we see that we can assume $\phi\in\LIP(\RR)\cap C^1(\RR)$ with $\phi(0)=0$. Indeed, let $\{\rho_n\}_n$ be a family of Friedrich mollifiers and define $$\phi_n\defeq \phi\ast\rho_n-  (\phi\ast\rho_n)(0)\in\LIP(\RR)\cap C^1(\RR).$$ For any $v\in D(\dive^\infty)$, we have on the one hand
$$
\DIFF(\phi_n\circ f)(\XX)(v)=-\int_\XX \phi_n\circ f \dive\,v\dd{\mass}\rightarrow -\int_\XX \phi\circ f \dive\,v\dd{\mass}=\DIFF(\phi \circ f)(\XX)(v)
$$
and on the other hand 
$$ \phi_n'\circ f\DIFF f(\XX)(v)=\int_\XX \phi_n'\circ f L_f(v)\abs{\DIFF f}\rightarrow \int_\XX \phi'\circ f L_f(v)\abs{\DIFF f}= \phi'\circ f\DIFF f(\XX)(v),$$
where we used that $\phi_n'\rightarrow\phi\ \LL^1$-a.e.\ so that by \eqref{abspushf} it holds $\phi_n'\circ f\rightarrow\phi\circ f\ \abs{\DIFF f}$-a.e.
Then, if the chain rule holds for $\phi_n$, by the characterization of the differential given in Theorem \ref{difffloccaomp} above we obtain that it holds for $\phi$.

We thus proved that it suffices to prove the claim under the assumption $\phi\in \LIP(\RR)\cap C^1(\RR)$ with $\phi(0)=0$. If this is the case, we can take an approximating sequence $\{\tilde{\phi}_n\}_n$ as follows: for every $n$, $\tilde{\phi}_n$ is piecewise affine, at its points of non-differentiability $\tilde{\phi}_n$ coincides with $\phi$, $\tilde{\phi}_n\rightarrow\phi$ uniformly and $|{\tilde{\phi}_n'-\phi'}|\rightarrow 0\ \LL^1$-a.e. Let now $\{\hat{\phi}_n\}_n$ be defined as $\hat{\phi}_n\defeq\tilde{\phi}_n-\tilde{\phi}_n(0)$. Arguing as above, we see that it suffices to check that \eqref{chaincont} holds for any $\hat{\phi}_n$ to conclude the proof.

To conclude then we prove the chain rule under the assumption that $\phi\in\LIP(\RR)$ is piecewise affine and $\phi(0)=0$. Thus let $\{A_i\}_i\subseteq\XX$ be the at most countable collection of open sets of the form $f^{-1}(I)$ for $I\subseteq\RR$ interval where $\phi$ is affine. Then \eqref{abspushf} ensures that $|\DIFF f|(\XX\setminus\cup_iA_i)=0$ and then an argument based on the fact that $|\DIFF f|(\XX)<\infty$ shows that $\chi_{\cup_{i<n}A_i}\DIFF (\phi\circ f)(\XX)(v)\to \DIFF (\phi\circ f)(\XX)(v)$ and $\chi_{\cup_{i<n}A_i}\phi'\circ f\DIFF f (\XX)(v)\to\phi'\circ f \DIFF f (\XX)(v)$ as $n\to\infty$ for any $v\in D(\dive^\infty)$.

Hence to conclude it is  enough to check that  $\chi_{A_i}\DIFF(\phi\circ f)=\chi_{A_i}\phi'\circ f\DIFF f$ for any $i$, and then again by an argument based on $|\DIFF f|(\XX)<\infty$  that it is sufficient to prove that $\chi_{B}\DIFF(\phi\circ f)=\chi_{B}\phi'\circ f\DIFF f$ for any bounded open set $B$ contained in some of the $A_i$'s. In turn, \eqref{eq:locmass} (applied with $\VV:=D(\dive)$ but then we use Proposition \ref{lipbscb}) and \eqref{eq:defbart}  show that to prove this latter statement it is sufficient to prove that
\[
\DIFF(\phi\circ f)(B)(v)=\phi'\circ f\DIFF f(B)(v)
\]
for any $B$ as before and  $v\in D(\dive^\infty)$ with $\supp v\subseteq B$. To see this notice that 
\[
\begin{split}
	\DIFF(\phi\circ f) (B)(v)&=\DIFF(\phi\circ f)(\XX)(v)=-\int_\XX \phi\circ f\dive\,v\dd{\mass}\stackrel{\eqref{eq:intdiv}}=-\phi'_{| I}\int_\XX f\dive\,v\dd{\mass}=\phi'\circ f\DIFF f (B)(v).
\end{split}
\]
The conclusion follows.
\end{proof}
The Leibniz rule is simply obtained by polarization of the chain rule with $\phi(t)=t^2$.
\begin{prop}[Leibniz rule]\label{leibnizcontprop}
Let $(\XX,\dist,\mass)$ be a metric measure space and $f,g\in\BVv\cap \Cb(\XX)$. Then $f g \in\BVv$ and
\[
	\DIFF (f g)=f\DIFF f+g\DIFF f.
\]
In particular, 
\begin{equation}\label{csdcasdc}
	 \abs{\DIFF(f g)}\le \abs{f}\abs{\DIFF g}+\abs{g}\abs{\DIFF f}.
\end{equation}
\end{prop}
\begin{proof}
Using the chain rule with $\phi\in\LIP(\RR)$ that coincides with $t\mapsto t^2$ on a sufficiently large neighbourhood of $0$, we see that 
\[
\begin{split}
\DIFF (f+g)^2&=2(f+g)\DIFF (f+g),\\
\DIFF f^2&=2 f\DIFF f,\\
\DIFF g^2&=2 g\DIFF g.
\end{split}
\]
The conclusion easily follows from the linearity of the differential.
\end{proof}
\begin{rem}
We wish to point out that the language of local vector measures is not necessary to achieve the inequality \eqref{csdcasdc}. We sketch an alternative proof.
Take first $\phi\in\LIP(\RR)$ bi-Lipschitz (hence strictly monotone, say strictly increasing) and assume that $\phi(0)=0$. For $A\subseteq\XX$ open, we compute, by \eqref{coareaeqdiff} and the change of variables formula,
\begin{align*}
|\DIFF(\phi\circ f)|(A)&=\int_\RR \per(\{\phi\circ f>t,A\})\dd{t}=\int_\RR \per(\{ f>s,A\})\phi'(s)\dd{s}\\&=\int_\RR \int_A \phi'(f(x))\dd\per(\{ f>s,A\})(x)\dd{s}=\int_A  \phi'(f(x))\dd{\abs{\DIFF f}},
\end{align*}
so that\begin{equation}\label{cdsvdscas}
	 \abs{\DIFF (\phi\circ f)}\le \abs{\phi'\circ f}\abs{\DIFF f}.
\end{equation}

Now we notice that, using \eqref{coareaeqdiff}, \eqref{abspushf} and the regularity of the measures involved, we see that it is enough to check \eqref{csdcasdc}
on $A$, where $A\subseteq\XX$ is a bounded open set such that $f,g\in( c,C)$ for some $c,C\in(0,\infty)$. Up to scaling, we can assume $c=2$.
We compute then, on $A$,
\begin{align*}
	\abs{\DIFF (f g)}=|{\DIFF e^{\log(f g)}}|\le e^{\log (f g)}\abs{\DIFF (\log f+\log g)}\le  fg\abs{\DIFF (\log f)}+ fg\abs{\DIFF (\log g)}\le g\abs{\DIFF f}+f\abs{\DIFF g},
\end{align*}
where we used \eqref{cdsvdscas} twice.

If $f,g\in\BVv\cap\Lpi$ are not continuous, other versions of the inequality investigated are available: if one denotes with $\bar{f},\bar{g}$ the precise representatives of $f,g$ (e.g.\ \eqref{veeandwedge} and the equation below \eqref{veeandwedge}) a reasonable claim would be
$$
\abs{\DIFF (fg)}\le |{\bar{f}}|\abs{\DIFF g}+\abs{\bar{g}}\abs{\DIFF f}
$$
which is exactly what one obtains in the smooth context.
On metric measure spaces this property may fail (see e.g.\ \cite{lahti2018sharp}, where also an optimal bound on $\abs{\DIFF (f g)}$ was provided for $\PI$ spaces), whereas for finite dimensional $\RCD$ spaces the sharp version has been proved in \cite{BGBV}.
\end{rem}

\subsubsection{Strongly local measures}\label{normedmodules} 

Let $(\XX,\dist,\mass)$ be a metric measure space (complete, separable, with measure finite on bounded sets), $\mathscr M$ an $\Lp^p(\mass)$-normed $\Lp^\infty(\mass)$-module over it and $\VV\subseteq{\mathscr M}$ be a $\Cb(\XX)$ submodule (in the algebraic sense) such that for every $v\in\VV$, $\abs{v}\in\Lp^\infty(\mass)$. We have already noticed that  $\VV$ equipped with the norm
$$
\Vert v\Vert\defeq\Vert \abs{v}\Vert_{\Lp^\infty(\mass)}.
$$
is $(\VV,\Vert\,\cdot\,\Vert)$ is a normed $\Cb(\XX)$-module. Local vector measures $\nvect$ defined on $\VV$ are, by definition, weakly local, i.e.\ they satisfy
\begin{equation}
\label{eq:wl}
\nvect(A)(v)=0,\qquad\text{ for every $A\subseteq\XX$ open and $v\in\VV$ with $\|v\|_{|A}=0$.}
\end{equation}
In some sense, due to the nature of the definition of general normed $\Cb(\XX)$-modules, this is the most we ask for when speaking about locality. In the current setting,  however, the  elements of $\VV$ are also elements of $\mathscr M$ and thus are `defined $\mass$-a.e.', in a sense (see also discussion in \cite{Gigli14}). In practice, not only we can say whether $\|v\|_{|A}=0$ for any open set $A$, but we can also ask whether $|v|=0$ $\mass$-a.e.\ on $B$ for $B\subseteq\XX$ Borel (and this certainly occurs if $B$ is open and $\|v\|_{|B}=0$). Because of this, we ask whether a given local vector measure $\nvect$ is local in the following sense, that we shall call \emph{strong locality}:
\begin{equation}
\label{eq:sl}
\nvect(B)(v)=0,\qquad\text{ for every $B\subseteq\XX$ Borel and $v\in\VV$ with $|v|=0$ $\mass$-a.e.\ on $B$.}
\end{equation}
What just said ensures that \eqref{eq:sl} implies \eqref{eq:wl}. There are two reasons for which it might happen that the converse implication fails:
\begin{itemize}
\item[1)] It might be that $|\nvect|\not\ll\mass$. In this case picking $B$ with $\mass(B)=0$ and $|\nvect|(B)>0$ we see that \eqref{eq:sl} cannot hold.
\item[2)] It might be that $|\nvect|\ll\mass$ but still  \eqref{eq:sl} fails, so that we can't improve the locality information from open sets to Borel ones. In investigating this matter it might be worth to notice that the germ seminorm $|v|_g$ coincides with the $\mass$-essential upper semicontinuous envelope of the pointwise norm $|v|$ (because $\Vert v\Vert_{| A}=\Vert \abs{v}\chi_A\Vert_{\Lp^\infty(\mu)}$ for any $A\subseteq\XX$ open).
\end{itemize}
Using Hahn-Banach on $\Lp^\infty$ it is easy to build examples where these can actually occur:

\begin{ex}
Let $(\XX,\dist,\mass)$ be the unit interval $[0,1]$ equipped with the usual distance and measure and $\VV:=\Lp^\infty(\mass)$. Also, let $V\subseteq\VV$ be the subspace of those functions that are $\mathcal L^1$-a.e.\ constant in a neighbourhood of 0 and $L:V\to\RR$ be the functional assigning to $f\in V$ the value it a.e.\ assumes in such neighbourhood. Then clearly $L$ has norm 1 and can be extended, via Hahn-Banach, to a functional with norm 1 on $\VV$, still denoted $L$. 

Now we define $\nvect:=L\delta_0$, i.e.\ we put $\nvect(B)(f):=\delta_0(B)L(f)$ for every $B\subseteq [0,1]$ Borel and $f\in \VV$. It is clear that $\nvect$ is a vector valued measure on $\VV'$; to check weak locality we notice that for $A\subseteq[0,1]$ open we have $\|f\|_{|A}=0$ if and only if $f=0$ $\mathcal L^1$-a.e.\ on $A$, whence the conclusion follows from the very definition of $L$. Since, rather clearly, we have $|\nvect|=\delta_0$, we have an example where $(1)$ above holds.

A variation of this construction also gives an example where $(2)$ holds. Namely, let  $(\XX,\dist,\mass)$ be the unit interval $[0,1]$ equipped with the usual distance and the measure $\mass:=\delta_0+\mathcal L^1$ and let $\VV:=\Lp^\infty(\mass)$. Also, let $V\subseteq\VV$ be the subspace of those functions that are $\mathcal L^1$-a.e.\ constant in a neighbourhood of 0 and $L:V\to\RR$ be the functional assigning to $f\in V$ the value it a.e.\ assumes in such neighbourhood (notice that $L(f)$ might be different from $f(0)$). Then, as before,  $L$ has norm 1 and can be extended, via Hahn-Banach, to a functional with norm 1 on $\VV$, still denoted $L$. 

As before, we define $\nvect:=L\delta_0$ and notice that the same arguments as above ensure that $\nvect$ is a local vector measure defined on $\VV$ with $|\nvect|=\delta_0\ll\mass$. To see that \eqref{eq:sl} fails let $f\in \VV$ be identically 1 on $(0,1]$ and $f(0)=0$. Then the pointwise norm of $f$ in 0 is $0$ (notice that, as already discussed, the pointwise norm is not the same as the germ seminorm)  so that if \eqref{eq:sl} is in place we should have $\nvect(\{0\})(f)=0$, but the fact that $f$ is in $V$ gives $L(f)=1$, so that  $\nvect(\{0\})(f)=1$.
\fr
\end{ex}
With this said, our main result in this section, namely Proposition \ref{polreprmodules} below, concerns characterization of strongly local  vector measures and extension of such measures initially defined only on appropriate subspaces, i.e.\ we are going to adapt Proposition \ref{lipbscb} to the strongly local case.

Before coming to the actual statement, let us point out an easy  to spot  class of strongly   local vector measures. Take $M\in\mathscr M^*$ (the dual in the sense of modules) with $|M|\in \Lp^\infty(\mass)$ and define $\nvect$ as
\begin{equation}
\label{eq:reprm}
\nvect(A)(v)\defeq\int_A M(v)\dd\mass\quad\text{for $A\subseteq\XX$ Borel and $v\in\VV$,}
\end{equation}
where $\VV \defeq\{v\in\mathscr M: |v|\in\Lpi\}$.
It is then easy to see  that $\nvect$ is a local vector measure satisfying \eqref{eq:sl} and that $|\nvect|=|M|\mass$ (one would also like to say that the polar decomposition $\nvect=L|\nvect|$ of $\nvect$ is given by $\nvect=\frac{M}{|M|}|\nvect|$ but this requires a bit of technical care because $M$ is not a map from $\XX$ to $\VV'$ but rather a `local' map from $\VV$ to $\Lp^1(\mass)$). There are strict  links between the two notions - relying on Corollary \ref{thmapp} - but we won't discuss this topic further and rather refer to the upcoming \cite{GLP22}]). One of the conclusions of Proposition \ref{polreprmodules} below is that - perhaps not surprisingly - in fact all strongly local  vector measures are of this form.

A more interesting question concerns the possibility of extending a strongly local  vector measure that initially is defined only on some normed $\mathcal R$-module $\WW$ dense in $\VV$ in a suitable sense.
We point out that for $v\in\VV$ ($\VV$ seen as a normed $\Cb(\XX)$-module) or $v\in\WW$ ($\WW$ seen as a normed $\Rr$-module) it holds$$\Vert v\Vert_{|A}=\Vert \chi_A|v|\Vert_{\Lp^\infty(\mass)}\qquad\text{for every $A\subseteq\XX$ open},$$
in particular the local seminorm $\Vert\,\cdot\,\Vert_{|A}$ (and hence the notion of germ seminorm and support) is independent of the choice of $\Rr$.

We have seen in Proposition \ref{lipbscb} that  all (weakly) local vector measure admit a unique extension, thus uniqueness is also in place in the strongly local case. We have not been able to achieve an equally general conclusion for what concerns existence, nor to find counterexamples; this is the same as to say that we don't know whether the extension of a strongly local vector measure given by Proposition \ref{lipbscb} is still strongly local. Still, we identified a sufficient condition on $\WW$ for this to hold: it amounts at asking that
\begin{equation}\label{stabilityfrac}
	\frac{1}{1\vee\abs{v}}v\in\WW\quad\text{for every }v\in\WW.
\end{equation}
A typical example of a situation when this happens is for $\WW=\Lp^\infty(\XX)\cap {\rm W}^{1,2}(\XX)$ (in our applications in the $\RCD$ setting we will pick a space of bounded Sobolev vector fields, see Section \ref{se:RCD}). Notice also that the map $v\mapsto \frac{1}{1\vee\abs{v}}v$ is a sort of `truncation' operation, as it leaves $v$ unchanged on $\{|v|\leq 1\}$ and normalizes it to $\frac{v}{|v|}$ on $\{|v|>1\}$.

Finally, we notice that once a representation like \eqref{eq:reprm} holds, one can easily extend the measure from the completion of $\Cb(\XX)\cdot\WW$ to $\{v\in\mathscr M:|v|\in\Lpi\}$: we shall use this observation in identifying `polar' and `representable' measures in Section \ref{se:RCD}, see Proposition \ref{equivrepr}.

\bigskip

With this said, our main result here is:
\begin{prop}\label{polreprmodules} Let $(\XX,\dist,\mass)$ be a metric measure space, $\mathscr M$ an $\Lp^p(\mass)$-normed $\Lp^\infty(\mass)$-module and $\VV\subseteq\mathscr M$ the normed $\Cb(\XX)$-module made of elements of $\mathscr M$ with pointwise norm in $\Lp^\infty(\mass)$. Also, let $\WW\subseteq \VV$ be a subspace that, with the inherited structure, is also  normed $\mathcal R$-module for some subring $\mathcal R\subseteq\Cb(\XX)$ that approximates open sets (Definition \ref{def:Rap}).

Assume also that $\WW$ satisfies \eqref{stabilityfrac} and that $\WW$ generates, in the sense of modules, $\mathscr M$.

Let $\nvect$ be a local vector measure defined on $\WW$ such that $|\nvect|\ll\mass$. Then the following assertions are equivalent:
\begin{itemize}
	\item[i)] $\nvect$ is strongly local on $\WW$, i.e.\ for any $v\in\WW$ and $B\subseteq\XX$ Borel we have
\[
		\nvect(B)(v)=0\quad\text{whenever }\abs{v}=0\ \abs{\nvect}\text{-a.e.\ on $B$}.
\]
		\item[ii)] there exists $M_\nvect$ in ${\mathscr M}^*$ (the dual in the sense of modules) such that $\abs{M_\nvect}=1\ \abs{\nvect}$-a.e.\ and for every $v\in\WW$, it holds
\begin{equation}\label{repr0}
		 L_\nvect(x)(v)=M_\nvect(v)(x)\quad\text{for $\abs{\nvect}$-a.e.\ $x\in\XX$},
\end{equation}
\end{itemize}
Moreover, if these holds the formula
\begin{equation}
\label{eq:extsl}
\hat \nvect(B)(v):=\int_BM_\nvect(v)\,\dd|\nvect|,\qquad\forall B\subseteq\XX,\ Borel,\ v\in\VV,
\end{equation}
provides the unique extension of $\nvect$ to a strongly  local vector measure defined on $\VV$.
\end{prop}
\begin{proof}
	 The implication $ii)\Rightarrow i)$ is obvious, so we turn to the opposite one. We start by showing that for every $v\in\WW$ we have
\begin{equation}\label{todual0}
	\abs{L_\nvect(x)(v)}\le\abs{v}(x)\quad\text{for }\abs{\nvect}\text{-a.e.\ }x\in\XX.
\end{equation}
Let then $v\in\WW$ and let $B$ be a Borel subset of $\XX$. If $ v=0\ \abs{\nvect}$-a.e.\ on $B$, then $\nvect(B)(v)=0$. Otherwise we set $$w_v\defeq\frac{\Vert \abs{v}\chi_B\Vert_{\Lp^\infty(\abs{\nvect})}}{{\Vert \abs{v}\chi_B\Vert_{\Lp^\infty(\abs{\nvect})}}\vee\abs{v}} v$$ and notice that $w_v\in\WW$ by our assumption \eqref{stabilityfrac} and that $|v-w_v|=0$ $|\nvect|$-a.e.\ on $B$. Having assumed $i)$,
this implies
$$\abs{\nvect(B)(v)}=\abs{\nvect(B)(w_v)}\le\abs{\nvect}(B)\Vert w_v\Vert \le\abs{\nvect}(B)\Vert \abs{v}\chi_B\Vert_{\Lp^\infty(\abs{\nvect})}.$$
Therefore, for every $B\subseteq\XX$ Borel we have
$$\abs{\int_B L_\nvect(x)(v)\dd{\abs{\nvect}}}\le\abs{\nvect}(B)\Vert \abs{v}\chi_B\Vert_{\Lp^\infty(\abs{\nvect})}\quad\text{for every }v\in\WW,$$
so that \eqref{todual0} follows.

By the fact that  $\WW$ generates, in the sense of modules $\mathscr M$ and with \eqref{todual0} in mind, we can apply	\cite[Proposition 1.4.8 and Theorem 1.2.24]{Gigli14} to obtain existence and uniqueness of $M_\nvect\in{\mathscr M}^*$ such that \eqref{repr0} holds for every $v\in\WW$ and $\abs{M_\nvect}\le 1\ \abs{\nvect}$-a.e.. Then 
using \eqref{essusp} we show that $\abs{M_\nvect}= 1\ \abs{\nvect}$-a.e.

The fact that formula \eqref{eq:extsl} provides a strongly local extension of $\nvect$ is obvious. To  see that it is the only one, use the implication $(i)\Rightarrow(ii)$ just proved with $\VV$ in place of $\WW$ to find a (unique) corresponding $M_{\hat\nvect}\in\mathscr M^*$ such that \eqref{repr0} holds. Then the uniqueness of both $M_{\hat\nvect}$ and $M_{\nvect}$ forces the equality $M_{\hat\nvect}=M_{\nvect}$ and gives the conclusion.
\end{proof}

%

\section{The theory for RCD spaces}
In this section we treat the case of local vector measures defined on a particular class of $\Cb(\XX)$-normed modules, namely tangent modules on $\RCD$ spaces.
The reason for dealing with $\RCD$ spaces is having at our disposal a fine tangent module (with respect to the capacity). This fine tangent module is useful, in the practice, as many relevant objects turn out to have total variation which is absolutely continuous with respect to the capacity (e.g.\ the distributional derivative of a function of bounded variation).
Even tough it is fairly easy to adapt the theory developed in this section to a more general context, we decided to stick to this particular case for the sake of clarity and to avoid overloading the paper with the axiomatization of the properties regarding the interplay of modules involved, which are by now well known in the $\RCD$ setting.

\subsection{Some useful knowledge} 
With the introduction above in mind, let us briefly introduce $\RCD$ metric measure spaces. An $\RCD(K,N)$ space is an infinitesimally Hilbertian space (\cite{Gigli12}) satisfying a lower Ricci curvature bound and an upper dimension bound (meaningful if $N<\infty$) in synthetic sense according to \cite{Sturm06I,Sturm06II,Lott-Villani09}. General references on this topic are \cite{AmbICM,AmbrosioGigliMondinoRajala12,AmbrosioGigliSavare11,Ambrosio_2014,AmbrosioGigliSavare12,Gigli17,Gigli14,GP19,Villani2017} and we assume the reader to be familiar with this material.

\bigskip

Following \cite{Gigli14,Savare13} (with the additional request of a $\Lp^\infty$ bound on the Laplacian), we define the vector space of test functions on an $\RCD(K,\infty)$ space as
\begin{equation}\notag
	\TestF(\XX)\defeq\{f\in\LIP(\RR)\cap\Lpi\cap D(\Delta): 	\Delta f\in \HSs\cap\Lpi\},
\end{equation}
and the vector space of test vector fields as
\begin{equation}\label{testVdef}
	\TestV(\XX)\defeq\left\{ \sum_{i=1}^n f_i\nabla g_i : f_i\in\Ss\cap\Lpi,g_i\in\TestF(\XX)\right\}.
\end{equation}
Notice that   the original definition of $\TestV(\XX)$ given by the second  named author was slightly different, namely it was, $\left\{\sum_{i=1}^n f_i\nabla g_i: f_i,g_i\in\TestF(\XX)\right\}$.

Our choice is motivated by the desire of having \eqref{stabilityfrac} at our disposal without introducing further space of vectors. In this direction we point out that, rather clearly from the studies in \cite{Gigli14}, for any $v\in \TestV(\XX)$ we have $|v|\in {\rm W}^{1,2}(\XX)\cap\Lp^\infty(\XX)$, thus $\frac1{1\vee|v|}\in {\rm W}^{1,2}(\XX)\cap\Lp^\infty(\XX)$ as well, so that our definition of $\TestV(\XX)$ ensures that such space has the property \eqref{stabilityfrac}.
For what concerns the definition of the spaces  $\WHCSs, \WHHSs$  as closure of the space of test vector fields, having the enlarged space of vector fields makes no difference, as such enlargement is still, trivially, contained in the spaces  $\WHCSs, \WHHSs$   as originally introduced, and therefore such spaces can be equivalently defined taking the closure (with respect to the relevant norm) of the space defined in \eqref{testVdef}.

\bigskip

We assume familiarity with the definition of capacitary modules, quasi-continuous functions and vector fields and related material in \cite{debin2019quasicontinuous}. A summary of the material we will use can be found in \cite[Section 1.3]{bru2019rectifiability}. For the reader's convenience, we write the results that we will need most frequently.
Exploiting Sobolev functions, we define the $2$-capacity (to which we shall simply refer as capacity) of any set $A\subseteq\XX$ as 
$$\capa(A)\defeq\inf\left\{ \Vert f\Vert_{\HSs}^2:f\in\HSs,\ f\ge 1\ \mass\text{-a.e.\ on some neighbourhood of $A$}\right\}. $$
An important object will be the one of fine tangent module, as follows ($\qcr$ stands for `quasi continuous representative').
\begin{thm}[{\cite[Theorem 2.6]{debin2019quasicontinuous}}]\label{tancapa}
	Let $(\XX,\dist,\mass)$ be an $\RCD(K,\infty)$ space.
	Then there exists a unique couple $(\tanXcap,\nablatilde)$, where $\tanXcap$ is a $\Lp^0(\capa)$-normed $\Lp^0(\capa)$-module and $\nablatilde:	\TestF(\XX) \rightarrow\tanXcap$ is a linear operator such that:
	\begin{enumerate}[label=\roman*)]
		\item
		$|{\nablatilde f}|=\qcr(\abs{\nabla f}) \ \capa$-a.e.\ for every $f\in\TestF(\XX)$,
		\item
		the set $\left\{\sum_{n} \chi_{E_n}\nablatilde f_n\right\}$, where $\{f_n\}_n\subseteq\TestF(\XX)$ and $\{E_n\}_n$ is a Borel partition of $\XX$ is dense in $\tanXcap$.
	\end{enumerate}
	Uniqueness is intended up to unique isomorphism, in the following sense: if another couple $(\tanXcap',\nablatilde')$ satisfies the same properties, then there exists a unique module isomorphism $\Phi:\tanXcap\rightarrow\tanXcap'$ such that $\Phi\circ \nablatilde=\nablatilde'$.
	Moreover, $\tanXcap$ is a Hilbert module that we call capacitary tangent module.
\end{thm}

Notice that we can, and will, extend the map $\qcr$ from $\HSs$ to $\Ss\cap\Lpi$ by a locality argument.
We define
\begin{equation}\notag
	\TestVbar(\XX)\defeq \left\{ \sum_{i=1}^n \qcr(f_i) \nablatilde g_i :f_i\in\Ss\cap\Lpi,g_i\in\TestF(\XX) \right\}.
\end{equation}

We define also the vector subspace of quasi-continuous vector fields, $\Cqcvf$, as the closure of $\TestVbar(\XX)$ in $\tanXcap$ and finally
\begin{equation}\label{Cqcvfdef}
	\Cqcvfinf\defeq\left\{v\in\Cqcvf:\abs{v}\text{ is $\capa$-essentially bounded} \right\}.
\end{equation}

Recall now that as $\mass\ll\capa$, we have a natural projection map
\begin{equation}\notag
	\Pr:\Lpc\rightarrow\Lpo \quad \text{defined as}\quad[f]_{\Lpc}\mapsto [f]_{\Lpo}
\end{equation}
where $[f]_{\Lpc}$ (resp.\ $[f]_{\Lpo}$) denotes the $\capa$ (resp.\ $\mass$) equivalence class of $f$. It turns out that $\Pr$, restricted to the set of quasi-continuous functions, is injective.
We have the following projection map $\Prbar$, given by \cite[Proposition 2.9 and Proposition 2.13]{debin2019quasicontinuous}, which plays the role of $\Pr$ on vector fields. 
\begin{prop}\label{prbardef}
	Let $(\XX,\dist,\mass)$ be an $\RCD(K,\infty)$ space. There exists a unique linear continuous map \begin{equation}\notag
		\Prbar :\tanXcap\rightarrow\tanXzero
	\end{equation}
	that satisfies
	\begin{enumerate}[label=\roman*)]
		\item $\Prbar (\nablatilde f)=\nabla f$ for every $f\in\TestF(\XX)$,
		\item $\Prbar (g v)=\Pr(g)\Prbar(v)$ for every $g\in\Lpc$ and $v\in\tanXcap$.
	\end{enumerate}
	Moreover, for every $v\in\tanXcap$,
	\begin{equation}\notag
		\abs{\Prbar(v)}=\Pr(\abs{v})\quad\mass\text{-a.e.}
	\end{equation}
	and $\Prbar$, when restricted to the set of quasi-continuous vector fields, is injective.
\end{prop}
We point out that if $v\in\Cqcvf$, \cite[Proposition 2.12]{debin2019quasicontinuous} shows that $\abs{v}\in\Lpc$ is quasi-continuous, in particular,  $v\in \Cqcvfinf$ if and only if $\Prbar(v)\in\tanXinf$.

In what follows, with a little abuse, we will often write, for $v\in\tanXcap$, $v\in D(\dive)$ if and only if $\Prbar (v)\in D(\dive)$ and, if this is the case, $\dive\, v=\dive(\Prbar(v))$. Similar notation will be used for other operators acting on subspaces of $\tanXzero$. 

The following theorem describes the analogue of  the map $\qcr$ (defined on functions) in the case of vector fields.
\begin{thm}[{\cite[Theorem 2.14 and Proposition 2.13]{debin2019quasicontinuous}}]
	Let $(\XX,\dist,\mass)$ be an $\RCD(K,\infty)$ space. Then there exists a unique map $\qcrbar:\WHCSs\rightarrow\tanXcap$ such that
	\begin{enumerate}[label=\roman*)]
		\item $\qcrbar (v)\in\Cqcvf$ for every $v\in\WHCSs$,
		\item $\Prbar\circ{\qcrbar}(v)=v$ for every $v\in\WHCSs$.
	\end{enumerate}
	Moreover, $\qcrbar$ is linear and satisfies
	\begin{equation}\notag
		\abs{\qcrbar(v)}=\qcr(\abs{v})\quad \capa\text{-a.e.\ for every }v\in\WHCSs,
	\end{equation}
	so that $\qcrbar$ is continuous as map from $\WHCSs$ to $\tanXcap$.
\end{thm}

We will often omit to write the $\qcrbar$ operator for simplicity of notation. This should cause no ambiguity thanks to the fact that 
\begin{equation}\label{qcrfactorizes}
	\qcrbar(g v)=\qcr(g)\qcrbar(v) \quad\text{for every }g\in\HSs\cap\Lpi\text{ and } v\in\WHCSs\cap\tanXinf.
\end{equation}
This can be proved easily as the continuity of the map $\qcr$ implies that $\qcr(g) \qcrbar(v)$ as above is quasi-continuous and the injectivity of the map $\Prbar$ restricted the set of quasi-continuous vector fields yields the conclusion.
Again by locality, we have that \eqref{qcrfactorizes} holds even for $g\in\Ss\cap\Lpi$.

\bigskip

The following theorem, which is {\cite[Section 1.3]{bru2019rectifiability}}, will be crucial in the construction of modules tailored to particular measures.
\begin{thm}
	\label{finemodule}
	Let $(\XX,\dist,\mass)$ be a metric measure space and let $\mu$ be a Borel measure finite on balls such that $\mu\ll\capa$. Let also $\MM$ be a $\Lpc$-normed $\Lpc$-module. Define the natural (continuous) projection 
	\begin{equation}\notag
		\pi_\mu:\Lpc\rightarrow\Lp^0(\mu).
	\end{equation}
	We define an equivalence relation $\sim_\mu$ on $\MM$ as 
	\begin{equation}\notag
		v\sim_\mu w \text{ if and only if } \abs{v-w}=0 \quad \mu\text{-a.e.}
	\end{equation} 
	Define the quotient space $\MM_{\mu}^0\defeq{\MM}/{\sim_\mu}$ with the natural  (continuous) projection
	\begin{equation}\notag
		\pibar_\mu:\MM\rightarrow\MM_{\mu}^0.
	\end{equation}
	Then $\MM_{\mu}^0$ is a $\Lp^0(\mu)$-normed $\Lp^0(\mu)$-module, with the pointwise norm and product induced by the ones of $\MM$: more precisely, for every $v\in\MM$ and $g\in\Lpc$,
	\begin{equation}\label{defnproj}
		\begin{cases}
			\abs{\pibar_\mu(v)}\defeq\pi_\mu(\abs{v}),\\
			\pi_\mu(g)\pibar_\mu(v)\defeq\pibar_\mu( g v).
		\end{cases}
	\end{equation}
	
	If $p\in[1,\infty]$, we set
	\begin{equation}\notag
		\MM^{p}_{\mu}\defeq\left\{ v\in\MM^0_{\mu}:\abs{v}\in\Lp^p(\mu)\right\},
	\end{equation}
	which is a $\Lp^p(\mu)$-normed $\Lp^\infty(\mu)$-module.
	Moreover, if $\MM$ is a Hilbert module, also $\MM_\mu^0$ and $\MM_\mu^2$ are Hilbert modules. 
\end{thm}

In the particular case in which $\MM=\tanXcap$ and $\mu$ is a Borel measure finite on balls such that $\mu\ll\capa$, we set
\begin{equation}\notag
	\tanbvXp{p}{\mu}\defeq(\tanXcap)_\mu^p\quad\text{for }p\in\{0\}\cup[1,\infty].
\end{equation} 
In the case $\mu=\mass$ notice that considering the map
$$\dot{\nabla}:\TestF(\XX)\stackrel{\nablatilde}{\longrightarrow}\tanXcap \stackrel{\pibar_\mass }{\longrightarrow}(\tanXcap)_\mass^0 $$we can show that $(\tanXcap)_\mass^0$ is isomorphic to the usual $\Lp^0$ tangent module via a map that sends $\nabla f$ to $\dot{\nabla} f$ so that we have no ambiguity of notation and, by construction, the map $\pibar_\mass$ coincides with $\Prbar$ defined in Proposition \ref{prbardef}.
We define the traces
\begin{alignat}{5}\notag
	&\tr_\mu:\HSsloc\rightarrow\Lp^0( \mu)&&\quad\text{as}\quad &&&\tr_\mu\defeq\pi_{\mu}\circ\qcr,\\\notag
	&\trbar_\mu:\WHCSs\rightarrow\tanbvXzero{\mu}&&\quad\text{as}\quad&&&\trbar_\mu\defeq \pibar_{\mu}\circ\qcrbar.
\end{alignat}
To simplify the notation, we will often omit to write the trace operators. This should cause no ambiguity because from \eqref{qcrfactorizes} and \eqref{defnproj} it follows that
\begin{equation}\label{qcrfactorizes2}
	\trbar_\mu(g v)=\tr_\mu (g )\trbar_\mu(v) \quad\text{for every }g\in\HSsloc\cap\Lpi\text{ and } v\in\WHCSs\cap\tanXinf.
\end{equation}

We define
\begin{equation}\notag
	\TestV_\mu(\XX)\defeq\trbar_\mu(\TestV(\XX))\subseteq\tanbvXp{\infty}{\mu}
\end{equation}
and the proof of \cite[Lemma 2.7]{bru2019rectifiability} gives what follows.
\begin{lem}\label{densitytestbargen}
	Let $(\XX,\dist,\mass)$ be an $\RCD(K,\infty)$ space and let $\mu$ be a finite Borel measure such that $\mu\ll\capa$. Then $\TestV_\mu(\XX)$ is dense in $\tanbvXp{p}{\mu}$ for every $p\in[1,\infty)$.
\end{lem}

\bigskip

We will also need Cartesian products of normed modules.
Fix $n\in\NN$, $n\ge 1$ and denote by $\Vert\,\cdot\,\Vert_e$ the Euclidean norm of $\RR^n$. 
Given a $\Lpo$-normed $\Lpo$-module $\mathcal{N} $, we can consider its Cartesian product $\mathcal{N}^n$ and endow it with the natural module structure and with the pointwise norm
$$\abs{(v_1,\dots,v_n)}\defeq\Vert (\abs{v_1},\dots,\abs{v_n})\Vert_e$$
which is induced by a scalar product if and only if the one of $\mathcal{N} $ is, and if this is the case, we will still denote the pointwise scalar product on $\mathcal{N}^n$ by $\,\cdot\,$. Similarly, if $\mathcal N$ is an $\Lp^p$-normed module, then  $\mathcal{N}^n$ has a natural structure of $\Lp^p$-normed module as well, where for $v=(v_1,\dots,v_n)\in \mathcal N^n$ we have
\begin{equation}
\label{eq:lpnorm}
\Vert v\Vert\defeq \| \abs{v}\|_{\Lp^p(\mass)}=\big\|\Vert (\abs{v_1},\dots,\abs{v_n})\Vert_e\big\|_{\Lp^p(\mass)}
\end{equation}
It is then clear that  a subspace $\mathcal{N}_1$ of $\mathcal{N}$ is dense if and only if $(\mathcal{N}_1)^n$ is dense in $\mathcal{N}^n$. Similar considerations hold if $\mass$ is replaced by a Borel measure finite on balls and (with the suitable interpretation) in the case of $\Lpc$-normed $\Lpc$-modules or if we alter the integrability exponent. It is clear $\MM$ is a $\Lpc$-normed $\Lpc$-module and $\mu$ is a Borel measure finite on balls such that $\mu\ll\capa$, then also $$(\MM_\mu^p)^n \cong(\MM^n)_\mu^p\quad\text{for }p\in\{0\}\cup[1,\infty].$$
Finally, we adopt the natural notation
$$ \mathrm{L}_\mu^p(T^{ n}\XX)\defeq \tanbvXp{p}{\mu}^{ n}.$$

\subsection{Definitions and results}\label{se:RCD}
Fix now an $\RCD(K,\infty)$ space $(\XX,\dist,\mass)$ and $n\in\NN$, $n\ge 1$.

In this section, we  often consider local vector measures defined on $\TestV(\XX)^n$, which is endowed with the structure inherited from $\tanXinfn$. 
We recall that the {space $\TestV(\XX)$, defined in \eqref{testVdef}}, slightly differs from the one that one usually finds in literature, as discussed right after the definition \eqref{testVdef}, but with this definition it follows that $\TestV(\XX)^n$ is a normed $\Rr$-module, for $\Rr\defeq\LIPb(\XX)$ and we are going to exploit this property throughout.  Also, $\TestV(\XX)^n$ has the property  \eqref{stabilityfrac}, as one may readily check
using \eqref{vndsjosocn} below.
Notice that, according to the conventions discussed at the end of the last section, the norm of $\TestV(\XX)^n$  is given by formula \eqref{eq:lpnorm} (in particular, in general $\|v\|\neq \|(\||v_1|\|_{\Lp^\infty},\ldots,\||v_n|\|_{\Lp^\infty})\|_e$).

We wish to remark the fact that,
\begin{equation}\label{vndsjosocn}
	 \text{if $f_1, \dots, f_n\in\HSs$ and $\varphi\in\LIP(\RR^n;\RR)$ is such that $\varphi(0)=0$, then $\varphi(f_1,\dots,f_n)\in\HSs$}
\end{equation} 
with $$\qcr(\varphi(f_1,\dots,f_n))=\varphi(\qcr(f_1),\dots,\qcr(f_n))\quad\capa\text{-a.e.}$$
This is trivial in the case $f_1,\dots,f_n\in\HSs\cap\LIPb(\XX)$ and the general case is proved by approximation.
In particular, if $v=(v_1,\dots,v_n)\in\TestV(\XX)^n$, then we have the following compatibility relation $$\qcr(\abs{v})=\Vert\qcr(\abs{v_1}),\dots,\qcr(\abs{v_n})\Vert_e=|{\qcrbar(v)}|\quad\capa\text{-a.e.}$$
where we define the map $\qcrbar:\TestV(\XX)^n\rightarrow (\Cqcvfinf)^n$ componentwise.
We will often omit to write the maps $\qcrbar$ and $\qcr$.

If $\mu$ is a Borel measure finite on balls such that $\mu\ll\capa$, we define the trace map $$\trbar:\TestV(\XX)^n\rightarrow\tanbvXpn{0}{\mu}\quad\text{as}\quad\trbar\defeq\pibar_\mu\circ\qcrbar,$$
where $\pibar_\mu$ is given by Theorem \ref{finemodule}. Similarly as for $\qcr$ and $\qcrbar$, we shall often omit to explicitly write $\trbar$ and $\tr$.

\bigskip



We now give the following two crucial definitions:
\begin{defn}[Polar measures]
	Let $(\XX,\dist,\mass)$ be an $\RCD(K,\infty)$ space and let $\nvect$ be a local vector measure defined on $\TestV(\XX)^n$. We say that $\nvect$ is polar (or that $\nvect$ is a polar vector measure) if $\abs{\nvect}\ll\capa$ and for every $A\subseteq\XX$ Borel
\[
	\nvect(A)(v)=0\quad\text{for every }v\in\TestV(\XX)^n\text{ such that }\abs{v}=0\ \abs{\nvect}\text{-a.e.\ on }A.
\]
\end{defn}

Here and after, we endow, naturally, $\tanXcapinfn$ with the $\Lpcinf$ norm of the pointwise norm. Notice that from the trivial identity
\[
|fv|=|f|\,|v|\qquad \capa\text{-a.e.}\qquad \forall v\in\Lp^0(\capa),\ f\in\Cb(\XX)
\]
it follows that  $\tanXcapinfn$ is a normed $\Cb(\XX)$-module.
\begin{defn}[Representable measures]\label{repmeas}
Let $(\XX,\dist,\mass)$ be an $\RCD(K,\infty)$ space and let $\nvect$ be a local vector measure defined on $\tanXcapinfn$. 
We say that $\nvect$ is representable  (or that $\nvect$ is a representable vector measure) if there exists a finite measure $\mu_\nvect\ll\capa$ and $\nu_\nvect\in\tanXcapn$ with $\abs{\nu_\nvect}=1$ $\mu_\nvect$-a.e.\ such that $\nvect=\nu_\nvect\mu_\nvect$, in the sense that
\begin{equation}\label{reprmeas}
	\nvect(A)(v)=\int_A v\,\cdot\,\nu_\nvect \dd{\mu_\nvect}\quad\text{for every }v\in\tanXcapinfn
\end{equation}
for every $A\subseteq\XX$ Borel. 
\end{defn}
An immediate difference between polar and representable vector measures is their domain of definition: the former are defined on $\TestV(\XX)^n$ whereas the latter are defined on  the whole $\tanXcapn$. We shall see in Proposition \ref{equivrepr} that this is basically the only difference between these notions (to this aim we shall exploit the results in Section \ref{normedmodules}). 
%
%
%
%
\begin{rem}\label{mvmrcdrem}
It is easy to show what follows.
\begin{enumerate}[label=\roman*)]
\item		
The representation of a representable vector measure is unique, in the sense that if $\nvect=\nu_\nvect \mu_\nvect=\nu_\nvect'\mu'_\nvect$, then $\mu_\nvect=\mu'_\nvect$ and $\nu_\nvect=\nu_\nvect'$ $\mu_\nvect$-a.e. Lemma \ref{densitytestbargen} shows moreover that if two representable vector measures coincide on $\XX$ on (the trace of) $\TestV(\XX)^n$, then they are equal.
\item
If $\nvect=\nu_\nvect{\mu_\nvect}$ is a representable  vector measure defined on $\tanXcapinfn$,
 $\mvect=\nu_\mvect{\mu_\mvect}$ is a representable  vector measure defined on $\tanXcapinfm$ and also $f\in\Lp^\infty(\mu_\nvect)^{k\times n}$ and $g\in\Lp^\infty(\mu_\mvect)^{k\times m}$, then $f \nvect 
  +g \mvect $ is a representable  vector measure defined on $\tanXcapinfk$. Indeed, we set $G\defeq \mu_\nvect+\mu_\mvect$ and $$\omega\defeq f \nu_\nvect\dv{\mu_\nvect}{G}+g \nu_\mvect\dv{\mu_\mvect}{G},$$ then  $$f \nvect+ g \mvect=\frac{\omega}{\abs{\omega}} \abs{\omega}G.$$

On the other hand, if $f\in\Lpcinf^{k\times n}$ and $\nvect$ is as above,  $$f \nvect (A)((v_1,\dots,v_k))=\nvect(A)(f^T(v_1,\dots,v_k)),$$
where $\cdot^T$ denotes the transpose operator.


\item
In general, we don't know whether a local vector measure whose total variation is absolutely continuous with respect to $\capa$ is polar, unless other hypothesis are satisfied (cf.\ Proposition \ref{polreprmodules}).
\item
We remark that, if $\abs{\nvect}\ll\capa$, then, for every $v\in\TestV(\XX)^n$,
$$\nvect(A)(v)=0\quad\text{for every $A\subseteq\XX$ Borel}\text{ such that }\abs{v}=0\ \abs{\nvect}\text{-a.e.\ on }A$$
if and only if 
$$\nvect(A)(v)=0\quad\text{for every $A\subseteq\XX$ Borel}\text{ such that }\abs{v}=0\ \capa\text{-a.e.\ on }A.$$
Indeed, if $\abs{v}=0\ {\capa}$-a.e.\ on $A$, then $\abs{v}=0\ \abs{\nvect}$-a.e.\ on $A$, so that the first line implies the second. Conversely, assume that $\abs{v}=0\ \abs{\nvect}$-a.e.\ on $A$. Then we can split $A=A_1\cup A_2$, where $\abs{\nvect}(A_1)=0$ and $\abs{v}=0 \ \capa$-a.e.\ on $A_2$ (just fix a Borel representative of $|v|$ and put $A_2:=\{|v|=0\}$) and therefore we conclude.
\item
It may seem not natural to include the request that $\abs{\nvect}$ is absolutely continuous with respect to the capacity in the definition of polar vector measure. However it makes sense as the quasi-continuous representative of a vector field is the finest representative at our disposal.\fr
%

\end{enumerate}
\end{rem}


\bigskip

Here we describe the polar decomposition of  a representable  vector measure. This will be crucial to exploit the main result of Section \ref{normedmodules}, i.e.\ Proposition \ref{polreprmodules}, whose consequence is the link between polar and representable   vector measures (see Proposition \ref{equivrepr}).
\begin{prop}\label{reprtopolar}
	Let $(\XX,\dist,\mass)$ be an $\RCD(K,\infty)$ space and let $\nvect=\nu\mu$ be a representable  vector measure. Then $\abs{\nvect}=\mu$ and $\nvect$ admits the polar decomposition $L_\nvect\abs{\nvect}$, where $$L_\nvect(x)(v)=v\,\cdot\,\nu(x)\quad\text{for }\abs{\nvect}\text{-a.e.\ $x\in\XX$}\text{ for every $v\in\TestV(\XX)^n$}.$$
\end{prop}
\begin{proof}
The fact that $\abs{\nvect}=\mu$ follows immediately from the fact that $\nvect$ is defined on $\tanXcapinfn$.
	The second assertion follows from Proposition \ref{allhaspolarabs}, taking into account the uniqueness of the Radon-Nikodym derivative.
\end{proof}

For the following proposition, we use the fact that representable  vector measures can be seen as polar vector measures: given a representable  vector measure $\nvect=\nu\mu$, we can always define a polar  vector measure ${\sf I}(\nvect)$   by restriction to $\TestV(\XX)^n\subseteq\tanXcapinfn$, namely $${\sf I}(\nvect)(A)(v)\defeq\int_A v\,\cdot\,\nu\dd{\mu},$$
where, as usual, we took the trace of $v$. 
\begin{prop}\label{equivrepr}
	Let $(\XX,\dist,\mass)$ be an $\RCD(K,\infty)$ space and consider the Banach spaces 
	\begin{align}
				{\sf Rep}_n(\XX)&\defeq\left(\left\{\text{representable  vector measures defined on $\tanXcapinfn$}\right\},\abs{\,\cdot\,}(\XX)\right)\notag\\
		{\sf Pol}_n(\XX)&\defeq\left(\left\{\text{polar vector measures defined on $\TestV(\XX)^n$}\right\},\abs{\,\cdot\,}(\XX)\right)\notag.
	\end{align}
	Then the natural inclusion map 
	\begin{align}\notag
		\sf{I}:	{\sf Rep}_n(\XX)\rightarrow 	{\sf Pol}_n(\XX)
	\end{align}
	is a bijective isometry.
\end{prop}

\begin{proof}
	Thanks to Proposition \ref{vectvalbanach}, recalling \eqref{deflimit}, we easily see that ${\sf Pol}_n(\XX)$ is indeed Banach space and the fact that also ${\sf Rep}_n(\XX)$ is a Banach space will follow from the fact that $\sf I$ is an isometry. Notice that $\sf I$ is clearly linear.
\\\textsc{Step 1}.
	We prove that $\sf I$ is surjective.  Take then a polar vector measure $\nvect$ and notice that by restriction it induces a local vector measure, still denoted $\nvect$, on $\TestV(\XX)^n$.

	We are going to  apply Proposition \ref{polreprmodules} with $|\nvect|$ in place of  $\mass$, the (trace of) elements in $\TestV(\XX)^n$ in $\Lp^\infty(|\nvect|)$ in place of $\WW$ (recall that $\TestV(\XX)^n$ is a normed $\LIPb(\XX)$-module) and $\Lp^\infty_{|\nvect|}(T^n\XX)$ in place of $\VV$. The required density comes from Lemma \ref{densitytestbargen}. Taking also into account Riesz theorem for Hilbert modules (\cite[Theorem 1.2.24]{Gigli14})  we thus find $\nu\in {\rm L}_{\abs{\nvect}}^2(T^n\XX)$ such that $\abs{\nu}=1\ \abs{\nvect}$-a.e.\ and 
	 \[
	  L_\nvect(v)=\nu\,\cdot\, v\qquad\abs{\nvect}\text{-a.e.}\qquad\forall v\in \TestV(\XX)^n.
	  \]
It is then clear that  formula \eqref{reprmeas} with $|\nvect|$ and $\nu$ in place of $\mu_\nvect,\nu_\nvect$ defines a  representable  vector measure whose image via  ${\sf I}$ is precisely $\nvect$ (to be more precise, in  Definition \ref{repmeas} we require $\nu$ to be in $\tanXcapn$ and then use its trace in of formula \eqref{reprmeas}: this obviously makes no difference with what we have  done, since, by the definition given in Theorem \ref{finemodule} elements of $ {\rm L}_{\abs{\nvect}}^2(T^n\XX)$ are defined as traces of elements in $\tanXcapn$).
	\\\textsc{Step 2}. We prove that ${\sf I}$ is an isometry. Take then a representable  vector measure $\nu\mu$ and let $\nvect\defeq{\sf I}(\nu\mu)$.
	If $A\subseteq\XX$ is Borel and $v\in\TestV(\XX)^n$, we can compute $$\abs{\nvect(A)(v)}=\abs{\int_A v\,\cdot\,\nu\dd{\mu}}\le \int_A \abs{v}\abs{\nu}\dd{\mu}\le \Vert v\Vert{}\mu(A)$$
and this shows that $\abs{\nvect}\le \mu$.

Conversely, by Lemma \ref{densitytestbargen}, take $\{v_k\}_k\subseteq\TestV(\XX)^n$ such that $v_k\rightarrow \nu$ in $\tanbvXn{\mu}$.
Set $$w_k\defeq\frac{1}{1\vee \abs{v_k}} {v_k}$$ and notice $w_k\in\TestV(\XX)^n$, $\abs{w_k}\le 1\ \mass$-a.e.\ for every $k$ and $w_k\rightarrow \nu$ in $\tanbvXn{\mu}$. We can compute, by dominated convergence, recalling \eqref{qcrfactorizes2} $$
\abs{\nvect}(\XX)\ge \nvect(\XX)(w_k)=\int_\XX w_k\,\cdot\,\nu\dd{\mu}\rightarrow\int_\XX\dd{\mu},$$
so that, $\abs{\nvect}(\XX)\ge \mu(\XX)$. Then, as we have already showed $\abs{\nvect}\le \mu$, we have $\abs{\nvect}=\mu$.
\end{proof}

The following theorem builds upon the theory of quasi-continuous functions to improve the conclusion of Theorem \ref{weakder}, under a mild additional assumption. Namely, it allows us to prove that, under an additional tightness condition (see \eqref{tightplus}), the (unique) local vector measure given by Theorem \ref{weakder} is polar. The additional tightness condition just mentioned turns out to be rather manageable, explecially in practice, when one deals with differential objects. 
\begin{thm}\label{improvetopolar}
Let $(\XX,\dist,\mass)$ be an $\RCD(K,\infty)$ space and let $F\in (\TestV(\XX)^n)'$ be tight. Assume that $F$ satisfies 
\begin{equation}\label{tightplus}
	\begin{split}
		\text{for every sequence }&\text{$\{f_k\}_k\subseteq\HSs$ equibounded in $\Lpi$ with $f_k\rightarrow 0$ in $\HSs$,}\\&\text{it holds that $F(f_k v)\rightarrow 0$ for every $v\in\TestV(\XX)^n$. }
	\end{split}
\end{equation}
Then there exists a unique \textbf{polar} vector measure $\nvect_F$ defined on $\TestV(\XX)^n$ such that
$$
\nvect_F(\XX)(v)=F(v)\quad\text{for every }v\in\TestV(\XX)^n.
$$
	Moreover, it holds that $\abs{\nvect_F}=\mu$, where $\mu$ is the finite Borel measure given by Lemma \ref{bbtightiffmeas}.
\end{thm}
\begin{proof} We call $\nvect$ the unique local vector measure given by Theorem \ref{weakder}. Assume that $F$ satisfies \eqref{tightplus}, we just have to show that $\nvect$ is polar. 
\\\textsc{Step 1}.  We claim that if $\{f_k\}_k$ is as in \eqref{tightplus} and $B\subseteq\XX$ is Borel, then it holds
\begin{equation}\notag
\nvect(B)(f_k v)\rightarrow 0\quad\text{ for every $v\in\TestV(\XX)^n$. }
\end{equation}
Indeed, let $A\subseteq\XX$ be  open and let $K\subseteq A\cap B$ be compact and then take $\psi\in\LIPbs(\XX)$ taking values in $[0,1]$ be identically 1 on a neighbourhood of $K$ and with support contained in $A$. Then $\{\psi f_k\}_k$ still is as in \eqref{tightplus}, and by weak locality we have
\begin{align*}
	\abs{\nvect(B)( f_kv)}&\le\abs{\nvect(A)(f_k\psi v)}+\big(\abs{\nvect}(B\setminus K)+\abs{\nvect}(A\setminus K)\big)\Vert f_k\Vert_{\Lpi}\Vert v\Vert \\ &= |F(f_k\psi v)|+\big(\abs{\nvect}(B\setminus K)+\abs{\nvect}(A\setminus K)\big)\Vert f_k\Vert_{\Lpi}\Vert v\Vert
\end{align*}
for any $k\in\NN$, so the conclusion follows by first letting $k\to\infty$ and then using the arbitrariness of $A,K$ in conjunction with the regularity of $|\nvect|$.
\\\textsc{Step 2}. We claim that $\nvect\ll\capa$. By regularity of $\abs{\nvect}$, we just have to show that if $K$ is a compact set with $\capa(K)=0$, then $\abs{\nvect}(K)=0$. 

By \eqref{conicidence}, we conclude if we show that $\nvect(K)(v)=0$ for any $v\in\TestV(\XX)^n$. 
As $\capa(K)=0$, we can find a sequence $\{f_k\}_k$ as in \eqref{tightplus} such that $f_k(x)= 1$ for every $x$ in a neighbourhood of $K$. Thus by weak locality we have $\nvect(K)(v)=\nvect(K)(f_kv)$ for every $k\in\NN$ and then the conclusion follows from Step 1.
\\\textsc{Step 3} Now we show that  then $\nvect$ is polar. Taking into account item $iv)$ of Remark \ref{mvmrcdrem} and the regularity of $\abs{\nvect}$, it is sufficient to show that if $v\in\TestV(\XX)^n$ and $K$ is a compact set such that $\abs{v}=0\ \capa$-a.e.\ on $K$, then $\nvect(K)(v)=0$. Fix then $v\in\TestV(\XX)^n$, we can assume with no loss of generality that $\Vert v\Vert=1$. Let now $\varepsilon>0$. By the quasi-continuity of $\abs{v}$, we can find an open set $A_\varepsilon$ such that $\abs{v}$ is continuous on $\XX\setminus A_\varepsilon$ and $\capa(A_\varepsilon)<\varepsilon$. We fix now a continuous version of $\abs{v}$ on $\XX\setminus A_\varepsilon$. Also, as $\abs{v}=0\ \capa$-a.e.\ on $K$, we can assume, up to slightly enlarging $A_\varepsilon$, that $\abs{v}(x)=0$ for every $x\in K\setminus A_\varepsilon $ (still $\capa(A_\varepsilon)<\varepsilon$). 
 Then, $\abs{v}< \varepsilon$ on $B_\varepsilon\setminus A_\varepsilon$, where $B_\varepsilon$ is a suitable open subset of $\XX\setminus A_\varepsilon$ containing $K\setminus A_\varepsilon$.
 Let now $f_\varepsilon\in\HSs$ be such that $f_\varepsilon(x)=1$ for every $x$ in $A_\varepsilon$, $f_\varepsilon(x)\in[0,1]$ for every $x\in\XX$ and $\Vert f_\varepsilon\Vert_{\HSs}<\varepsilon$. 
  Now, by construction, $\abs{(1-f_\varepsilon) v}(x)<\varepsilon$ for every  $x\in B_\varepsilon\cup A_\varepsilon$, which is an open set in $\XX$ containing  of $K$.
 We can thus compute 
 $$ \abs{\nvect(K)(v)}\le\abs{  \nvect(K)( f_\varepsilon v)}+\abs{\nvect(K)(( 1-f_\varepsilon) v)}\le \abs{  \nvect(K)( f_\varepsilon v)}+\varepsilon\Vert v\Vert\abs{\nvect}(K) .$$
 Notice now that $ \nvect(K)( f_\varepsilon v)\rightarrow 0$ as $\varepsilon\searrow 0$ by Step 1. The  conclusion follows.
\end{proof}

\begin{cor}\label{corrcd}
	Let $(\XX,\dist,\mass)$ be  an $\RCD(K,\infty)$ space and consider the Banach spaces 
	\begin{align}
		{\sf Pol}_n(\XX)&\defeq\left(\left\{\text{polar vector measures defined on $\TestV(\XX)^n$}\right\},\abs{\,\cdot\,}(\XX)\right),\notag\\ {\sf Tig}_n(\XX)&\defeq\left(\left\{F\in(\TestV(\XX)^n)': \text{$F$ is tight and $F$ satisfies \eqref{tightplus}}\right\},\Vert\,\cdot\,\Vert'\right).\notag
	\end{align}
	Then the map 
\[
		{\sf Pol}_n(\XX)\rightarrow {\sf Tig}_n(\XX)\quad\text{defined as}\quad \nvect\mapsto \nvect(\XX)
\]
	is a bijective isometry.
\end{cor}
\begin{proof}
Taking into account Corollary \ref{Rieszcor} and  Theorem \ref{improvetopolar}, it is enough to show that for every polar  vector measure $\nvect$, $F\defeq\nvect(\XX)$ satisfies \eqref{tightplus}.

	Let then $\nvect$ be a polar  vector measure. Then we can use Proposition \ref{equivrepr} to represent $\nvect$ as $\nu_\nvect\mu_\nvect$ and hence compute, if $\{f_k\}_k$ is as in \eqref{tightplus} and $v\in\TestV(\XX)^n$, 
$$\nvect(\XX)(f_k v)=\int_\XX f_k v\,\cdot\,\nu_\nvect\dd{\mu_\nvect}.
$$
Now, we recall that if $\{f_k\}_k\subseteq\HSs$ is as above, up to taking a (non relabelled) subsequence, \cite[Theorem 1.20, Proposition 1.12 and Proposition 1.17]{debin2019quasicontinuous} show that the quasi-continuous representatives of $f_k$ converge to $0$ $\capa$-a.e. 
Then claim then follows by standard arguments. 
\end{proof}
%

\subsection{An example: improved results for the differential of BV functions}\label{sectBVRCD}
In the previous section we developed the theory to deal with polar/representable  vector measures and to recognize the local vector measures with this particularly nice behaviour. We give an application of this abstract theory: working on $\RCD(K,\infty)$ spaces, we are able to improve the description of the local vector measure giving the distributional differential of a $\BV$ function studied in Section \ref{sectBV}. This amounts in improving weak locality to `strong locality' (i.e.\ being polar) and hence it gives us the framework to state finer calculus rules.


First, we recall the simple \cite[Remark 2.2]{BGBV}, based on the coarea formula, which provides us with the possibility to give a meaning to the integrals  $\int_\XX f\dive\, v\dd\mass$ even though $\dive\,v\notin\Lpi$.
\begin{rem}\label{interpretationint}
	If $f\in\BVv$, $v\in D(\dive)\cap\Lpi$ and $\{n_k\}_k\subseteq(0,\infty)$, $\{m_k\}_k\subseteq(0,\infty)$ are two sequences with $\lim_k n_k=\lim_k m_k=+\infty$, then the limit
	\begin{equation}\label{defint}
		\lim_k  \int_\XX(f\vee -m_k)\wedge n_k\dive\,v\dd{\mass}
	\end{equation}
	exists finite and does not depend on the particular choice of the sequences $\{n_k\}_k$ and $\{m_k\}_k$.  

Therefore, if $f\in\BVv$ and $v\in D(\dive)\cap\Lpi$,  we can write
$$\int_\XX f\dive\, v\dd{\mass}$$
with the convention that it has to be interpreted as the limit  in \eqref{defint}.\fr
\end{rem}

For what follows, see \cite{BGBV} and the references therein.
\begin{defn}
	Let $(\XX,\dist,\mass)$ be a metric measure space and $F\in\BVv^n$.
	We define, for any $A$ open subset of $\XX$,
	\begin{equation}\label{defntvvector}
		\abs{\DIFF F}(A)\defeq \inf \left\{\liminf_k \int_A \Vert (\lip(F_{i,k}))_{i=1,\dots,n}\Vert_e\dd{\mass}\right\}
	\end{equation}
	where the infimum is taken among all sequences $\{F_{i,k}\}_k\subseteq\LIPloc(A)$ such that $F_{i,k}\rightarrow F_i$ in $\Lp^1(A,\mass)$ for every $i=1,\dots,n$. 
\end{defn}

\begin{prop}
	Let $(\XX,\dist,\mass)$ be a metric measure space and $F\in\BVv^n$. Then $ \abs{\DIFF F}(\,\cdot\,)$ as defined in \eqref{defntvvector} is the restriction to open sets of a finite non negative Borel measure that we call total variation of $F$ and still denote with the same symbol.
\end{prop}

In view of the following proposition, recall that the interpretation of the integral in \eqref{intbypartsvet} is given by Remark \ref{interpretationint}.
\begin{prop}\label{reprvett1}
	Let $(\XX,\dist,\mass)$ be an $\RCD(K,\infty)$ space and $F\in\BVv^n$. Then, for every $A$ open subset of $\XX$, it holds that 
	\begin{equation}
		\label{intbypartsvet}
		\abs{\DIFF F}(A)=\sup\left\{\sum_{i=1}^n\int_A F_i \dive\, v_i\dd{\mass}\right\},
	\end{equation}
	where the supremum is taken among all $v=(v_1,\dots,v_n)\in\mathcal{W}_A^n$, where
\[
		\begin{split}
			\mathcal{{W}}_A^n\defeq\Big\{v=(v_1,\dots v_n)\in\TestV(\XX)^n:\abs{v}\le 1\ \mass\text{-a.e.\ }\text{and } \supp \abs{v}\subseteq A\Big\}.
		\end{split}
\]
\end{prop}

In Section \ref{sectBV} we built a local vector measure describing the distributional differential of a function of bounded variation. We improve now the result, in the framework of $\RCD(K,\infty)$ spaces.
Indeed, here we show that we can treat vector valued functions of bounded variation and also that we have a  more powerful description of the local vector measure describing the weak derivative, as it turns out to be representable.
With a slight abuse, we will denote  the distributional differential of $F$ on a $\RCD(K,\infty)$ space by $\DIFF F$, even though the same notation has been used in Section \ref{sectBV} for the distributional differential  on general metric measure spaces. As in this section we will work only  on $\RCD(K,\infty)$ spaces, this should case no confusion. Also, we justify again the notation $\DIFF F$ as we show that the total variation of the local vector measure $\DIFF F$ is (by construction) equal to the total variation of the $\BV$ function $F$.

In view of the following theorem, recall that the interpretation of the integral in \eqref{difffloccaompeq2} is given by Remark \ref{interpretationint}. Recall also \eqref{Cqcvfdef}.

\begin{thm}\label{weakder2}
	Let $(\XX,\dist,\mass)$ be an $\RCD(K,\infty)$ space and let  $F\in\BVv^n$.
	Then there exists a unique representable  vector measure $\DIFF F$ (hence defined on $\tanXcapinfn$) such that it holds 
	\begin{equation}\label{difffloccaompeq2}
		\sum_{i=1}^n\int_\XX F_i\dive\, v_i\dd\mass=-v\,\cdot\,\DIFF F(\XX)\quad\text{for every }v=(v_1,\dots,v_n)\in (\Cqcvfinf\cap D(\dive))^n.
	\end{equation}
\end{thm}

\begin{proof}
	Notice first that we know that such measure, if exists, is unique, being representable (by $\rm i)$ of Remark \ref{mvmrcdrem}). We start with the case $F_i\in\Lpi$ for every $i=1,\dots,n$. 
	
	Define $\mathcal{F}:\TestV(\XX)^n\rightarrow\RR$ as $$\mathcal{F}(v)\defeq-\sum_{i=1}^n\int_\XX F_i\dive\,v_i\dd{\mass}.$$
	Notice now that from Proposition \ref{reprvett1} it follows that $$\sup
	\left\{\FF(v):v\in\TestV(\XX)^n,\ \Vert v\Vert \le 1, \supp v\subseteq A\right\} =\abs{\DIFF F}(A).$$ 
	Now we want to argue as in the proof of Theorem \ref{difffloccaomp}, building upon Theorem \ref{improvetopolar} instead of Theorem \ref{weakder}.  Take then $\{f_k\}_k$ as in \eqref{tightplus}. We compute, if $v\in\TestV(\XX)^n$,
	$$\mathcal F(f_k v)=-\sum_{i=1}^n\int_\XX F_i \dive(f_k v_i)\dd\mass=-\sum_{i=1}^n\int_\XX F_i f_k\dive v_i\dd\mass-\sum_{i=1}^n\int_\XX F_i \nabla f_k\,\cdot\,v_i\dd\mass$$
	and notice that the right hand side converges to $0$ by the assumption $f_k\rightarrow 0$ in $\HSs$ and $v_i\in\TestV(\XX)$. 
	
	We therefore obtain a polar  vector measure  $\DIFF F$ that satisfies \eqref{difffloccaompeq2} for $v\in\TestV(\XX)^n$ and whose total variation coincides with $\abs{\DIFF F}$. Then, by Proposition \ref{equivrepr}, $\DIFF F$ induces a unique representable  vector measure (that we still call $\DIFF F$) defined on $\tanXcapinfn$, which still has total variation $\abs{\DIFF F}$ and still satisfies \eqref{difffloccaompeq2} for $v\in\TestV(\XX)^n$. 
	
	By \cite[Lemma 3.2]{BGBV},  \eqref{difffloccaompeq2} holds for every $v\in (\WHHSs\cap\tanXinf)^n$. Then the very same argument of \cite[Theorem 3.13]{BGBV} shows that \eqref{difffloccaompeq2} holds for every $v\in (\Cqcvfinf\cap D(\dive))^n.$

	In the general case, we can define $F^m\in(\BVv\cap\Lpi)^n$ as $F^m_i\defeq (F_i\vee -m)\wedge m$ and therefore consider the sequence of polar  vector measures $\{\DIFF F^m\}_m$ given by the paragraphs above. By uniqueness we have that $\DIFF F^l-\DIFF F^m=\DIFF (F^l-F^m)$ and by  \eqref{coareaeqdiff},
	$$|\DIFF (F^m-F)|(\XX)\stackrel{\eqref{intbypartsvet}}{\le}\sum_{i=1}^n|{\DIFF (F_i^m-F_i)}|(\XX)\rightarrow 0\quad\text{as }m\rightarrow\infty.$$  We therefore have that
	$\{\DIFF F^m\}_m$ is a Cauchy sequence that, thanks to Proposition \ref{vectvalbanach}, converges to a representable  vector measure whose total variation is $\abs{\DIFF F}$ (see also Proposition \ref{equivrepr}). Also, taking into account \eqref{deflimit} and Remark \ref{interpretationint}, ${\DIFF F}$ still satisfies \eqref{difffloccaompeq2}.
\end{proof}
Notice that $\DIFF F=\nu_F\abs{\DIFF F}$, characterized by \eqref{difffloccaompeq2}, is coherent with the notions developed in \cite{bru2019rectifiability,BGBV}. In particular, if $f=\chi_E$, where $E$ is a set of finite perimeter and finite mass, we obtain the Gauss-Green integration by parts formula stated in \cite{bru2019rectifiability}. This motivates the following definition.
\begin{defn}
	Let $(\XX,\dist,\mass)$ be an $\RCD(K,\infty)$ space and $F\in\BVv^n$. We call the local vector measure $\DIFF F$ given by Theorem \ref{weakder2} the distributional derivative of $F$. 
\end{defn}

\medskip

We state now the Leibniz rule for bounded functions of bounded variation, that is \cite[Proposition 3.35]{BGBV} and we encourage the reader to compare it with Proposition \ref{leibnizcontprop}. This calculus rule has been, in \cite{BGBV}, the building block to prove the chain rule for vector valued functions of bounded variation, recalled in Theorem \ref{volprop} below. We recall that for a $\mass$-measurable function $f:\XX\rightarrow\RR$ it is customary to define
\begin{equation}\label{veeandwedge}
\begin{alignedat}{3}
	f^{\wedge}(x)&\defeq \apliminf_{y\rightarrow x} f(y)&&\defeq\sup&&\left\{t\in\bar{\RR}: \lim_{r\searrow 0} \frac{\mass(B_r(x)\cap\{f<t\})}{\mass(B_r(x))}=0\right\}, \\
	f^{\vee}(x)&\defeq \aplimsup_{y\rightarrow x} f(y)&&\defeq\inf&&\left\{t\in\bar{\RR}:\lim_{r\searrow 0} \frac{\mass(B_r(x)\cap\{f>t\})}{\mass(B_r(x))}=0\right\}
\end{alignedat}
\end{equation}
and finally $$\bar f\defeq\frac{f^\vee+f^\wedge}{2},$$
with the convention that $+\infty-\infty=0$.

\begin{prop}[Leibniz rule]\label{leibnizrcd}
	Let $(\XX,\dist,\mass)$ be an $\RCD(K,N)$ space and let $f,g\in\BVv\cap\Lpi$. Then $f g \in\BVv$ and
	\begin{equation}\notag
		\DIFF (f g) =\bar{f}\DIFF g+\bar{g}\DIFF f.
	\end{equation}
	In particular, $\abs{\DIFF (f g)}\le \abs{\bar{f}}\abs{\DIFF g}+ \abs{\bar{g}}\abs{\DIFF f}$.
\end{prop}

To state Theorem \ref{volprop}, which is a restatement of \cite[Theorem 3.38]{BGBV} in the language of local vector measures, we recall Definition \ref{intmatrix}. We also need the following proposition, extracted from \cite{BGBV}, to define the functions $F^l$ and $F^r$. 
\begin{prop}\label{normalvectval}
	Let $(\XX,\dist,\mass)$ be an $\RCD(K,N)$ space and $F\in\BVv^n$. Then there exists a pair of $\abs{\DIFF F}$-measurable functions $F^l, F^r: \XX\rightarrow\RR^n$  such that for  $\abs{\DIFF F}\text{-a.e.\ }x$ the following holds. Either $F^l(x)=F^r(x)$ and then 
	$$ 
	\lim_{r\searrow 0}\dashint_{B_r(x)} \abs{F- F^r(x)}\dd{\mass}=0,
	$$
	or  there exists a Borel set $E\subseteq\XX$ with
	$$
	\lim_{r\searrow 0} \frac{\mass(E\cap B_r(x))}{\mass(B_r(x))}=\frac{1}{2}
	$$
	such that  
	\begin{equation}\notag
		\lim_{r\searrow 0}\dashint_{B_r(x)\cap E} \abs{F- F^r(x)}\dd{\mass}=\lim_{r\searrow 0}\dashint_{B_r(x)\cap(\XX\setminus E)} |F- F^l(x)|\dd{\mass}=0.
	\end{equation}
	If $\tilde{F}^l, \tilde{F}^r: \XX\rightarrow\RR^n$ is another pair as above, then for $\abs{\DIFF F}\text{-a.e.\ }x$ either $(\tilde{F}^l(x), \tilde{F}^r(x))=({F}^l(x), {F}^r(x))$ or $(\tilde{F}^l(x), \tilde{F}^r(x))=({F}^r(x), {F}^l(x))$.
\end{prop}

We state now the following result about the chain rule in the $\BV$ setting. Notice that, requiring that the space is $\RCD(K,N)$, we can improve considerably what stated in Proposition \ref{propchaincont}: not only we treat vector valued functions, but we also drop the continuity assumption on the $\BV$ function.
\begin{thm}[Chain rule for vector valued functions]\label{volprop}
	Let $(\XX,\dist,\mass)$ be an $\RCD(K,N)$ space and $F\in\BVv^n$. Let $\phi\in C^1(\RR^n;\RR^m)\cap\LIP(\RR^n;\RR^m)$ for some $m\in\NN$, $m\ge 1$ such that $\phi(0)=0$.
	Then $\phi\circ F\in\BVv^m$ and 
\[
	\DIFF(\phi\circ F)=\left(\int_0^1 \nabla\phi(t F^r+(1-t) F^l)\dd{t}\right)\DIFF F,
\]
where $F^l,F^r$ are given by Proposition \ref{normalvectval}.
\end{thm}

\section*{Appendix}
\setcounter{equation}{0}
\renewcommand\theequation{A.\arabic{equation}}

In this appendix we show how using Doob's martingale convergence theorem we can construct Borel representatives of functions in $\Lp^1(\mass)$ that `linearly' depend on the given function. Notice that no Choice, other than Countable Dependent, is needed in the proof, so our construction differs from similar ones based on the concept of von Neumann lifting: the price that we pay for this is that the Borel representatives are defined only on subsets of full measure.

Since we will need to distinguish between functions and representatives, for given $f$ Borel and integrable, we denote by $[f]$ its equivalence class   in the Lebesgue space $\Lpu$.
Also, $\Lpuloc$  denotes the space of measurable functions $f$ such that for every $x\in\XX$, there exists a neighbourhood of $x$, $B_x$, with $f\in \Lp^1(\mass\mres B_x)$.  As before,  we denote by $[f]$ the equivalence class of $f$, for $f\in\Lpuloc$ Borel.

\medskip

We recall that a measure space is a triplet $(\XX,\FF,\mass)$ where $\XX$ is a set, $\FF$ is a $\sigma$-algebra and $\mass$ is a measure defined on $\FF$. We say that the measure space is separable if there exists a countable collection $\{A_n\}_{n\in\NN}\subseteq\FF$ such that for every $B\in\FF$ with $\mass(B)<\infty$ we can find a sequence $\{B_n\}_n\subseteq\{A_n\}_n$ with $\mass(B_n\Delta B)\rightarrow 0$. 

It is easy to see that for a measure space $(\XX,\FF,\mass)$ the following assertions are equivalent:
\begin{itemize}[label=$\bullet$]
	\item $(\XX,\FF,\mass)$ is separable, say $\{A_n\}_n$ is a countable dense subset of $\FF$,
	\item $\Lp^p(\mass)$ is separable for \emph{some} $p\in[1,\infty)$,
	\item $\Lp^p(\mass)$ is separable for \emph{every} $p\in[1,\infty)$.
\end{itemize}
Moreover, if some (hence all) of the item above is satisfied, a countable dense subset of $\Lp^p(\mass)$, for $p\in[1,\infty)$, can be obtained considering the linear span over $\mathbb Q$ of $\{\chi_{A_n}\}_n$, for $\{A_n\}_n$ as above.
\begin{thma}[`Linear' choice of measurable representatives for measure spaces]\label{ultimo}
Let $(\XX,\FF,\mass)$  be a separable measure space with $\mass$ $\sigma$-finite. 
Then there exist two maps\newline $\leb:\Lpu\rightarrow\FF$ and $\FF\mathrm{Rep}:\Lpuloc\rightarrow\left\{\text{$\FF$-measurable real valued maps on $\XX$}\right\}$ such that 
\begin{enumerate}[label=\roman*)]
	\item $\mass(\XX\setminus\leb([f])=0$ for every $[f]\in\Lpuloc$,
	\item for every $[f]\in\Lpu$ and $f'\in[f]$, we have $f'=\FF\mathrm{Rep}([f])\ \mass$-a.e.
	\item for every $[f], [g]\in\Lpu$ and $\alpha,\beta\in\RR$, we have 
	$$\leb([f])\cap \leb([g])\subseteq\leb([\alpha f+\beta g])$$
	and for every $x\in \leb([f])\cap \leb([g])$, it holds 
	$$\alpha\FF\mathrm{Rep}([f])(x)+\beta\FF\mathrm{Rep}([g])(x)=\FF\mathrm{Rep}([\alpha f+\beta g])(x).$$
\end{enumerate}
\end{thma}
\begin{proof}
By a gluing argument, we can clearly assume that $\mass$ is finite. Let $\{A_n\}_n$ denote the countable dense subset of $\FF$.

We take a sequence of finite $\FF$-measurable partitions of $\XX$, $\{{\mathcal{E}}^k\}_{k\in\NN}$, where ${\mathcal{E}}^k=\{{\mathcal{E}}^k_l\}_{l=1,\dots,n(k)}\subseteq\FF$, with the following properties:
	\begin{itemize}[label=\roman*)]
		\item[$a)$] ${\mathcal{E}}^{k+1}$ is a refinement of ${\mathcal{E}}^k$, in the sense that for every $l$, ${\mathcal{E}}^{k+1}_l\subseteq {\mathcal{E}}^k_m$ for some $m=m(l)$,
		\item[$b)$] for every $n$, there exists $k=k(n)$ such that $A_n$ can be written as union of sets in ${\mathcal{E}}^{k}$.
	\end{itemize}
	We build such sequence as follows: first, let $\FF_k$ denote the $\sigma$-algebra generated by $\{A_1,\dots,A_k\}$ and then let 
	$\mathcal E^k$ be the finest partition of $\XX$ whose sets belong to $\FF_k$.
	
	We then define  a sequence of linear maps $\{P_k\}_{k\in\NN}$ $$P_k:\Lpu\rightarrow\left\{\text{$\FF$-measurable real valued maps on $\XX$}\right\}$$ as follows:
	\begin{equation}\label{defPk}
		P_k( [f])(x) \defeq
		\begin{cases}
			\dashint_{{\mathcal{E}}^k_l}f\dd{\mass}\quad&\text{if }x\in {\mathcal{E}}^k_l \text{ and }\mass( {\mathcal{E}}^k_l)>0,\\
			0&\text{otherwise}.
		\end{cases}
	\end{equation}
	Notice that for every $k$, $P_k:\Lpu\rightarrow\Lpu$ is $1$-Lipschitz, in particular $\Vert P_k([f])\Vert_{\Lpu}\le \Vert f \Vert_{\Lpu}$. We can easily check that the discrete stochastic process $\{P_k([f])\}_k$ is a martingale with respect to the filtration $\{\FF_k\}_k$. To this aim we use property $a)$ in the construction of $\{\mathcal E^k\}_k$.
	Therefore, by \cite[Theorem 2.2]{revyor}, $P_k([f])$ converges $\mass$-a.e.\ to a finite limit. We define then $$\leb([f])\defeq \left\{x: \lim_k P_k([f])(x)\text{ exists finite}\right\}$$
	and then the Borel function
	$$ 
	\FF\mathrm{Rep}([f])(x)\defeq
	\begin{cases}
		\lim_k P_k([f])(x)\quad&\text{if }x\in\leb([f]),\\
		0&\text{otherwise}.
	\end{cases}
	$$
	We notice now that property $i)$ of the statement is trivially satisfied while $iii)$ follows from the linearity of the integral. We only need to show property $ii)$ of the statement. First note that by the request $b)$ in the construction of $\{\mathcal{E}^k\}_k$, property $ii)$ holds true if $f$ belongs to the span over $\mathbb Q$ of $\{\chi_{A_n}\}_n$. Indeed, in such case, $P_k([f])=f$, eventually. Notice now that as $\FF\mathrm{Rep}$ is defined as pointwise limit of $P_k([\,\cdot\,])$ and the maps $P_k:\Lpu\rightarrow\Lpu$ are $1$-Lipschitz, it follows from Fatou's lemma (and a slight abuse of notation) that also $\FF\mathrm{Rep}:\Lpu\rightarrow\Lpu$ is $1$-Lipschitz. Then the conclusion follows by density.
\end{proof}

\begin{cora}[`Linear' choice of Borel representatives for Polish spaces]\label{thmapp}
	Let $(\XX,\tau)$ be a Polish space and let $\mass$ be a $\sigma$-finite Borel measure on $\XX$. Then there exist two maps\newline $\leb:\Lpuloc\rightarrow\mathcal{B}(\XX)$ and $\borrep:\Lpuloc\rightarrow\left\{\text{Borel real valued maps on $\XX$}\right\}$ such that 
	\begin{enumerate}[label=\roman*)]
		\item $\mass(\XX\setminus\leb([f])=0$ for every $[f]\in\Lpuloc$,
		\item for every $[f]\in\Lpuloc$ and $f'\in[f]$, we have $f'=\borrep([f])\ \mass$-a.e.
		\item for every $[f], [g]\in\Lpuloc$ and $\alpha,\beta\in\RR$, we have 
		$$\leb([f])\cap \leb([g])\subseteq\leb([\alpha f+\beta g])$$
		and for every $x\in \leb([f])\cap \leb([g])$, it holds 
		$$\alpha\borrep([f])(x)+\beta\borrep([g])(x)=\borrep([\alpha f+\beta g])(x),$$
		\item for every $[f]\in \Lpuloc$, it holds
		\begin{equation}\label{buoncontrollo}
		|\borrep([f])(x)|\le\inf_{r>0} \Vert f\Vert_{\Lp^\infty(\mass\mressmall B_r(x))} 	\qquad\text{for every }x\in \leb([f]).
		\end{equation}
	\end{enumerate}
\end{cora}
\begin{proof}
	Fix a complete and separable distance $\dist$ on $\XX$ inducing the topology $\tau$.
	Notice first that, up to a $\mass$-negligible Borel set, $\XX=\bigcup_n K_n$, where $\{K_n\}_n$ is an increasing sequence of compact sets such that $\mass\mres K_n$ is a finite measure for every $n$. 
	With a gluing argument, we see that it is enough to  prove the theorem for the compact metric measure space $(K_n,\dist,\mass\mres K_n)$.  In particular, $\Lp^1(\mass\mres K_n)=\Lploc^1(\mass\mres K_n)$. Notice also that by basic measure theory we can check that $(K_n,\mathcal B(K_n),\mass\mres K_n)$ is a separable measure space, where, as usual,  $\mathcal B$ denotes the Borel $\sigma$-algebra. Now we apply Theorem \ref{ultimo}.  It remains to show \eqref{buoncontrollo}. To this aim, we have to modify slightly the partitions $\mathcal E^{k}$ used to prove Theorem \ref{ultimo} in order to ensure that 
	$$
	\lim_{k\rightarrow \infty} \sup_l\diam(\mathcal E_l^k)\rightarrow0.$$
	This can be easily done: using the notation of the proof of Theorem \ref{ultimo}, for $k\ge 2$ we just have to redefine $\FF_k$  as the $\sigma$-algebra generated by $\FF_{k-1}$, $\{A_1,\dots,A_k\}$ and a finite covering of $K_n$ of sets with diameter smaller than $k^{-1}$, and we leave $\mathcal F_1$ unchanged.
\end{proof}

\bibliographystyle{alpha}
	\bibliography{Biblio11}
\end{document}